\documentclass[11pt,final]{amsart}

\usepackage[T1]{fontenc}
\usepackage[utf8]{inputenc}

\usepackage[english]{babel}
\usepackage{geometry}
\usepackage{graphicx}

\usepackage{xcolor}
\usepackage{amsmath,amssymb,amsthm,amsfonts,stmaryrd,mathtools,mathrsfs}
\usepackage{bbm}
\usepackage{eurosym}
\usepackage{array,dcolumn,mathdots,multirow}
\usepackage{rotating}
\usepackage{fancyvrb}
\usepackage[activate={true,nocompatibility},spacing,kerning]{microtype}

\graphicspath{%
{./Images/}
}

\definecolor{darkgreen}{rgb}{0,0.5,0}
\definecolor{darkred}{rgb}{0.5,0,0}
\definecolor{darkblue}{rgb}{0.004,0.396,0.741}
\definecolor{gcolor}{rgb}{0.004,0.396,0.741}	
\definecolor{warning}{rgb}{1.0,0.6,0.0}	
\definecolor{gray}{rgb}{0.6,0.6,0.6}

\usepackage[pdftex,
            colorlinks,
            hyperfootnotes=false,
            pdffitwindow=true,
            plainpages=false,
            pdfpagelabels=true,
            pdfpagemode=UseOutlines,
            pdfpagelayout=TwoPageRight,
            pdftitle={Asymptotic expansions relating to longest increasing subsequences},
            pdfauthor={Folkmar Bornemann, TU Munich},
            hyperindex,
            setpagesize=false,
            filecolor=purple,
            urlcolor=darkred,
            citecolor=gcolor,
            linkcolor=gcolor]{hyperref}
\usepackage{remreset}
\usepackage{bm}

\usepackage{caption}
\theoremstyle{plain}
\newtheorem{theorem}{Theorem}[section]
\newtheorem{lemma}{Lemma}[section]
\newtheorem{corollary}{Corollary}[section]
\theoremstyle{definition}
\newtheorem{definition}{Definition}[section]
\theoremstyle{remark}
\newtheorem{remark}{Remark}[section]

\newtheorem{example}{Example}[section]

\input{macros.sty}
\VerbatimFootnotes
\makeindex
\begin{document}

\title[Asymptotic expansions relating to longest increasing subsequences]{Asymptotic expansions relating to the distribution of the length of longest increasing subsequences}
\author{Folkmar Bornemann}
\address{Department of Mathematics, Technical University of Munich, Germany}
\email{bornemann@tum.de}


\begin{abstract}
We study the distribution of the length of longest increasing subsequences in random permutations of $n$ integers as $n$ grows large and establish an asymptotic expansion in powers of $n^{-1/3}$. Whilst the limit law was already shown by Baik, Deift and Johansson to be the GUE Tracy--Widom distribution $F$, we find explicit analytic expressions of the first few finite-size correction terms as linear combinations of higher order derivatives of $F$ with rational polynomial coefficients. Our proof replaces Johansson's de-Poissonization, which is based on monotonicity as a Tauberian condition, by analytic de-Poissonization of Jacquet and Szpankowski, which is based on growth conditions in the complex plane; it is subject to a tameness hypothesis concerning complex zeros of the analytically continued Poissonized length distribution. In a preparatory step an expansion of the hard-to-soft edge transition law of LUE is studied, which is lifted into an expansion of the Poissonized length distribution for large intensities. Finally, expansions of Stirling-type approximations and of the expected value and variance of the length distribution are given.

\end{abstract}
\keywords{random permutations, random matrices, asymptotics, analytic de-Poissonization}
\subjclass[2010]{05A16, 60B20, 30D15, 30E15, 33C10}

\maketitle

\section{Introduction}

The length $L_n(\sigma)$ of longest increasing subsequences\footnote{Defined as the maximum of all $k$ for which there are
$1\leq i_1 < i_2 < \cdots < i_k \leq n$ with $\sigma_{i_1} < \sigma_{i_2} < \cdots < \sigma_{i_k}$.} of permutations $\sigma$ on  $\{1,2,\ldots,n\}$ becomes a discrete random variable when the permutations are drawn randomly with uniform distribution.
This way the problem of enumerating all permutations $\sigma$ that satisfy $L_n(\sigma)\leq l$ gets encoded in the discrete probability distribution $\prob(L_n \leq l)$. The present paper studies an asymptotic expansion of this distribution when $n$ grows large. 
As there are relations to KPZ growth models (directly so for the PNG model with droplet initial condition, see
 \cite{MR2841918,Spohn2000,MR1933446} and \cite[Chap.~10]{MR2641363}), we expect our findings to have a bearing there, too.

\subsection*{Prior work} We start by recalling some fundamental results and notions. More details and references can be found in the outstanding surveys and monographs \cite{MR1694204,MR3468920,MR3468738,MR2334203}.

\subsubsection*{Ulam's problem}
The study of the behavior as $n$ grows large dates back to Ulam \cite{MR0129165} in 1961, who mentioned that Monte-Carlo computations of E. Neighbor would indicate  $\E(L_n)\approx 1.7\sqrt{n}$. Ulam continued by asking: ``Another question of interest would be to find the distribution of the length of the maximum monotone subsequence around this average.''

Refined numerical experiments by Baer and Brock \cite{MR228216} in 1968 suggested that $\E(L_n) \sim 2\sqrt{n}$
might be the precise leading order. In a 1970 lecture, Hammersley \cite{MR0405665} presented a proof, based on subadditive ergodic theory, that the limit $c = \lim_{n\to \infty} \E(L_n)/\sqrt{n}$ exists. Finally, in 1977, Vershik and Kerov \cite{MR0480398} as well as Logan and Shepp \cite{MR1417317} succeeded in proving $c=2$.

\subsubsection*{Poissonization}
A major tool used by Hammersley was a random process that is, basically, equivalent to the following {\em Poissonization} of the random variable $L_n$: by drawing from the different permutation groups independently and by taking $N_r \in \{0,1,2,\ldots\}$ to be a further independent random variable with a Poisson distribution of intensity $r>0$, the combined random variable $L_{N_r}$ is distributed according to
\[
\prob(L_{N_r} \leq l) = e^{-r} \sum_{n=0}^{\infty} \prob(L_n\leq l)\frac{r^n}{n!} =: P(r;l).
\]
The entire function\footnote{Throughout the paper we will use $n$ as an integer $n\geq 0$, $r$ as a corresponding real variable $r> 0$ (intensity) and $z$ as its continuation into the complex plane.} $P(z;l)$ is the {\em Poisson generating function} of the sequence $\prob(L_n \leq l)$ ($n=0,1,\ldots$) and $f(z;l):=e^{z}P(z;l)$ is the corresponding {\em exponential generating function}. As it turns out, it is much simpler to analyze the Poissonized distribution of $L_{N_r}$ as the intensity $r$ grows large than the original distribution of $L_n$ as $n$ grows large. 

There is, however, also a way back from $L_{N_r}$  to $L_n$: namely, the expected value of the Poisson distribution being $\E(N_r) = r$, combined with some level of concentration, suggests
\[
\prob(L_n \leq l) \approx P(n;l)
\]
when $n\to\infty$ while $l$ is kept near the mode of the distribution. Being a Tauberian result, such a {\em de-Poissonization} is subject to additional conditions, which we will discuss in a moment.

Starting in the early 1990s the Poisson generating function $P(z;l)$ (or the exponential one to the same end) has been represented in terms of one of the following interrelated forms:

\smallskip

\begin{itemize}\itemsep=3pt
\item a Toeplitz determinant in terms of modified Bessel functions \cite{Gessel90},
\item Fredholm determinants of various (discrete) integral operators \cite{MR1791137,MR1866169, MR1986402, MR1758751, MR1826414},
\item a unitary group integral \cite{Rains98}.
\end{itemize} 

\smallskip

\noindent A particular case  of those representations plays a central role in our study: namely\footnote{A derivation from the group integral is found in \cite[§2]{MR1986402} and from the Toeplitz determinant in \cite[Eq.~(3.33)]{MR1303076}.}
\begin{equation}\label{eq:poisson}
P(r;l) = E_2^\text{hard}(4r;l),
\end{equation}
where\footnote{Throughout the paper, we will use $l$ as an integer $l\geq 1$ and $\nu$ as a corresponding real variable $\nu>0$, which is used whenever an expression of $l$ generalizes to non-integer arguments.}
 $E_2^{\text{hard}}(s;\nu)$ denotes the probability that, in the hard-edge scaling limit, the scaled smallest eigenvalue  of the Laguerre unitary ensemble (LUE) with real parameter $\nu > 0$ is bounded from below by $s\geq 0$. This probability is known to be given in terms of a Fredholm determinant (see \cite{MR1236195}):
\begin{equation}\label{eq:E2}
E_2^{\text{hard}}(s;\nu) = \det(I - K_\nu^{\text{Bessel}})\big|_{L^2(0,s)},
\end{equation}
where $K_\nu$ denotes the Bessel kernel in $x,y \geq 0$ (for the integral formula see \cite[Eq.~(2.2)]{MR1266485}):
\begin{equation}\label{eq:besselkern}
K_\nu^{\text{Bessel}}(x,y) := \frac{J_\nu(\sqrt{x}) \sqrt{y}J_\nu'(\sqrt{y})-J_\nu(\sqrt{y}) \sqrt{x}J_\nu'(\sqrt{x})}{2(x-y)}
= \frac{1}{4}\int_0^1 J_\nu(\sqrt{\sigma x})J_\nu(\sqrt{\sigma y})\,d\sigma.
\end{equation}
Obviously, the singularities at the diagonal $x=y$ are removable. 

The work of Tracy and Widom \cite{MR1266485} establishes that the Fredholm determinant \eqref{eq:E2} can be expressed in terms of Painlevé III. Recently, based on Okamoto's Hamiltonian $\sigma$-PIII$'$ framework, Forrester and Mays \cite{arxiv.2205.05257} used that connection to compile a table of the exact rational values of $\prob(L_n \leq l)$ for up to $n=700$;\footnote{Previously, by combinatorial means, Baer and Brock \cite{MR228216} had compiled a table for up to $n=36$, supplemented later by Odlyzko and Rains \cite{OdlyzkoTable, MR1771285} with the cases $n=60, 90, 120$. The cases $n=30,60,90$ got printed in \cite{Mehta04}.} whereas in our
work \cite{arxiv.2206.09411}, based on an equivalent representation in terms of a Chazy I equation, we have compiled such a table\footnote{Available for download at \url{https://box-m3.ma.tum.de/f/7c4f8cb22f5d425f8cff/}.} for up to $n=1000$.

In their seminal 1999 work \cite{MR1682248}, by relating the representation of $P(z;l)$ in terms of the Toeplitz determinant to the machinery of Riemann--Hilbert problems and studying the underlying double-scaling limit by the Deift--Zhou method of steepest descent, Baik, Deift and Johansson  answered Ulam's question and proved that, for $t$ being any fixed real number,
\begin{equation}\label{eq:BDJ1}
\lim_{r\to\infty}\prob\left(\frac{L_{N_r}-2\sqrt{r}}{r^{1/6}}\leq t\right) = F(t),
\end{equation}
where $F$ is the GUE Tracy--Widom distribution: that is, the distribution which expresses, among many other limit laws, the probability that in the soft-edge scaling limit of the Gaussian unitary ensemble (GUE) the scaled largest eigenvalue is bounded from above by~$t$. As for the Poissonized length distribution itself, the Tracy--Widom distribution can be represented in terms of a Fredholm
determinant (see \cite{MR1236195}): namely 
\begin{equation}\label{eq:TW2}
F(t) = \det(I - K_0)|_{L^2(t,\infty)}
\end{equation}
where $K_0$ denotes the Airy kernel in $x,y\in\R$ (for the integral formula see \cite[Eq.~(4.5)]{MR1257246}):
\begin{equation}\label{eq:airykern}
K_0(x,y) := \frac{\Ai(x)\Ai'(y)-\Ai'(x)\Ai(y)}{x-y} = \int_0^\infty \Ai(x+ \sigma) \Ai(y+\sigma)\,d \sigma.
\end{equation}
Obviously, also in this case the singularities at the diagonal $x=y$ are removable. Since the limit distribution in \eqref{eq:BDJ1}
is continuous, by a standard Tauberian follow-up \cite[Lemma~2.1]{MR1652247} of the Portmanteau theorem the limit law holds {\em uniformly} in $t$.

In 2003, Borodin and Forrester gave an alternative proof of \eqref{eq:BDJ1} which is based on studying the hard-to-soft edge transition of LUE for $\nu\to\infty$ in form of the limit law \cite[Thm.~1]{MR1986402}
\begin{equation}\label{eq:BF2003}
\lim_{\nu\to\infty} E_2^{\text{hard}}\left(\big(\nu-t (\nu/2)^{1/3}\big)^2;\nu\right) = F(t)
\end{equation}
(see also \cite[§10.8.4]{MR2641363}), which will be the starting point of our study. Still, there are other proofs of~\eqref{eq:BDJ1} based on representations in terms of Fredholm determinants of further (discrete) integral operators; for expositions and references see the monographs \cite{MR3468920,MR3468738}.

\subsubsection*{De-Poissonization} In the literature, the de-Poissonization of the limit law \eqref{eq:BDJ1} has so far been based exclusively on variants  of the following lemma (cf. \cite[Cor.~2.5]{MR3468920}, originally stated as \cite[Lemma~2.5]{MR1618351}), which uses monotonicity as the underlying Tauberian condition.

\begin{lemma}[Johansson's de-Poissonization lemma~\protect{\cite{MR1618351}}]\label{lem:johan} Suppose the sequence $a_n$ of probabilities $0\leq a_n \leq 1$ satisfies the monotonicity condition $a_{n+1} \leq a_n$ for all $n=0,1,2,\ldots$ and denote its Poisson generating function by
\begin{equation}\label{eq:pgf}
P(z) = e^{-z} \sum_{n=0}^\infty a_n \frac{z^n}{n!}.
\end{equation}
Then, for $s\geq 1$ and $n\geq 2\,${\rm:}\footnote{Note the trade-off between sharper error terms $\mp n^{-s}$ and less sharp perturbations of $n$ by $\pm 2\sqrt{s n \log n}$.}
$P\big(n + 2\sqrt{s n \log n} \big) - n^{-s} \leq a_n \leq P \big(n - 2\sqrt{s n \log n} \big) + n^{-s}$.
\end{lemma}

After establishing the Tauberian condition of monotonicity and applying a variant of Lemma~\ref{lem:johan} to \eqref{eq:BDJ1}, Baik, Deift and Johansson \cite[Thm.~1.1]{MR1682248} got
\begin{equation}\label{eq:BDJ2}
\lim_{n\to\infty}\prob\left(\frac{L_{n}-2\sqrt{n}}{n^{1/6}}\leq t\right) = F(t),
\end{equation}
which holds uniformly in $t$ for the same reasons as given above. (The simple calculations based on Lemma~\ref{lem:johan} are given in \cite[p.~239]{MR3468920}; note that the uniformity of the limit law \eqref{eq:BDJ1} is used there without explicitly saying so.) Adding tail estimates to the picture, those authors were also able to lift the limit law to the moments and got, expanding on Ulam's problem, that the expected value satisfies \cite[Thm.~1.2]{MR1682248} 
\begin{equation}\label{eq:BDJmean}
\E(L_n) = 2\sqrt{n} + M_1 n^{1/6} + o(n^{1/6}),\qquad M_1 = \int_{-\infty}^\infty t F'(t)\,dt \approx -1.7711.
\end{equation}

\nopagebreak

\subsubsection*{Expansions} To our knowledge, only for the Poissonized limit law \eqref{eq:BDJ1} a finite-size correction term has been rigorously established prior to the present paper:\footnote{\label{fn:review}Expansions of probability distributions are sometimes called {\em Edgeworth expansions} in reference to the classical one for the central limit theorem. In random matrix theory quite a variety of such expansions, or at least some precise estimates of convergence rates, have been studied: e.g., for the soft-edge scaling limits of the Gauss and Laguerre ensembles \cite{MR2233711,MR2492622, MR2294977, MR3025686} and of the Jacobi ensembles \cite{MR2485010}, for the hard-edge scaling limit of the Laguerre ensembles \cite{MR3513610,MR3513602,MR4019586,MR3455282}, for the bulk scaling limit of the circular ensembles \cite{MR3647807}, and for various joint probability distributions \cite{MR2150191,MR3161478,MR2605065,MR2841918,MR2787973,MR2054175}.} namely, as a by-product along the way of their study of the limiting distribution of maximal crossings and nestings of Poissonized random matchings, Baik and Jenkins \cite[Thm.~1.3]{MR3161478} obtained (using the machinery of Riemann--Hilbert problems and Painlevé representations of the Tracy--Widom distribution), as $r\to\infty$ with $t$ being any fixed real number,
\begin{subequations}\label{eq:BaikJenkins1}
\begin{equation}
\prob\left(\frac{L_{N_r} - 2\sqrt{r}}{r^{1/6}}\leq t\right)  = F\big(t^{(r)}\big) - \frac{1}{10}\Big(F''(t) + \frac{t^2}{6} F'(t)\Big)r^{-1/3} + O(r^{-1/2}),
\end{equation}
where (with $\lfloor\cdot\rfloor$ denoting the Gauss bracket)
\begin{equation}
t^{(r)} := \frac{\lfloor 2\sqrt{r} + t r^{1/6}\rfloor - 2\sqrt{r}}{r^{1/6}}.
\end{equation}
\end{subequations}
However, even if there is enough uniformity in this result and the option to Taylor expand the Poisson generating function $P(r)$ at $n$  with a uniform bound while $l$ is kept near the mode of the distributions (see Sect.~\ref{sect:CDFexpan} for details on this option), the sandwiching in Johansson's de-Poissonization Lemma~\ref{lem:johan} does not allow us to obtain a result better than (cf.~\cite[§9]{MR3161478})
\begin{equation}\label{eq:BaikJenkinsOrig}
\prob\left(\frac{L_n - 2\sqrt{n}}{n^{1/6}}\leq t\right)  = F\big(t^{(n)}\big) + O\big(n^{-1/6} \sqrt{\log n}\,\big).
\end{equation}
In their recent study of finite-size effects, Forrester and Mays \cite[Prop.~1.1]{arxiv.2205.05257} gave a different proof of \eqref{eq:BaikJenkins1} based on the Bessel kernel determinant \eqref{eq:E2}. Moreover, suggested by exact data for $n=700$ and a Monte-Carlo simulation for $n=20\,000$ they were led to conjecture \cite[Conj.~4.2]{arxiv.2205.05257} 
\begin{equation}\label{eq:FM1}
\prob\left(\frac{L_n - 2\sqrt{n}}{n^{1/6}}\leq t\right)  = F\big(t^{(n)}\big) + F^D_1(t)n^{-1/3} + \cdots
\end{equation}
with the approximate graphical form of $F^D_1(t)$ displayed in \cite[Fig.~7]{arxiv.2205.05257}.

The presence of the Gauss bracket in $t^{(r)}$ and $t^{(n)}$, while keeping  $t$ at other places of the expansions \eqref{eq:BaikJenkins1} and \eqref{eq:FM1}, causes undesirable effects in the error terms (see Remark~\ref{rem:gauss} below for a detailed discussion). Therefore, in our work \cite{arxiv.2206.09411} on a Stirling-type formula approximating the distribution $\prob(L_n \leq l)$, we suggested to use the integer $l$ in the continuous expansion terms instead of introducing the continuous variable $t$ into the discrete distribution in the first place, with the latter variant turning the discrete distribution into a piecewise constant function of $t$. By introducing the scaling
\[
t_\nu(r)  := \frac{\nu-2\sqrt{r}}{r^{1/6}} \qquad (r>0)
\]
we were led (based on numerical experiments using the Stirling-type approximation for $n$ getting as large as $10^{10}$),  to conjecture the expansion
\[
\prob(L_n \leq l) = F(t) + F_{1}^D(t)\, n^{-1/3} + F_{2}^D(t)\, n^{-2/3} + O(n^{-1})\,\Big|_{t=t_l(n)},
\]
displaying the graphical form of the functions $F^D_1(t)$, $F^D_2(t)$  in the left panels of \cite[Figs.~4/6]{arxiv.2206.09411}. Moreover, as a note added in proof (see \cite[Eq.~(11)]{arxiv.2206.09411}), we announced  that inserting the Baik--Jenkins expansion~\eqref{eq:BaikJenkins1} into the Stirling-type formula and using its (numerically observed) apparent order $O(n^{-2/3})$ of approximation would yield the functional form of $F_1^D$ to be 
\begin{equation}\label{eq:F1Dconj}
F_1^D(t) = -\frac{1}{10}\left(6 F''(t) + \frac{t^2}{6} F'(t) \right) .
\end{equation}
The quest for a proof, and for a similar expression for $F_2^D$, motivated our present work.
 
\subsection*{The new findings of the paper} 
In the analysis of algorithms in theoretical computer science, or the enumeration of combinatorial structures to the same end, the original enumeration problem is often represented in form of recurrences or functional/differential equations. For instance, this situation arises in a large class of algorithms involving a splitting process, trees, or hashing. Embedding such processes into a Poisson process\footnote{As a heuristic principle in probability and combinatorics, Poissonization was popularized by Aldous' book~\cite{MR969362}.} often leads to more tractable equations, so that sharp tools for a subsequent de-Poissonization were developed in the 1990s; for references and details see  \cite{MR1625392,MR3524836,MR1816272} and Appendix~\ref{sect:jasz}. In particular, if the Poisson generating function $P(z)$  of a sequence of real $a_n >0$, as defined in \eqref{eq:pgf}, is an entire function, an application of the saddle point method to the Cauchy integral
\[
a_n = \frac{n!}{2\pi i} \oint P(z) e^z \frac{dz}{z^{n+1}}
\]
yields, under suitable growth conditions on $P(z)$ as $z\to \infty$ in the complex plane, the Jasz\footnote{Dubbed so in \cite[§VIII.18]{MR2483235} to compliment the seminal work  of Jacquet and Szpankowski \cite{MR1625392}.} expansion
\[
a_n \sim P(n)  + \sum_{j=2}^\infty b_j(n) P^{(j)}(n),
\]
where the polynomial coefficients $b_j(n)$ are the diagonal Poisson--Charlier polynomials (that is, with intensity $r=n$).  
In Appendix~\ref{sect:jasz} we give a heuristic derivation of that expansion and recall, in the detailed form of Thm.~\ref{thm:jasz}, a specific {\em analytic de-Poissonization} result from the comprehensive memoir  \cite{MR1625392} of Jacquet and Szpankowski---a result which applies to a family of Poisson generating functions at once, providing uniform error bounds.

Now, the difficult part of applying Thm.~\ref{thm:jasz} is checking the Tauberian growth conditions in the complex plane, which are required to hold uniformly for the family of Poisson generating functions (recall that, in the case of the longest increasing subsequence problem, $P(z)=P(z;l)$ depends on the integer $l$ near the mode of the length distribution). After observing a striking similarity of those growth conditions with the notion of $H$-admissibility for the corresponding exponential generating function (as introduced by Hayman in his memoir~\cite{Hayman56} on the generalization of Stirling's formula), a closer look at the proof of Hayman's \cite[Thm.~XI]{Hayman56} revealed the following result (see Thm.~\ref{thm:genuszero} for a precise quantitative statement):

\medskip
\begin{quote} {\em The family of all entire functions of genus zero which have, for some $\epsilon>0$, no zeros in the sector $|\!\arg z| \leq \pi/2+ \epsilon$ satisfies a universal bound that implies Tauberian growth conditions  suitable for  analytic 
de-Poissonization.}
\end{quote}
\medskip

On the other hand, in our work \cite[Thm.~2.2]{arxiv.2206.09411} on Stirling-type formulae for the problem of longest increasing subsequences, when proving the $H$-admissibility of the exponential generating functions $f(z;l)=e^z P(z;l)$ (for each $l$), we had established, based on the re\-pre\-sen\-ta\-tion~\cite{Rains98} of $P(z;l)$ as a group integral:

\medskip
\begin{quote} {\em For any integer $l\geq 0$ and any $\delta >0$, the exponential generating function $f(z;l)$ is an entire function of genus zero having at most finitely many zeros in the sector $|\!\arg z| \leq \pi - \delta$, none of them being real.}
\end{quote}
\medskip

Under the reasonable assumption (supported by numerical experiments) that those finitely many complex zeros do not come too close to the real axis and do not grow too fast as $n\to\infty$ while $l$ stays near the mode of the length distribution, the uniformity of the Tauberian growth conditions can be preserved (see Corollary~\ref{cor:nonresonant} for the technical details). We call this assumption the {\em tameness hypothesis}\footnote{Proving it seems to be rather difficult, though---at least we were lacking the methodology to do so.} concerning the zeros of the family of $P(z;l)$. 

Subject to the tameness hypothesis, the main result of the present paper, Thm.~\ref{thm:lengthdistexpan}, gives the asymptotic expansion
\[
\prob(L_n \leq l) = F(t) + \sum_{j=1}^m F_{j}^D(t)\, n^{-j/3}  + O\big(n^{-(m+1)/3}\big)\,\bigg|_{t=t_l(n)},
\]
which is uniformly valid when $n,l \to\infty$ while $t_l(n)$ stays bounded\footnote{This is meant, in fact, when we say that $l$ stays near the mode of the length distribution.} and the $F_j^D$ are certain smooth functions.

Finally, now without any detour via the Stirling-type formula, Thm.~\ref{thm:lengthdistexpan} confirms that the expansion term $F_1^D$ is given by \eqref{eq:F1Dconj}, indeed, and yields the striking formula 
\[
F_2^D(t) = \Big(-\frac{139}{350} + \frac{2t^3}{1575}\Big) F'(t) + \Big(-\frac{43t}{350} + \frac{t^4}{7200}\Big) F''(t) + \frac{t^2}{100} F'''(t) + \frac{9}{50} F^{(4)}(t);
\]
see \eqref{eq:F22} for a display of a similarly structured expression for $F_3^D(t)$.

Put to the extreme, with the help of a CAS such as Mathematica, the methods of the present paper can be used to calculate the concrete functional form of the expansion terms for up to $m=10$ and larger.\footnote{\label{fn:m5}A supplementary Mathematica notebook displaying the results up to $m=10$ comes with the source files at \url{https://arxiv.org/abs/2301.02022}.} 
In all cases inspected we observe that
the expansion terms take the form of a linear combination of higher order derivatives of the limit law (that is, the Tracy--Widom distribution $F$) with certain rational polynomials as coefficients; we conjecture that this is generally true for the problem at hand.

\subsection*{Generalization: involutions, orthogonal and symplectic ensembles}  In our subsequent work \cite{2306.03798} we present a similar structure for the expansion terms relating to longest monotone subsequences in (fixed-point free) random involutions, $F$ then being one of the Tracy--Widom distributions for $\beta=1$ or $\beta=4$. The limit laws were first obtained by Baik and Rains \cite{MR1845180}, using the machinery of Riemann--Hilbert problems, and later reclaimed by Borodin and Forrester \cite{MR1986402} through establishing hard-to-soft edge transition laws for LOE and LSE similar to \eqref{eq:BF2003}. In \cite{2306.03798} we  derive the asymptotic expansions by using determinantal formulae \cite{MR2229797,MR2165698} of the hard and soft edge limits for $\beta=1,4$ while taking advantage of their algebraic interrelations with the $\beta=2$ case studied in the present paper (basically, the expansions of the hard-to-soft edge limit laws for $\beta=1,4$ turn out to correspond to a certain factorization of the $\beta=2$ case). This is then followed by applying analytic de-Poissonization, once again subject to a tameness hypothesis.

\subsection*{Organization of the paper}   

The paper splits into two parts: a first one, where all results and proofs are unconditional, addressing asymptotic expansions of Fredholm determinants, of the hard-to-soft edge transition law and the Poissonized length distribution; and a second one, where we restrict ourselves to assuming the tameness hypothesis when addressing analytic de-Poissonization and its various consequences.

\subsubsection*{Part I: Unconditional Results}

In Sect.~\ref{sect:airyperturb} we start with a careful discussion of expansions of perturbed Airy kernel determinants. We stress the importance of such kernel expansions to be differentiable (i.e., one can differentiate into the error term) to easily lift the error bounds to trace norms. The subtle, but fundamental difficulty of such a lift seems to have been missed, more often than not, in the existing literature on convergence rates and expansions of limit laws in random matrix theory (notable exceptions are, e.g., \cite{MR2294977,MR2485010,MR3025686}).

In the rather lengthy Sect.~\ref{sect:hard-to-soft} we study the asymptotic expansion of the Borodin--Forrester hard-to-soft edge transition law \eqref{eq:BF2003}. It is based on a uniform version of Olver's asymptotic expansion of Bessel functions of large order in the transition region, which we discuss in Appendix~\ref{sect:bessel}. In Sect.~\ref{sect:hard-to-soft} we lay the foundational work for the concrete functional form of all subsequent finite-size correction terms. We reduce the complexity of computing these terms by using a coordinate transform on the level of kernels to simplify the kernel expansion---a coordinate transform which gets subsequently reversed on the level of the distributions.\footnote{In \cite{MR3513610} we applied a similar transformation ``trick'' to the expansion of the hard-edge limit law of LUE.} As yet another application of that technique we simplify in Sect.~\ref{sect:Choup} the finite-size correction terms of Choup \cite{MR2233711} to the soft-edge limit law of GUE and LUE.

In  Sect.~\ref{sect:BaikJenkins} we expand the Poissonized length distribution, thereby generalizing the result~\eqref{eq:BaikJenkinsOrig} of Baik and Jenkins. In Sect.~\ref{sect:BaikJenkins} we also discuss the potentially detrimental effect of using Gauss brackets, as in~\eqref{eq:BaikJenkinsOrig}, alongside with the continuous variable $t$ in the expansion terms. 

\subsubsection*{Part II: Results Based on the Tameness Hypothesis}

In  Sect.~\ref{sect:main} we state and prove the main result of the paper: the expansion of the Baik--Deift--Johansson limit law \eqref{eq:BDJ2} of the length distribution. Here we use the Jasz expansion of analytic de-Poissonization (as detailed in Appendix~\ref{sect:jasz}). The universal bounds for entire functions of genus zero, which are used to prove the Tauberian growth conditions in the complex plane, and their relation to the theory of $H$-admissibility are prepared for in Appendix~\ref{sect:hayman}. Additionally, in Sect.~\ref{sect:main} we discuss the modifications that apply to the discrete density $\prob(L_n=l)$ (that is, to the PDF of the length distribution).

In Sect.~\ref{sect:stirling} we apply our findings to the asymptotic expansion of the Stirling-type formula which we introduced in \cite{arxiv.2206.09411} as an accurate tool for the numerical approximation of the length distribution. Subject to the tameness hypothesis, we prove the observation \cite[Eq.~(8b)]{arxiv.2206.09411} about a leading $O(n^{-2/3})$ error of that formula.

Finally, in Sect.~\ref{sect:Ulam} we study the asymptotic expansion of the expected value and variance. Based on a reasonable hypothesis about some uniformity in the tail bounds, we add several more concrete expansion terms to the Baik--Deift--Johansson solution \eqref{eq:BDJmean} of Ulam's problem.

\part*{Part I: Unconditional Results}

\section{Expansions of perturbed Airy kernel determinants}\label{sect:airyperturb}

In Sect.~\ref{sect:hard-to-soft} we will get, with $m$ being some non-negative integer and $t_0$ some real number, kernel expansions
of the form
\begin{equation}\label{eq:kernelexpan}
K_{(h)}(x,y) = K_0(x,y) + \sum_{j=1}^m h^j K_j(x,y) + h^{m+1} O\big(e^{-(x+y)}\big),
\end{equation}
which are 
\begin{itemize}\itemsep=3pt
\item uniformly valid for $t_0 \leq x,y < c h^{-1}$ as $h\to 0^+$, where $c>0$ is some constant;
\item repeatedly differentiable w.r.t. $x$, $y$ as uniform expansions under the same conditions.
\end{itemize}

\smallskip\noindent
Here, $K_h$ is a family of smooth kernels, $K_0$ denotes the Airy kernel \eqref{eq:airykern} and the $K_j(x,y)$ are finite sums of rank one kernels $u(x)v(y)$, with factors $u(\xi)$, $v(\xi)$ of the functional form 
\begin{equation}\label{eq:airyform}
p(\xi) \Ai(\xi)\quad \text{or}\quad p(\xi) \Ai'(\xi),
\end{equation}
where $p(\xi)$ is a polynomial in $\xi$. Since the existing literature tends to neglect the issue of estimating trace norms in terms of kernel bounds, this section aims at  establishing a relatively easy framework for lifting such an expansion to one of the Fredholm determinant.

\subsection{Bounds on the kernels and induced trace norms}\label{sect:bounds}

Bounds on the kernels $K_j$, and on the trace norms of the induced integral operators, can be deduced from 
the estimates,\footnote{This bound, chosen for convenience but not for optimality, follows from the superexponential decay of the Airy function and its derivative as $\xi\to+\infty$ and the bounds $O\big((-\xi)^{-1/4}\big)$ and $O\big((-\xi)^{1/4}\big)$ as $\xi\to-\infty$, cf. the expansions \eqref{eq:airyexpan} and \cite[Eq.~(9.7.9/10)]{MR2723248}.} with $p$ being any polynomial,
\[
|p(\xi)| \cdot \max\big(|\Ai(\xi)|,|\Ai'(\xi)|\big) \leq a_p e^{-\xi}\quad (\xi\in \R),
\]
where the constant $a_p$ does only depend on $p$ (note that we can take $a_p=1$ when $p(\xi)\equiv1$). This way we get from \eqref{eq:airykern}
\[
|K_0(x,y)| \leq \int_0^\infty |\Ai(x+\sigma) \Ai(y+\sigma)|\,d\sigma \leq e^{-(x+y)} \int_0^\infty e^{-2\sigma}\,d\sigma = \tfrac{1}{2}e^{-(x+y)} \quad (x,y \in \R)
\]
and, for $1\leq j\leq m$, constants $c_j$ such that 
\[
|K_j(x,y)| \leq c_j e^{-(x+y)} \quad (x,y \in \R).
\]

For a given continuous kernel $K(x,y)$ we denote the induced integral operator on $L^2(t,c h^{-1})$ by $\bar {\mathbf K}$ and 
the one on $L^2(t,\infty)$, if defined, by ${\mathbf K}$ (suppressing the dependence on $t$ in both cases). The spaces of trace class and Hilbert--Schmidt operators acting on $L^2(t,s)$ are written as ${\mathcal J}^p(t,s)$ with $p=1$ and $p=2$. By using the orthogonal projection of $L^2(t,\infty)$ onto the subspace $L^2(t,ch^{-1})$ we see that 
\begin{equation}\label{eq:traceclassinclusion}
\|\bar {\mathbf K}\|_{{\mathcal J}^1(t,ch^{-1})} \leq \|{\mathbf K}\|_{{\mathcal J}^1(t,\infty)}.
\end{equation}
The Airy operator ${\mathbf K}_0$ (being by \eqref{eq:airykern} the square of the Hilbert--Schmidt operator ${\mathbf A}_t$ with kernel $\Ai(x+y-t)$)  and the expansion operators ${\mathbf K}_j$ ($j\geq 1$) (being finite rank operators) are trace class on the space $L^2(t,\infty)$. Their trace norms are bounded by
\begin{multline*}
\|{\mathbf K}_0\|_{{\mathcal J}^1(t,\infty)} \leq \|{\mathbf A}_t\|_{{\mathcal J}^2(t,\infty)}^2 = \int_t^\infty \int_t^\infty |\Ai(x+y-t)|^2 \,dx\,dy \\*[2mm]
=  \int_0^\infty \int_0^\infty |\Ai(x+y+t)|^2 \,dx\,dy 
\leq e^{-2t} \int_0^\infty\int_0^\infty e^{-2(x+y)}\,dx\,dy = \tfrac{1}{4} e^{-2t} \quad (t\in\R),
\end{multline*}
and, likewise with some constants $c_j^*$, by
\[
\|{\mathbf K}_j\|_{{\mathcal J}^1(t,\infty)} \leq c_j^* e^{-2t}\qquad (1\leq j\leq m, \; t\in \R).
\] 
Here we have used
\[
\|u \otimes v\|_{{\mathcal J}^1(t,\infty)} \leq \|u\|_{L^2(t,\infty)}\|v\|_{L^2(t,\infty)} \leq \frac{a_p a_q}{2} e^{-2t} \quad (t\in \R)
\]
for factors $u(\xi)$, $v(\xi)$ of the form \eqref{eq:airyform} with some polynomials $p$ and $q$.

On the other hand, there is in general no direct relation between kernel bounds and bounds of the trace norm of induced integral operators 
(see \cite[p.~25]{MR2154153}). Writing the kernel of the error term  in \eqref{eq:kernelexpan}, and its bound, in the form \begin{equation}\label{eq:kernelerrorterm}
h^{m+1} R_{m+1,h}(x,y) = h^{m+1} O\big(e^{-(x+y)}\big)
\end{equation}
does therefore not offer, as it stands, any direct method of lifting the bound to trace norm.\footnote{This subtle technical point has frequently been missed in the literature when lifting kernel expansion to trace class operators. (E.g., the argument given in \cite[p.~12]{MR2233711} for lifting a kernel expansion to
the Edgeworth expansion of the largest eigenvalue distribution of GUE and LUE lacks in that respect. It can be made rigorous when supplemented by the estimates given here; see Thm.~\ref{thm:detexpan}. Another rigorous approach can be found in the work of Johnstone \cite{MR2485010,MR3025686}.)}

By taking, however, the explicitly assumed differentiability of the kernel expansion into account, there is a constant $c_\sharp$ such that
\[
S_{m+1,h}(x,y):=\partial_y R_{m+1,h}(x,y),\quad |S_{m+1,h}(x,y)| \leq c_\sharp e^{-(x+y)}
\]
holds true for all $t_0 \leq x,y < c h^{-1}$ and $0<h\leq h_0$ ($h_0$ chosen sufficiently small). If we denote an indicator function by $\chi$ and choose some $c_* < c$ close to $c$, integration gives
\[
R_{m+1,h}(x,y) = R_{m+1}(x,c_*h^{-1}) -\int_t^{c_* h^{-1}} S_{m+1,h}(x,\sigma)e^{\sigma/2} \cdot e^{-\sigma/2} \chi_{[t,\sigma]}(y)\,d\sigma
\]
if $t_0 \leq x,y \leq c_* h^{-1}$. This
shows that the thus induced integral operator on $L^2(t,c_* h^{-1})$, briefly written as
\[
\bar{\mathbf R}_{m+1,h}^{\flat} + \bar{\mathbf R}_{m+1,h}^{\sharp},
\]
is trace class since it is the sum of a rank-one operator and a product of two Hilbert--Schmidt operators. In fact, for $t_0 \leq t \leq c_* h^{-1}$, the trace norms of those terms are bounded by (denoting the implied constant in \eqref{eq:kernelerrorterm} by $c_\flat$)
\begin{align*}
\| \bar{\mathbf R}_{m+1,h}^{\flat}\|_{{\mathcal J}^1(t,c_* h^{-1})}^2 &\leq \|1\|_{L^2(t,c_*h^{-1})}^2 \cdot \|R_{m+1,h}(\,\cdot\,,c_*h^{-1})\|_{L^2(t,c_*h^{-1})}^2  \\*[1mm] 
&\leq (ch^{-1} - t_0) \cdot c_{\flat}^2 e^{-2c_*h^{-1}} \int_t^\infty e^{-2x}\,dx  = h^{-1}e^{-2c_*h^{-1}} O(e^{-2t}) 
\end{align*}
and
\begin{align*}
\| \bar{\mathbf R}_{m+1,h}^{\sharp}\|_{{\mathcal J}^1(t,c_* h^{-1})}^2 &\leq \left(\int_t^{c_*h^{-1}}\int_t^{c_*h^{-1}} S_{m+1,h}(x,y)^2e^{y}\,dx\,dy\right)\cdot
\left(\int_t^{c_*h^{-1}}\int_t^x e^{-x}\,dx\,dy\right)\\*[1mm]
&\leq c_\sharp^2 \left(\int_t^\infty\int_t^\infty e^{-2x-y}\,dx\,dy\right)\cdot
\left(\int_t^\infty\int_t^x e^{-x}\,dx\,dy\right) = O(e^{-4t}).
\end{align*}
Here, the implied constants are independent of the particular choice of $c_*$. By noting that the 
orthogonal projection of $L^2(t,ch^{-1})$ onto $L^2(t,c_*h^{-1})$ converges, as $c_* \to c$, to the identity in the strong operator topology, we obtain by a continuity theorem\footnote{We use just a simple special case: if $H$ is a separable Hilbert space, $A\in {\mathcal J}^1(H)$  and $P_n: H \to H$ a sequence of orthogonal projections converging to the identity in the strong operator topology, then $\|P_n A P_n - A\|_{{\mathcal J}^1(H)} \to 0$. See also also \cite[p.~28, Example~3]{MR2154153}.} of Grümm \cite[Thm.~1]{MR0327208}
\begin{align*}
\| \bar{\mathbf R}_{m+1,h}\|_{{\mathcal J}^1(t,c h^{-1})} & = \lim_{c_*\to c} \| \bar{\mathbf R}_{m+1,h}\|_{{\mathcal J}^1(t,c_* h^{-1})} \\*[2mm]
&\leq \limsup_{c_*\to c}\| \bar{\mathbf R}_{m+1,h}^{\flat}\|_{{\mathcal J}^1(t,c_* h^{-1})} + \limsup_{c_*\to c}\| \bar{\mathbf R}_{m+1,h}^{\sharp}\|_{{\mathcal J}^1(t,c_* h^{-1})}\\*[2mm] &= h^{-1/2} e^{-c h^{-1}} O(e^{-t}) + O(e^{-2t}).
\end{align*}
We have thus lifted the kernel expansion \eqref{eq:kernelexpan} to an operator expansion in ${\mathcal J}^1(t,ch^{-1})$, 
namely
\begin{subequations}\label{eq:operatorexpan}
\begin{equation}
\bar{\mathbf K}_{(h)} = \bar{\mathbf K}_0 + \sum_{j=1}^m h^j \bar{\mathbf K}_j + h^{m+1} \bar{\mathbf R}_{m+1,h},
\end{equation}
with the bounds (recall \eqref{eq:traceclassinclusion} and observe that we can absorb $h^{m+1/2} e^{-c h^{-1}}$ into $O(e^{-c h^{-1}})$)
\begin{align}
\|{\mathbf K}_j \|_{{\mathcal J}^1(t,\infty)} &= O(e^{-2t}),\\*[3mm]
h^{m+1}\|\bar{\mathbf R}_{m+1,h}\|_{{\mathcal J}^1(t,ch^{-1})} &= h^{m+1} O(e^{-2t}) + e ^{-c h^{-1}} O(e^{-t}),\label{eq:operatorexpanR}
\end{align}
\end{subequations}
uniformly valid for $t_0 \leq t < c h^{-1}$ as $h\to 0^+$.

\subsection{Fredholm determinants}
Given a continuous kernel $K(x,y)$ on the bounded rectangle $t_0 \leq x,y \leq t_1$, 
the Fredholm determinant (cf. \cite[§3.4]{MR2760897})
\begin{equation}\label{eq:FredholmDet}
\det(I - K)|_{L^2(t,s)} := \sum_{m=0}^\infty \frac{(-1)^m}{m!} \int_t^s \cdots\int_t^s \det_{j,k=1}^m K(x_j,x_k) \,dx_1\cdots\, dx_m
\end{equation}
is well-defined for $t_0\leq t < s \leq t_1$. If, as is the case for the kernels $K_j$ (see Sect.~\ref{sect:bounds}), the kernel is also continuous for all $x,y\geq t_0$ and there is a 
weighted uniform bound of the form 
\[
\sup_{x,y \geq t_0} e^{x+y} \,|K(x,y)| \leq M < \infty,
\]
the Fredholm determinant \eqref{eq:FredholmDet} is well-defined even if we choose $s=\infty$ (this can be seen by writing the integrals in terms of the weighted measure $e^{-x}\,dx$; cf. \cite[§3.4]{MR2760897}). 

Now, if the induced integral operator ${\mathbf K}|_{L^2(t,s)}$ on $L^2(t,s)$ is trace class, the Fredholm determinant can be expressed in terms of the operator determinant  (cf. \cite[Chap.~3]{MR2154153}):
\begin{equation}\label{eq:detrel}
\det(I - K)|_{L^2(t,s)} = \det\big({\mathbf I}- {\mathbf K}|_{L^2(t,s)}\big).
\end{equation}
Using the orthogonal decomposition 
\begin{equation}\label{eq:orthogonal}
L^2(t,\infty) = L^2(t,ch^{-1}) \oplus L^2(ch^{-1},\infty)
\end{equation}
and writing the operator expansion \eqref{eq:operatorexpan} in the form
\[
\bar{\mathbf K}_{(h)} = \bar{\mathbf K}_{m,h} + h^{m+1} \bar{\mathbf R}_{m+1,h}, \qquad \bar{\mathbf K}_{m,h}:=\bar{\mathbf K}_0 + \sum_{j=1}^m h^j \bar{\mathbf K}_j, 
\]
we get from the error estimate in \eqref{eq:operatorexpanR}, by the local Lipschitz continuity of operator determinants w.r.t. trace norm (cf. \cite[Thm.~3.4]{MR2154153}), that
\begin{align*}
\det(I - K_h)|_{L^2(t,ch^{-1})} &= \det({\mathbf I}- \bar{\mathbf K}_{(h)}) 
 = \det\big({\mathbf I}- \bar{\mathbf K}_{m,h}\big) + h^{m+1} O(e^{-2t}) +  e ^{-c h^{-1}} O(e^{-t}) \\*[1mm]
&= \det\big({\mathbf I}- {\mathbf K}_{m,h}\big) + h^{m+1} O(e^{-2t}) +  e ^{-c h^{-1}} O(e^{-t}),
\end{align*}
uniformly for $t_0 \leq t < ch^{-1}$ as $h\to 0^+$.
Here, the last estimate comes from a block decomposition of $\bar{\mathbf K}_{(h)}$ according to \eqref{eq:orthogonal} and applying the trace norm bounds of ${\mathbf K}_j$ in Sect.~\ref{sect:bounds} to the boundary case $t=ch^{-1}$, which yields 
\[
\det\big({\mathbf I}- \bar{\mathbf K}_{m,h}\big) = \det\big({\mathbf I}- {\mathbf K}_{m,h}\big) + O(e^{-2ch^{-1}});
\]
note that the error term of order $O(e^{-2ch^{-1}})$ can be absorbed into $e ^{-c h^{-1}} O(e^{-t})$ if $t < c h^{-1}$.

\subsection{Expansions of operator determinants}\label{sect:Plemelj}
Plemelj's formula gives, for trace class perturbations $\mathbf E$ of the identity on $L^2(t,\infty)$ bounded by $\|{\mathbf E}\|_{{\mathcal J}^1(t,\infty)} < 1$, the convergent series expansion (cf. \cite[Eq.~(5.12)]{MR2154153})
\begin{gather}
\det({\mathbf I} - {\mathbf E}) = \exp\left(-\sum_{n=1}^\infty n^{-1} \tr({\mathbf E}^n)\right) \label{eq:Plemelj}\\*[1mm]
= 1- \tr {\mathbf E}
 + \frac{1}{2}\left((\tr {\mathbf E})^2-\tr({\mathbf E}^2)\right) 
+ \frac{1}{6}\left(-(\tr {\mathbf E})^3 + 3 \tr{\mathbf E}\tr {\mathbf E}^2 - 2 \tr {\mathbf E}^3 \right)+ O\Big(\|{\mathbf E}\|_{{\mathcal J}^1(t,\infty)}^4\Big).\notag
\end{gather}
Thus, since ${\mathbf I}-{\mathbf K_0}$ is invertible with a uniformly bounded inverse as $t\geq t_0$,\footnote{\label{foot:inversebound}The Airy kernel $K_0$ induces a symmetric positive definite integral operator ${\mathbf K_0}$ on $L^2(t,\infty)$. Its norm as a bounded operator is thus is given by the spectral radius, which stays below $1$ uniformly as $t\geq t_0$, cf. \cite{MR1257246}:
\[
\|{\mathbf K_0} \| =\rho({\mathbf K_0}) \leq c(t_0)< 1.
\]
By functional calculus we thus get the uniform bound $\|({\mathbf I}-{\mathbf K_0})^{-1}\| = \frac{1}{1-\rho({\mathbf K_0}) } \leq \frac{1}{1-c(t_0)}$.} we have
\[
\det\big({\mathbf I}-{\mathbf K}_{m,h}\big) = \det({\mathbf I}-{\mathbf K_0}) \det\big({\mathbf I}-{\mathbf E}_{m,h} \big), \qquad
{\mathbf E}_{m,h} := \sum_{j=1}^m h^j ({\mathbf I}-{\mathbf K_0})^{-1}{\mathbf K_j},
\]
with the trace norm of ${\mathbf E}_{m,h} $ being bounded as follows (using the results of Sect.~\ref{sect:bounds} and observing that the trace class forms an ideal within the algebra of bounded operators):
\[
\|{\mathbf E}_{m,h} \|_{{\mathcal J}^1(t,\infty)} \leq  \|({\mathbf I}-{\mathbf K_0})^{-1}\|\, \frac{ h}{1-h}\cdot O\big(e^{-2t}\big) = h\cdot O(e^{-2t}), 
\]
uniformly for $t\geq t_0$ as $h\to0^+$. By \eqref{eq:Plemelj} this implies, uniformly under the same conditions,
\[
\det\big({\mathbf I}-{\mathbf E}_{m,h} \big) = 1 + \sum_{j=1}^m{d_j(t) h^j} + h^{m+1}\cdot O(e^{-2t}).
\]
Here,  $d_j(t)$ depends smoothly on $t$ and satisfies the (right) tail bound $d_j(t) = O(e^{-2t})$. If
 we write briefly
\[
{\mathbf E}_j = ({\mathbf I}-{\mathbf K_0})^{-1}{\mathbf K_j}
\]
the first few cases are given by the expressions (because the traces are taken for trace class operators acting on $L^2(t,\infty)$ they depend on $t$)
\begin{subequations}
\begin{align}
d_1(t) &= - \tr {\mathbf E}_1, \qquad d_2(t) = \frac{1}{2}(\tr{\mathbf E_1})^2 - \frac{1}{2}\tr{\mathbf E}_1^2 - \tr{\mathbf E_2},\\*[1mm]
d_3(t) &= -\frac{1}{6}(\tr{\mathbf E_1})^3 +\frac{1}{2}\tr {\mathbf E}_1\tr {\mathbf E}_1^2 -\frac{1}{2}\tr({\mathbf E}_1{\mathbf E}_2 + {\mathbf E}_2 {\mathbf E}_1) \label{eq:d3}\\*[1mm]
&\qquad - \frac{1}{3}\tr {\mathbf E}_1^3 + \tr{\mathbf E}_1 \tr{\mathbf E}_2 - \tr{\mathbf E}_3.\notag
\end{align}
\end{subequations}
Since the ${\mathbf K_j}$ are trace class, those trace expressions can be recast in terms of resolvent kernels and integral traces (cf. \cite[Thm.~(3.9)]{MR2154153}).

Taking the bound $0\leq F(t)\leq 1$ of the Tracy--Widom distribution (being a probability distribution) into account, the results of Sect.~\ref{sect:airyperturb} can be summarized in form of the following:

\begin{theorem}\label{thm:detexpan} Let $K_{(h)}(x,y)$ be a continuous kernel, $K_0$ the Airy kernel \eqref{eq:airykern} and let the $K_j(x,y)$ be finite sums of rank one kernels with factors of the form \eqref{eq:airyform}.
If, for some fixed non-negative integer $m$ and some real number $t_0$, there is a kernel expansions of the form 
\[
K_{(h)}(x,y) = K_0(x,y) + \sum_{j=1}^m h^j K_j(x,y) + h^{m+1} \cdot O\big(e^{-(x+y)}\big),
\]
which, for some constant $c>0$,  holds uniformly in $t_0\leq x,y < ch^{-1}$ as $h\to 0^+$and which can be repeatedly differentiated w.r.t. $x$ and $y$ as uniform expansions, then the Fredholm determinant of $K_{(h)}$ on $(t,ch^{-1})$ satisfies
\begin{equation}\label{eq:detexpan}
\det\big(I-K_{(h)}\big)|_{L^2(t,ch^{-1})} = F(t) +  \sum_{j=1}^m G_j(t) h^j + h^{m+1} O(e^{-2t}) +  e ^{-c h^{-1}} O(e^{-t}),
\end{equation}
uniformly for $t_0\leq t < ch^{-1}$ as $h\to 0^+$. Here, $F$ denotes the Tracy--Widom distribution~\eqref{eq:TW2} and the $G_j(t)$ are smooth functions depending on the kernels $K_1,\ldots,K_j$, satisfying the (right) tail bounds $G_j(t) = O(e^{-2t})$. The first two are
\begin{align*}
G_1(t) &= - F(t) \cdot \left.\tr\big((I-K_0)^{-1}K_1\big)\right|_{L^2(t,\infty)}, \\*[1mm]
G_2(t) &= F(t)\cdot \bigg(\frac{1}{2} \Big(\!\left.\tr\big((I-K_0)^{-1}K_1\big)\right|_{L^2(t,\infty)}\Big)^2\\*[1mm]
 &\qquad- \frac{1}{2}\left.\tr\big(((I-{K_0})^{-1}{K_1})^2\big)\right|_{L^2(t,\infty)}  
 - \left.\tr\big(({I}-{K_0})^{-1}{K_2}\big)\right|_{L^2(t,\infty)}\bigg),
\end{align*}
where $(I-K_0)^{-1}$ is understood as a resolvent kernel and the traces as integrals. The determinantal expansion \eqref{eq:detexpan} can repeatedly be differentiated w.r.t. $t$, preserving uniformity.
\end{theorem}

\begin{figure}[tbp]
\includegraphics[width=0.325\textwidth]{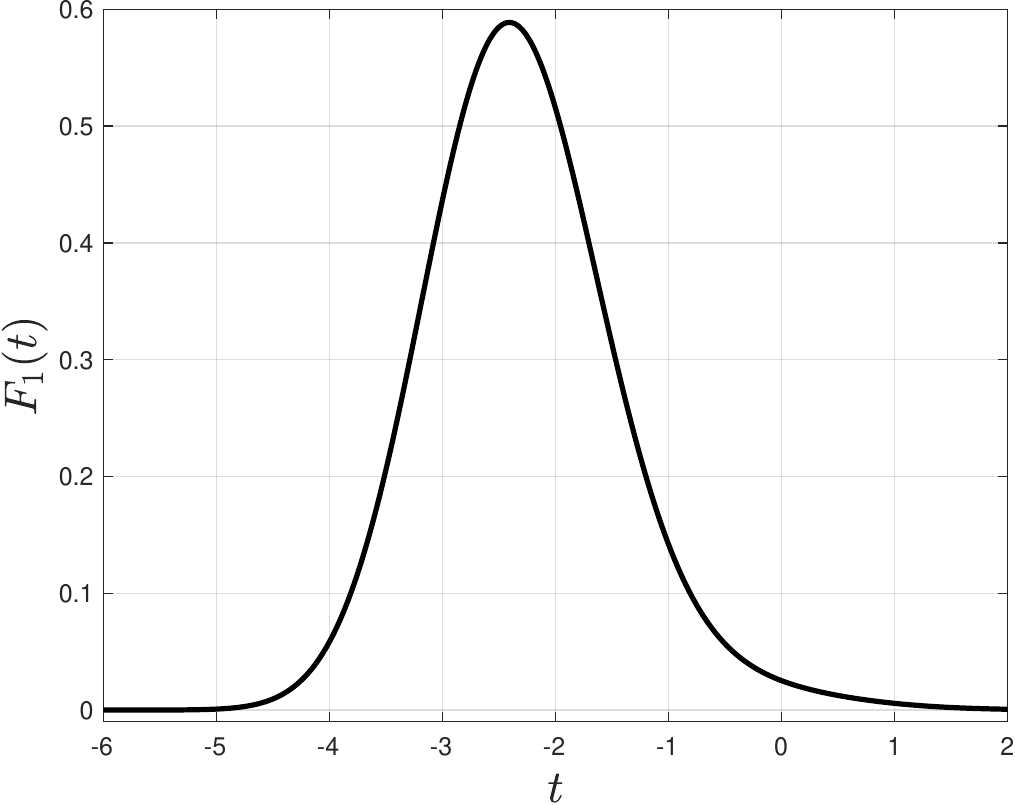}\hfil\,
\includegraphics[width=0.325\textwidth]{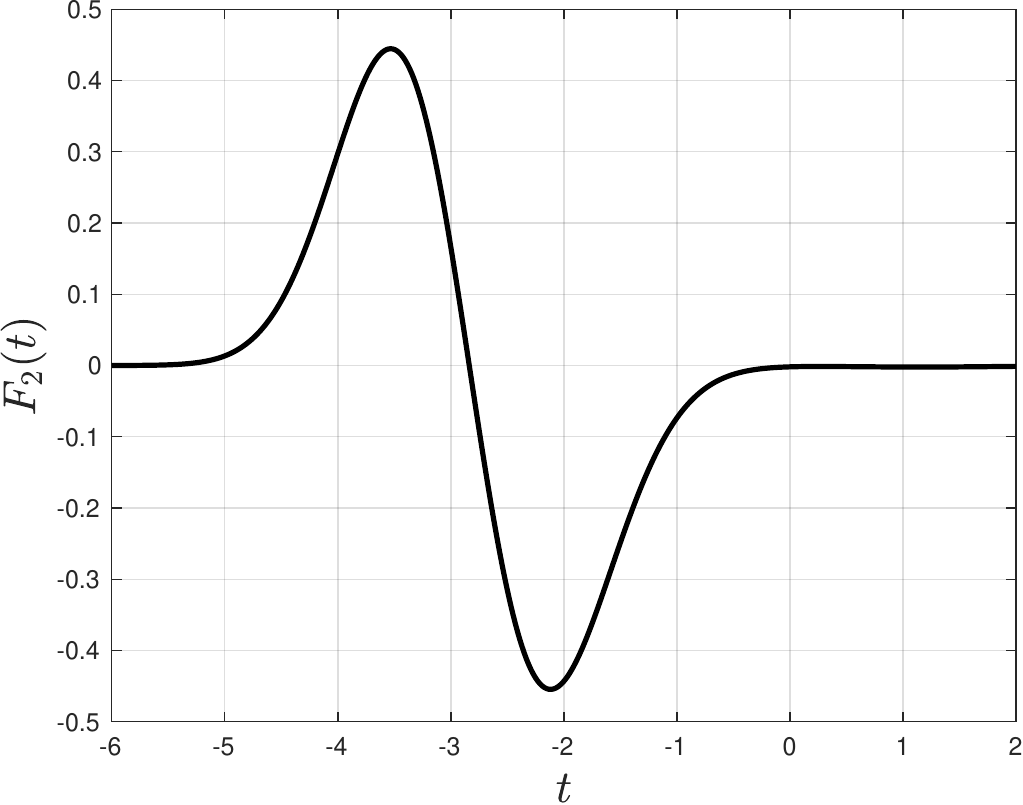}\hfil\,
\includegraphics[width=0.325\textwidth]{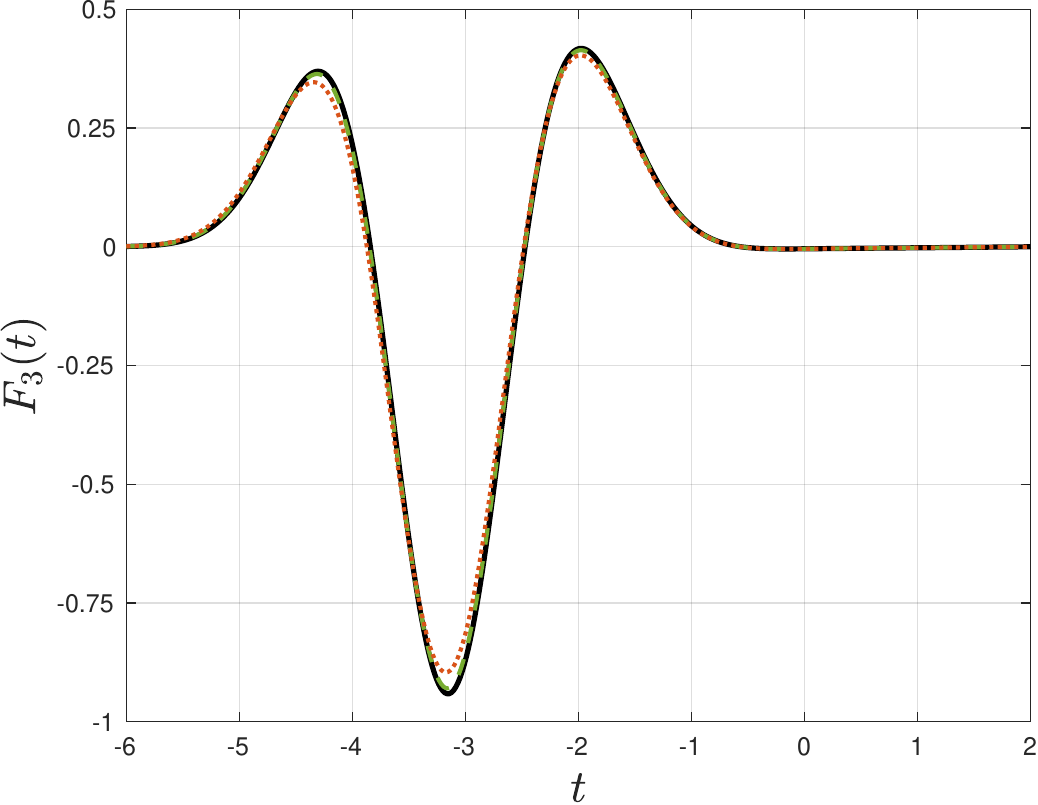}
\caption{{\footnotesize Plots of $F_{1}(t)$ (left panel) and $F_{2}(t)$ (middle panel) as in (\ref{eq:F22}a/b).
The right panel shows $F_{3}(t)$ as in (\ref{eq:F22}c) (black solid line) with the approximations \eqref{eq:F23} for $\nu=100$ (red dotted line) and $\nu=800$ (green dashed line): the close agreement validates the functional forms displayed in \eqref{eq:F22}. Details about the numerical method can be found in \cite{MR2895091, MR2600548,arxiv.2206.09411,MR3647807}.}}
\label{fig:hard2soft}
\end{figure}

\section{Expansion of the Hard-to-Soft Edge Transition}\label{sect:hard-to-soft}

In this section we prove an expansion for the hard-to-soft edge transition limit \eqref{eq:BF2003}. To avoid notational clutter, we use the quantity
\begin{equation}\label{eq:hnu}
h_\nu := 2^{-1/3} \nu^{-2/3}
\end{equation}
and study expansions in powers of $h_\nu$ as $h_\nu \to 0^+$. The transform $s=\phi_\nu(t)$ used in the transition limit can briefly be written as
\begin{equation}\label{eq:phitrans}
\phi_\nu(t) = \nu^2(1-h_\nu t)^2.
\end{equation}

\begin{theorem}\label{thm:hard2soft} 
There holds the hard-to-soft edge transition expansion
\begin{equation}\label{eq:hard2soft}
E_2^{\text{\em hard}}(\phi_\nu(t);\nu) = F(t) + \sum_{j=1}^m F_{j}(t) h_\nu^j + h_\nu^{m+1}\cdot O(e^{-3t/2}),
\end{equation}
which is uniformly valid when $t_0\leq t < h_\nu^{-1}$ as $h_\nu\to0^+$, $m$ being any fixed non-negative integer and $t_0$ any fixed real number. Preserving uniformity, the expansion can be repeatedly differentiated w.r.t. the variable $t$.
Here the $F_{j}$ are certain smooth functions starting with
\begin{subequations}\label{eq:F22}
\begin{align}
F_{1}(t) &= \frac{3t^2}{10} F'(t) - \frac{1}{5} F''(t),\label{eq:F21}\\*[2mm]
F_{2}(t) &=\Big(\frac{2}{175} + \frac{32t^3}{175}\Big) F'(t) + \Big(-\frac{16t}{175} + \frac{9t^4}{200}\Big) F''(t) - \frac{3t^2}{50} F'''(t) + \frac{1}{50} F^{(4)}(t),\\*[2mm]
F_3(t) &= 
\Big(\frac{64 t}{7875} + \frac{1037 t^4}{7875}\Big)  F'(t)
+\Big( -\frac{9t^2}{175}  + \frac{48t^5}{875} \Big) F''(t) \\*[1mm]
&\qquad +\Big(-\frac{122}{7875} -\frac{8 t^3}{125} + \frac{9 t^6}{2000}\Big) F'''(t)
+\Big(\frac{16t}{875} -\frac{9 t^4}{1000} \Big)F^{(4)}(t)\notag\\*[1mm]
&\qquad +\frac{3 t^2 }{500}F^{(5)}(t)
-\frac{1}{750} F^{(6)}(t).\notag
\end{align}
\end{subequations}
\end{theorem}
It is rewarding to validate intriguing formulae such as (\ref{eq:F22}a--c) by numerical methods: Fig.~\ref{fig:hard2soft} plots the functions $F_{1}(t)$, $F_{2}(t)$, $F_3(t)$ next to the approximation
\begin{equation}\label{eq:F23}
F_{3}(t) \approx h_\nu^{-3} \cdot \big( E_2^{\text{hard}}(\phi_\nu(t);\nu) - F(t) - F_{1}(t)h_\nu - F_{2}(t) h_\nu^2\big) 
\end{equation}
for $\nu=100$ and $\nu=800$: the close matching with $F_3(t)$ as displayed by the latter is a very strong testament of the correctness of (\ref{eq:F22}a--c) (in fact, some slips in preliminary calculations have been caught looking at plots which exhibited mismatches).

The proof of Thm.~\ref{thm:hard2soft} is split into several steps and will be concluded in Sects.~\ref{sect:generalform} and \ref{sect:functionalform}.

\subsection{Kernel expansions}

We start with an auxiliary result.

\begin{lemma}\label{lem:Phi}
Define for $h>0$ and $x,y<h^{-1}$ the function
\begin{equation}\label{eq:Phi}
\Phi(x,y;h):=\frac{\sqrt{(1- hx)(1-hy)}} {1-h(x+y)/2}.
\end{equation}
This function $\Phi$ satisfies the bound
\begin{equation}\label{eq:PhiBound}
0 < \Phi(x,y;h) \leq 1
\end{equation}
and has the convergent power series expansion
\begin{subequations}\label{eq:PhiExpan}
\begin{equation}
\Phi(x,y;h) = 1 - (x-y)^2 \sum_{k=2}^\infty r_k(x,y) h^k
\end{equation}
where the $r_k(x,y)$ are certain homogeneous symmetric rational\/\footnote{Throughout the paper the term ``rational polynomial'' is used for polynomials with rational coefficients.} polynomials of degree $k-2$, the first few of them being
\begin{equation}
r_2(x,y) = \frac{1}{8},\quad r_3(x,y) = \frac{1}{8}(x+y), \quad r_4(x,y) = \frac{1}{128}\big(13(x^2+y^2)+22xy\big).
\end{equation}
\end{subequations}
The series converges uniformly for $x,y<(1-\delta)h^{-1}$, $\delta$ being any fixed real positive number.
\end{lemma}
\begin{proof}
The bound \eqref{eq:PhiBound} is the inequality of arithmetic and geometric means for the two positive real quantities $1-hx$ and $1-hy$. 
By using
\[
\lim_{y\to x} \frac{1}{(x-y)^2}\Big(\Phi(x,y;h) -1\Big) = -\frac{1}{8}\left(\frac{h}{1-hx}\right)^2,
\]
the analyticity of $\Phi(x,y;h)$ w.r.t. $h$, and the scaling law
\[
\Phi(\lambda^{-1}x,\lambda^{-1}y;\lambda h) = \Phi(x,y;h) \quad (\lambda>0)
\]
we deduce the claims about the form and uniformity of the power series expansion \eqref{eq:PhiExpan}.
\end{proof}

Because of the representation \eqref{eq:E2} of $E_2^\text{hard}(s;\nu)$ in terms of a Fredholm determinant of the Bessel kernel \eqref{eq:besselkern}, we have to expand the induced transformation of that kernel. 

\begin{lemma}\label{lem:besselkernexpan} The change of variables $s=\phi_\nu(t)$, mapping $t< h_\nu^{-1}$ monotonically decreasing to $s> 0$, induces the symmetrically transformed Bessel kernel  
\begin{subequations}\label{eq:besselkernexpan}
\begin{align}
\hat K_\nu^{\text{\em Bessel}}(x,y) &:=\sqrt{\phi_\nu'(x)\phi_\nu'(y)} \, K_\nu^{\text{\em Bessel}}(\phi_\nu(x),\phi_\nu(y)). \\
\intertext{There holds the kernel expansion}
\hat K_\nu^{\text{\em Bessel}}(x,y) &= K_0(x,y) + \sum_{j=1}^m K_j(x,y) h_\nu^j + h_\nu^{m+1}\cdot O\big(e^{-(x+y)}\big),
\end{align}
\end{subequations}
which is uniformly valid when $t_0 \leq x,y < h_\nu^{-1}$ as $h_\nu\to 0^+$, $m$ being any fixed non-negative integer and $t_0$ any fixed real number. Here the $K_j$ are certain finite rank kernels of the form
\[
K_j(x,y) = \sum_{\kappa,\lambda\in\{0,1\}}  p_{j,\kappa\lambda}(x,y) \Ai^{(\kappa)}(x)\Ai^{(\lambda)}(y)
\]
where $p_{j,\kappa\lambda}(x,y)$ are rational polynomials{\em ;} the first two kernels are
\begin{subequations}\label{eq:K2}
\begin{align}
K_1(x,y) &= \frac{1}{10} \Big( -3\big(x^2+xy+y^2\big)\Ai(x)\Ai(y) \\*[1mm]
&\qquad  + 2 \big(\!\Ai(x)\Ai'(y) + \Ai'(x)\Ai(y)\big)+ 3(x+y)\Ai'(x)\Ai'(y) \Big),\label{eq:K1}\notag\\
K_2(x,y) &= \frac{1}{1400}\Big(\big(-235 \left(x^3+y^3\right)-319 x y
   (x+y)+56\big) \Ai(x)\Ai(y) \\*[1mm]
&\qquad   +\big(63(x^4+x^3 y-x^2 y^2-x y^3-y^4)-55 x+239 y\big) \Ai(x)\Ai'(y) \notag\\*[1mm]
&\qquad\quad   +\big(63(-x^4-x^3 y-x^2 y^2+x y^3+y^4)+239 x-55 y\big) \Ai'(x)\Ai(y) \notag\\*[1mm]
 &\qquad\qquad     +\big(340(x^2+y^2)+256 x y\big) \Ai'(x)\Ai'(y)\Big). \notag
\end{align}
\end{subequations}
Preserving uniformity, the kernel expansion \eqref{eq:besselkernexpan} can repeatedly be differentiated w.r.t. $x$, $y$.
\end{lemma}
\begin{proof} We have to prove, analytically, the claim about the domain of uniformity of the error of the expansion and, algebraically, the finite-rank structure of the expansion kernels $K_j$. In our original proof,\footnote{Which has the merit of being comparatively short and suggesting the nonlinear transform used in Lemma~\ref{lem:Khexpan}.} 
presented below, we use Olver's expansion of Bessel functions of large order in the transition region (see Appendix~\ref{sect:bessel}) and the finite-rank structure is obtained by explicitly inspecting (with Mathematica) a certain algebraic condition (see \eqref{eq:divisibility}) for the first instances $j=1,\ldots,m_*$---we choose to stop at $m_*=100$. However, using the machinery of Riemann--Hilbert problem, this restriction was recently removed by Yao and Zhang \cite{arXiv:2309.06733} (their proof extends over 18 pages); we will comment on their work at the end. 

\medskip

{\em The original proof, requiring $m\leq m_*$.} By using $\Phi(x,y;h)$ as defined in \eqref{eq:Phi} and writing
\begin{gather*}
\phi_\nu(t) = \omega_\nu(t)^2,\quad \omega_\nu(t)=\nu(1-h_\nu t),\\*[2mm]
 a_\nu(x,y)= (1-h_\nu y) \cdot \frac{1}{\sqrt{2h_\nu}}J_\nu\big(\omega_\nu(x)\big)\cdot \frac{d}{dy} \frac{1}{\sqrt{2h_\nu}}J_\nu\big(\omega_\nu(y)\big)
\end{gather*}
we can factor the transformed Bessel kernel in the simple form
\[
\hat K_\nu^{\text{Bessel}}(x,y) = \Phi(x,y;h_\nu) \cdot \frac{a_\nu(x,y)-a_\nu(y,x)}{x-y},
\] 
noting, by symmetry, the removability of the singularities at $x=y$ of the second factor. 

First, if $x$ or $y$ is between $\tfrac34 \cdot h_\nu^{-1}$ and $h_\nu^{-1}$, using the bound $0<\Phi\leq 1$ (see \eqref{eq:PhiBound}) one can argue as in the proof of Lemma~\ref{lem:Olver1952}: since at least one of the Bessel factors is of the form $J_\nu^{(\kappa)}(z)$ with $0\leq z \leq 1/4$, which plainly falls into the superexponentially decaying region as $\nu\to\infty$, and since at least one of the Airy factors of each term of the expansion is also superexponentially decaying as $\nu\to\infty$, the transformed Bessel kernel and the expansion terms in \eqref{eq:besselkernexpan} get completely absorbed into the error term (bounding the other factors as in Sect.~\ref{sect:bounds})
\[
 h_\nu^{m+1}\cdot O\big(e^{-(x+y)}\big).
\]
Here, the removable singularities at $x=y$ are dealt with by using the differentiability of the corresponding bounds (or by extending to the complex domain and using Cauchy's integral formula as in the proof of \cite[Prop.~8]{MR1986402}).

Therefore, we may suppose from now on that $t_0 \leq x,y \leq \tfrac34 \cdot h_\nu^{-1}$. By Lemma~\ref{lem:Phi}, in this range of $x$ and $y$,  the power series expansion
\begin{equation}\label{eq:lemPhiExpan}
\Phi(x,y;h_\nu) = 1 - (x-y)^2 \sum_{k=2}^\infty  r_k(x,y) h_\nu^k
\end{equation}
converges uniformly. Here, the $r_k(x,y)$ are certain homogeneous symmetric rational polynomials of degree $k-2$; the first of them being $r_2(x,y)=1/8$.

Next, we rewrite the uniform version of the large order expansion of Bessel functions in the transition region, as given in Lemma~\ref{lem:Olver1952}, in the form
\begin{equation}\label{eq:Olver_short}
\frac{1}{\sqrt{2h_\nu}} J_\nu\big(\omega_\nu(t)\big) =\big(1 + h_\nu p_m(t;h_\nu)\big) \Ai(t) +  h_\nu q_m(t;h_\nu) \Ai'(t) + h_\nu^{m+1} O(e^{-t}),
\end{equation} 
where the estimate of the remainder is
uniform for $t_0 \leq t < h_\nu^{-1}$ as $h_\nu\to 0^+$ and
\[
p_m(t;h) = 2^{1/3}\sum_{k=0}^{m-1}  A_{k+1}(-t/2^{1/3}) \,(2^{1/3}h)^{k}, \quad q_m(t;h) = 2^{2/3}\sum_{k=0}^{m-1} B_{k+1}(-t/2^{1/3}) \,(2^{1/3}h)^{k},
\]
with the polynomials $A_k(\tau)$ and $B_k(\tau)$ from \eqref{eq:Olver1952}. It follows from Remark~\ref{rem:PQcompute} that
$p_m(t;h)$ and $q_m(t;h)$ are rational polynomials in $t$ and $h$, starting with
\[
p_2(t;h) = \frac{2}{10} t + \frac{h}{1400}(63t^5+120t^2),\quad
q_2(t;h) = \frac{3}{10} t^2 + \frac{h}{1400}(340t^3+40).
\]
Also given in Lemma~\ref{lem:Olver1952}, under the same conditions,  the expansion~\eqref{eq:Olver_short} can be repeatedly differentiated w.r.t~$t$ while preserving uniformity.
From this we obtain, using the Airy differential equation $\Ai''(\xi) = \xi\Ai(\xi)$, that uniformly  (given the range $x$ and $y$)\footnote{Because of the superexponential decay \eqref{eq:airyexpan} of the Airy function $\Ai(t)$ and its derivative $\Ai'(t)$ as $t\to+\infty$, cross terms with the remainder are uniformly estimated in the form
\[
\text{polynomial$(x)$} \cdot \Ai^{(\kappa)}(x) \cdot O\big(e^{-y}\big) = O\big(e^{-(x+y)}\big).
\]}
\[
a_\nu(x,y) = \sum_{\kappa,\lambda\in\{0,1\}} p_{\kappa\lambda}^m(x,y;h_\nu)  \Ai^{(\kappa)}(x)\Ai^{(\lambda)}(y)  +  h_\nu^{m+1} O(e^{-(x+y)}),
\]
where 
\begin{align*}
p^m_{00}(x,y;h) &= h(1- hy)(1 + hp_m(x;h))\big(yq_m(y;h)+\partial_y p_m(y;h)\big), \\*[1mm]
p^m_{01}(x,y;h) &= (1- hy)(1 + hp_m(x;h))\big(1+h(p_m(y;h)+\partial_y q_m(y;h))\big),\\*[1mm]
p^m_{10}(x,y;h) &= h^2(1- hy)q_m(x;h)\big(y q_m(y;h)+\partial_y p_m(y;h)\big),\\*[1mm]
p^m_{11}(x,y;h) &= h(1- hy)q_m(x;h)\big(1+h(p_m(y;h)+\partial_y q_m(y;h))\big)
\end{align*}
are rational polynomials in $x$, $y$ and $h$. In particular, those factorizations show
\begin{gather*}
p^m_{00}(x,y;h) = O(h), \qquad p^m_{11}(x,y;h) = O(h),\\*[1mm]
p^m_{01}(x,y;h) = 1 + O(h), \qquad p^m_{10}(x,y;h) = O(h^2).
\end{gather*}
If we denote by $\hat p^m_{\kappa\lambda}(x,y;h)$ the polynomials obtained from  $p^m_{\kappa\lambda}(x,y;h)$ after dropping all powers of $h$ that have an exponent larger than $m$ (thus contributing terms to the expansion that get absorbed in the error term), we obtain
\[
a_\nu(x,y) =  \Ai(x)\Ai'(y) + \sum_{\kappa,\lambda\in\{0,1\}} \hat p^m_{\kappa\lambda}(x,y;h_\nu)  \Ai^{(\kappa)}(x)\Ai^{(\lambda)}(y)  +  h_\nu^{m+1} O(e^{-(x+y)}),
\]
with a polynomial expansion
\[
\hat p^m_{\kappa\lambda}(x,y;h) = \sum_{j=1}^m \hat p_{j,\kappa\lambda}(x,y) h^j
\]
whose coefficient polynomials $\hat p_{j,\kappa\lambda}(x,y)$, being the unique expansion coefficients as $h_\nu\to 0^+$,  are now {\em independent} of $m$. Hence, the anti-symmetrization of $a_\nu(x,y)$ satisfies the uniform expansion (given the range of $x$ and $y$)  
\begin{multline}\label{eq:antisymmetric}
a_\nu(x,y)-a_\nu(y,x) = \Ai(x)\Ai'(y)- \Ai'(x)\Ai(y) \\*[1mm]
+ \sum_{\kappa,\lambda\in\{0,1\}} \hat q^m_{\kappa\lambda}(x,y;h_\nu)  \Ai^{(\kappa)}(x)\Ai^{(\lambda)}(y)  +  h_\nu^{m+1} O(e^{-(x+y)})
\end{multline}
with the polynomial expansion
\begin{gather*}
\hat q^m_{\kappa\lambda}(x,y;h) = \sum_{j=1}^m \hat q_{j,\kappa\lambda}(x,y) h^j,\\*[1mm]
\hat q_{j,00}(x,y) = \hat p_{j,00}(x,y) - \hat p_{j,00}(y,x),\qquad  \hat q_{j,11}(x,y) = \hat p_{j,11}(x,y) - \hat p_{j,11}(y,x)\\*[1mm]
\hat q_{j,01}(x,y) = \hat p_{j,01}(x,y) - \hat p_{j,10}(y,x),\qquad \hat q_{j,10}(x,y) = - \hat q_{j,01}(y,x).
\end{gather*}
Since the rational polynomials $\hat q_{j,00}(x,y)$ and  $\hat q_{j,11}(x,y)$ are anti-symmetric in $x$, $y$, they factor in the form
\begin{equation}\label{eq:divisibility}
(x-y) \times (\text{symmetric rational polynomial in $x$ and $y$});
\end{equation}
the first few cases are
\begin{align*}
\hat q_{1,00}(x,y) &= -\frac{3}{10}(x-y)(x^2+xy+y^2),\\*[1mm] 
\hat q_{1,11}(x,y) &= \frac{3}{10}(x-y)(x+y),\\*[1mm]
\hat q_{2,00}(x,y) &= \frac{1}{1400} (x-y)\Big(-235\big(x^3+y^3\big)-319xy(x+y)+56\Big),\\*[1mm]
\hat q_{2,11}(x,y) &= \frac{1}{1400} (x-y)\Big(340\big(x^2+y^2\big)+256xy\Big).
\end{align*}
Even though there is no straightforward structural reason for the rational polynomials $\hat q_{j,01}$ (and thus $\hat q_{j,10}$) to be divisible by $x-y$ as well, an inspection\footnote{See Remark~\ref{rem:PQcompute} for the computation of the polynomials $A_k$, $B_k$ and thus $q_{j,01}(x,y)$---a Mathematica notebook comes with the source files at \url{https://arxiv.org/abs/2301.02022}.} of the first cases reveals this to be true for at least $j=1,\ldots,m_*$; the first two of them being 
\begin{align*}
\hat q_{1,01}(x,y) &= \frac{2}{10}(x-y),\\*[1mm]
\hat q_{2,01}(x,y) &= \frac{1}{1400}(x-y) \Big(63\big(x^4+x^3 y-x^2 y^2-x y^3-y^4\big) +120 x + 64y\Big). 
\end{align*}
Now, by restricting ourselves to the explicitly checked cases $m\leq m_*$, we denote by $q^m_{\kappa\lambda}(x,y;h)$ the polynomials obtained from $\hat q^m_{\kappa\lambda}(x,y;h)$ after division by the factor $x-y$. Since \eqref{eq:antisymmetric} is an expansion of an anti-symmetric function with anti-symmetric remainder which can repeatedly be differentiated, division by $x-y$ yields removable singularities at $x=y$ and does not change the character of the expansion (see also the argument given in the proof of \cite[Prop.~8]{MR1986402}):
\[
\frac{a_\nu(x,y)-a_\nu(y,x)}{x-y} =
K_0(x,y)+ \sum_{\kappa,\lambda\in\{0,1\}} q^m_{\kappa\lambda}(x,y;h_\nu)  \Ai^{(\kappa)}(x)\Ai^{(\lambda)}(y) 
  +  h_\nu^{m+1} O\big(e^{-(x+y)}\big).
\]
The lemma now follows by multiplying this expansion with \eqref{eq:lemPhiExpan}, noting that the terms
\[
-r_k(x,y) (x-y)^2 K_0(x,y) =  -r_k(x,y) (x-y) \big(\!\Ai(x)\Ai'(y) - \Ai'(x)\Ai(y) \big) 
\]
also take the form asserted for the terms in the kernels $K_j$ ($j\geq 1$).

Finally, since all the expansions can repeatedly be differentiated under the same conditions while preserving their uniformity, the same holds for the resulting expansion of the kernel.

\medskip

{\em Comments on the unconditional proof of Yao and Zhang \cite{arXiv:2309.06733}.} Instead of using expansions of the Bessel functions of large order in the transition region, Yao and Zhang address expanding the integrable kernel
\[
\frac{a_\nu(x,y)-a_\nu(y,x)}{x-y}
\]
directly by the machinery of Riemann--Hilbert problems (see \cite{MR1730504} for a general discussion of kernels of that form and their induced integral operators). The advantage of such a direct approach is a better understanding of the polynomial coefficients in \eqref{eq:antisymmetric} and the divisibility by $x-y$ follows from an interesting algebraic structure of the Airy functions: namely, by the Airy differential equation there are rational polynomials $\alpha_k, \beta_k \in \Q[\xi]$ such that
\[
\Ai^{(k)}(\xi) = \alpha_k(\xi) \Ai(\xi) + \beta_k(\xi) \Ai'(\xi)
\] 
and the divisibility \eqref{eq:divisibility} turns out to be equivalent \cite[pp.~18--20]{arXiv:2309.06733} to the relation \cite[Lemma~4.1]{arXiv:2309.06733}
\[
\sum_{j=0}^m \binom{m}{j} 
\begin{vmatrix}
\alpha_j & \alpha_{m+1-j} \\
\beta_j & \beta_{m+1-j}
\end{vmatrix}
= 0 \qquad (m\geq 1),
\]
which can be proved by induction.
\end{proof}

\begin{remark} The case $m=0$ of Lemma~\ref{lem:besselkernexpan}, i.e.,
\[
\hat K_\nu^{\text{Bessel}}(x,y) = K_0(x,y) + h_\nu\cdot O\big(e^{-(x+y)}\big)
\]
is the $\beta=2$ case of \cite[Eq.~(4.8)]{MR1986402} in the work of Borodin and Forrester. There, in \cite[Prop.~8]{MR1986402} it is stated that this expansion would be uniformly valid for $x,y \geq t_0$. However, stated in such a generality, it is not correct (see Fn.~\ref{foot:bessel_counter}) and, in fact, similar to our proof given above, their proof is restricted to the range $t_0\leq x,y < h_\nu^{-1}$, which completely suffices to address the hard-to-soft edge transition. (See Remark~\ref{rem:besselshort} for yet another issue with \cite[Prop.~8]{MR1986402}.)
\end{remark}

To reduce the complexity of calculating the functional form of the first three finite-size correction terms in the hard-to-soft edge transition \eqref{eq:hard2soft}, we consider a second kernel transform.

\begin{lemma}\label{lem:Khexpan} For $h>0$, the Airy kernel $K_0$ and the first expansion kernels $K_1$, $K_2$, $K_3$ from Lemma~\ref{lem:besselkernexpan} we consider 
\[
K_{(h)}(x,y) := K_0(x,y) + K_1(x,y) h + K_2(x,y)h^2 + K_3(x,y)h^3
\]
and the transformation,\footnote{Note that for $z=1-h_\nu t$ we thus get $\nu z = \omega_\nu(t)$ and $\nu^{2/3}\zeta = s$ in Olver's expansion \eqref{eq:Olver1954}. As it turns out, by using this transformation, the kernel expansion simplifies in the same fashion also for $m\geq 2$, cf. Fn.~\ref{fn:m5}.} where $\zeta(z)$ is defined as in Sect.~{\rm \ref{sect:bessel}},
\begin{equation}\label{eq:psih}
s = \psi_h^{-1}(t):=2^{-1/3} h^{-1} \zeta(1-h t).
\end{equation}
Then $t=\psi_h(s)$ maps $s \in \R$ monotonically increasing to $-\infty < t < h^{-1}$, with $t\leq \mu h^{-1}$, $\mu = 0.94884\cdots$, when $s\leq 2h^{-1}$, and induces the symmetrically transformed kernel
\begin{subequations}\label{eq:auxkernelexpan}
\begin{align}
\tilde K_{(h)}(x,y) &:= \sqrt{\psi_h'(x)\psi_h'(y)} \, K_{(h)}(\psi_h(x),\psi_h(y))\\
\intertext{which expands as}
\tilde K_{(h)}(x,y) &= K_0(x,y) + \tilde K_1(x,y)h + \tilde K_2(x,y)h^2 + \tilde K_3(x,y)h^3 + h^4 \cdot O\big(e^{-(x+y)}\big),
\end{align}
uniformly in $s_0 \leq x,y \leq 2h^{-1}$ as $h\to 0^+$, $s_0$ being a fixed real number. The three kernels are
\begin{align}
\tilde K_1 &= \frac{1}{5}(\Ai\otimes \Ai' + \Ai'\otimes \Ai)\label{eq:K1tilde},\\*[2mm]
\tilde K_2 &= -\frac{48}{175}\Ai\otimes\Ai + \frac{11}{70}(\Ai\otimes\Ai'''+\Ai'''\otimes\Ai) - \frac{51}{350}(\Ai'\otimes\Ai''+\Ai''\otimes\Ai'),\\*[2mm]
\tilde K_3 &= -\frac{176}{1125} (\Ai\otimes \Ai'' + \Ai''\otimes \Ai) +\frac{13}{450} (\Ai\otimes \Ai^{\rm V} + \Ai^{\rm V}\otimes \Ai) + \frac{3728}{7875}\Ai'\otimes \Ai' \\*[1mm]
& \qquad  -\frac{583}{5250} (\Ai'\otimes \Ai^{\rm IV} + \Ai^{\rm IV}\otimes \Ai') + \frac{13}{225} (\Ai''\otimes \Ai''' + \Ai'''\otimes \Ai'').\notag
\end{align}
\end{subequations}
Preserving uniformity, the kernel expansion can repeatedly be differentiated w.r.t. $x$, $y$.
\end{lemma}
\begin{proof} Reversing the power series \eqref{eq:zetaseries} gives
\[
t = \psi_h(s) = s - \frac{3s^2}{10} h - \frac{s^3}{350} h^2 +\frac{479 s^4}{63000}h^3  + \cdots,
\]
which is uniformly convergent for $s_0 \leq s \leq 2h^{-1}$ since $|ht| \leq \mu < 1$. Taking the expressions for $K_1$, $K_2$ given in \eqref{eq:K2}, and for $K_3$ from the supplementary Mathematica notebook referred to in Fn.~\ref{fn:m5}, a routine calculation with truncated power series gives formula \eqref{eq:K1tilde} for $\tilde K_1$ and 
\[
\begin{gathered}
\tilde K_2(x,y) = \frac{1}{350} \Big(14\Ai(x)\Ai(y) + (-51 x + 55 y)\Ai(x)\Ai'(y) + (55x -51 y) \Ai'(x)\Ai(y) \Big),\\*[2mm]
\tilde K_3(x,y) = \frac{1}{15750}\Big(266(x+y)\Ai(x)\Ai(y) + (-1749 x^2+910 x y+455 y^2) \Ai(x)\Ai'(y)\\*[1mm]
 \qquad + (455 x^2+910 x y-1749 y^2) \Ai'(x)\Ai(y) + 460 \Ai'(x)\Ai'(y)\Big).
\end{gathered}
\]
The Airy differential equation $\xi \Ai(\xi) = \Ai''(\xi)$ implies the replacement rule
\begin{equation}\label{eq:airyrule}
\xi^j \Ai^{(k)}(\xi) = \xi^{j-1} \Ai^{(k+2)}(\xi) - k \xi^{j-1} \Ai^{(k-1)}(\xi) \qquad (j\geq 1, \; k\geq 0)
\end{equation}
which, if repeatedly applied to a kernel of the given structure, allows us to absorb any powers of $x$ and $y$ into higher order derivatives of $\Ai$. This process yields the asserted form of $\tilde K_2$ and $\tilde K_3$, which will be the preferred form in course of the calculations in Sect.~\ref{sect:functionalform}.

Since we stay within the range of uniformity of the power series expansions and calculations with truncated powers series are amenable to repeated differentiation, the result now follows from the bounds given in Sect.~\ref{sect:bounds}.
\end{proof}

\subsection{Proof of the general form of the expansion}\label{sect:generalform}

By Lemma~\ref{lem:besselkernexpan} and Thm.~\ref{thm:detexpan} we get (the Fredholm determinants are seen to be equal by transforming the integrals)
\begin{multline*}
E_2^{\text{hard}}(\phi_\nu(t);\nu) = \det(I - K_\nu^{\text{Bessel}})|_{L^2(0,\phi_\nu(t))} = \det(I - \hat K_\nu^{\text{Bessel}})|_{L^2(t,h_\nu^{-1})} \\*[1mm]
= F(t) + \sum_{j=1}^m F_{j}(t) h_\nu^j + h_\nu^{m+1} O(e^{-2t}) + e^{-h_\nu^{-1}} O(e^{-t}),
\end{multline*}
uniformly for $t_0\leq t < h_\nu^{-1}$ as $h_\nu\to0^+$; preserving uniformity, this expansion can be repeatedly differentiated w.r.t. the variable $t$. By Thm.~\ref{thm:detexpan}, the $F_{j}(t)$ are certain smooth functions that can be expressed in terms of traces of integral operators of the form given in Thm.~\ref{thm:detexpan}. Observing
\[
e^{-h_\nu^{-1}} < e^{-h_\nu^{-1}/2} e^{-t/2} = h_\nu^{m+1} O(e^{-t/2})\qquad (h_\nu\to0^+)
\]
we can combine the two error terms as $h_\nu^{m+1} O(e^{-3t/2})$. This finishes the proof of \eqref{eq:hard2soft}.

\subsection{Functional form of \boldmath $F_{1}(t)$, $F_{2}(t)$ and $F_3(t)$\unboldmath}\label{sect:functionalform}

Instead of calculating $F_{1}$, $F_{2}$, $F_3$ directly from the formulae in Thm.~\ref{thm:detexpan} applied to the kernels $K_1$, $K_2$  in (\ref{eq:K2}) (and to the unwieldy expression for $K_3$ obtained in the supplementary Mathematica notebook referred to in Fn.~\ref{fn:m5}), we will reduce them to the corresponding functions $\tilde F_1, \tilde F_2, \tilde F_3$ induced by the much simpler kernels $\tilde K_1$, $\tilde K_2$, $\tilde K_3$ in (\ref{eq:auxkernelexpan}).

\subsubsection{Functional form of $\tilde F_{1}(t)$ and $\tilde F_{2}(t)$}
Upon writing
\[
u_{jk}(t) = \tr\big((I-K_0)^{-1} \Ai^{(j)} \otimes \Ai^{(k)} \big)\big|_{L^2(t,\infty)}
\]
and observing (the symmetry of the resolvent kernel implies the symmetry $u_{jk}(t)= u_{kj}(t)$)
\begin{align*}
\tr\big((I-K_0)^{-1} \tilde K_1\big)\big|_{L^2(t,\infty)} &= \frac{2}{5} u_{10}(t),\\*[1mm]
\tr\big(((I-K_0)^{-1} \tilde K_1)^2\big)\big|_{L^2(t,\infty)} &= \frac{2}{25} \big(u_{00}(t)u_{11}(t) + u_{10}(t)^2\big),\\*[1mm]
\tr\big((I-K_0)^{-1} \tilde K_2\big)\big|_{L^2(t,\infty)} &= \frac{1}{175}\big(-48u_{00}(t)-51u_{21}(t) +55u_{30}(t)\big),
\end{align*}
the formulae of Thm.~\ref{thm:detexpan} applied to  $\tilde K_1$ and $\tilde K_2$ give 
\begin{subequations}
\begin{align}
\tilde F_{1}(t) &= -\frac{2}{5} F(t) u_{10}(t), \\*[1mm]
\tilde F_{2}(t) &= F(t)\left(\frac{48}{175} u_{00}(t) - \frac{1}{25} \begin{vmatrix}
u_{00}(t) & u_{01}(t) \\
u_{10}(t) & u_{11}(t)
\end{vmatrix} +\frac{51}{175} u_{21}(t) - \frac{11}{35} u_{30}(t) \right).\label{eq:F22tildeVanilla}
\end{align}
\end{subequations}
We recall from \cite[Remark~3.1]{arxiv.2206.09411} that the simple recursion 
\[
u_{jk}'(t) = u_{j+1,k}(t) + u_{j,k+1}(t) - u_{j0}(t) u_{k0}(t)
\]
yields similar formulae for the first few derivatives of the distribution $F(t)$:
\begin{equation}\label{eq:F2der}
\begin{gathered}
F'(t) = F(t)\cdot u_{00}(t),\quad
F''(t) = 2F(t)\cdot u_{10}(t),\quad
F'''(t) = 2F(t)\cdot \big(u_{11}(t)+u_{20}(t)\big),\\*[1mm]
F^{(4)}(t) = 2F(t)\cdot\left( \begin{vmatrix}
u_{00}(t) & u_{01}(t) \\
u_{10}(t) & u_{11}(t)
\end{vmatrix}  + 3u_{21}(t) + u_{30}(t) \right). 
\end{gathered}
\end{equation}
By a linear elimination of the terms 
\[
u_{00}(t), \quad u_{10}(t),\quad \begin{vmatrix}
u_{00}(t) & u_{01}(t) \\
u_{10}(t) & u_{11}(t)
\end{vmatrix}
\]
we obtain, as an intermediate step,
\begin{subequations}\label{eq:tildeF22}
\begin{equation}
\tilde F_{1}(t) = -\frac{1}{5} F''(t),\quad \tilde F_{2}(t) = \frac{48}{175} F'(t) - \frac{1}{50} F^{(4)}(t) + \frac{24}{175}F(t)\big(3u_{21}(t) - 2 u_{30}(t)\big). 
\end{equation}
To simplify even further, we have to refer to the full power of the general Tracy--Widom theory (i.e., representing $F$ in terms of Painlevé II): by advancing its set of formulae, Shinault and Tracy \cite[p.~68]{MR2787973} showed, through an explicit inspection of each single case, that the functions $F(t)\cdot u_{jk}(t)$  in the range $0 \leq j+k\leq 8$ are linear combinations of the form
\begin{equation}\label{eq:ShinaultForm}
 p_1(t) F'(t) + p_2(t) F''(t) + \cdots + p_{j+k+1}(t) F^{(j+k+1)}(t)
\end{equation}
with  rational polynomials $p_\kappa(t)$ (depending, of course, on $j$, $k$). They conjectured this structure to be true for all~$j$,~$k$. In particular, their table \cite[p.~68]{MR2787973} has the entries
\[
F(t) \cdot u_{21}(t) = -\frac{1}{4}F'(t) +\frac{1}{8}F^{(4)}(t), \quad F(t)\cdot  u_{30}(t) = \frac{7}{12} F'(t) + \frac{t}{3} F''(t) + \frac{1}{24} F^{(4)}(t).
\]
This way we get, rather unexpectedly, the simple and short form\footnote{Note that a direct application of the table in \cite[p.~68]{MR2787973} 
to \eqref{eq:F22tildeVanilla} produces a far less appealing result, namely
\[
\tilde F_2(t) = \frac{19}{1050}F'(t) + \frac{t F'(t)^2}{75F(t)} - \frac{11t}{105} F''(t) + \frac{F''(t)^2}{100 F(t)} - \frac{F'(t)F'''(t)}{75F(t)} + \frac{7}{300}F^{(4)}(t),
\]
which is no longer linear in $F$ and its derivatives.
} 
\begin{equation}
\tilde F_2(t) = \frac{2}{175} F'(t) - \frac{16t}{175}F''(t) + \frac{1}{50} F^{(4)}(t).
\end{equation}
\end{subequations}

\subsubsection{Functional form of $\tilde F_{3}(t)$} 

For the $\tilde F_j(t)$ with $j\geq 3$, calculating such functional forms  requires a more systematic, algorithmic approach.
By ``reverse engineering'' the remarks of Shinault and Tracy about validating their table \cite[p.~68]{MR2787973}, we have actually found an algorithm to compile such a table, see Appendix~\ref{app:ST}. This algorithm can also be applied to {\em nonlinear} rational polynomials of the terms $u_{jk}$, resulting either in an expression of the desired form \eqref{eq:ShinaultForm} (with $j+k+1$ replaced by some $n$), or a message that such a form does not exist (for the given $n$).

Now, if we evaluate the expansion function $G_3(t)=F(t)d_3(t)$ of Thm.~\ref{thm:detexpan} by using \eqref{eq:d3}
and rewrite the traces in terms of the $u_{jk}$, we obtain
\begin{subequations}\label{eq:F3tilde}
\begin{align}
\tilde F_3(t) &= F(t) \left(-\frac{3728 }{7875}u_{11}(t) +\frac{352 }{1125}u_{20}(t) -\frac{51}{875} 
\begin{vmatrix}
u_{10}(t) & u_{11}(t) \\
u_{20}(t) & u_{21}(t)
\end{vmatrix}
 \right. \label{eq:F3tildeU}\\*[2mm]
& \hspace*{-0.25cm}\left. -\frac{11}{175} \begin{vmatrix}
u_{00}(t) & u_{01}(t) \\
u_{30}(t) & u_{31}(t)
\end{vmatrix}-\frac{26}{225} u_{32}(t)+\frac{583 }{2625}u_{41}(t)-\frac{13}{225} u_{50}(t)\right)\notag
\intertext{which evaluates by the algorithm of Appendix~\ref{app:ST} further to}
&= \frac{64 t }{7875}F'(t) -\frac{24 t^2 }{875}F''(t) -\frac{122}{7875}F'''(t) 
+\frac{16 t}{875} F^{(4)}(t) -\frac{1}{750} F^{(6)}(t).
\end{align}\\*[-1.5mm]
\end{subequations}
Alternatively, we arrive here by combining the tabulated expressions for $u_{11}$, $u_{20}$, $u_{32}$, $u_{41}$, $u_{50}$ as displayed in \cite[p.~68]{MR2787973}  with those for both of the $2\times 2$ determinants in \eqref{eq:det22}.

\subsubsection{Lifting to the functional form of $F_{1}(t)$, $F_{2}(t)$, $F_{3}(t)$}

The relation between $F_{1}(t)$, $F_{2}(t)$, $F_3(t)$ and their counterparts with a tilde is established by Lemma~\ref{lem:Khexpan}.
By using the notation introduced there, with $t$ being any fixed real number, the expansion parameter $h$ sufficiently small and
$s=\psi_{h}^{-1}(t)$, Thm.~\ref{thm:detexpan} yields (the Fredholm determinants are seen to be equal by transforming the integrals)
\begin{equation}\label{eq:compareKh}
\begin{aligned}
& \det(I-K_h)|_{L^2(t, \,\mu h^{-1})}  = F(t) + F_{1}(t) h + F_{2}(t)h^2 + F_{3}(t)h^3 + O(h^4) \\*[1mm]
= &\det(I-\tilde K_h)|_{L^2(s, \, 2 h^{-1})} = F(s) + \tilde F_{1}(s) h + \tilde F_{2}(s)h^2 + \tilde F_{3}(s)h^3 + O(h^4),
\end{aligned}
\end{equation}
where we have absorbed the exponentially small contributions $e^{-\mu h^{-1}} O(e^{-t})$ and $e^{-2h^{-1}} O(e^{-t})$ into the $O(h^3)$ error term. Using the power series \eqref{eq:zetaseries}, that is,
\[
s = 2^{-1/3} h^{-1} \zeta(1-h t) = t+ \frac{3t^2}{10}h + \frac{32t^3}{175} h^2  + \frac{1037t^4}{7875} h^3 + \cdots,
\]
we get by Taylor expansion, for any smooth function $G(s)$,
\begin{align*}
G(s) &= G(t) + \frac{3t^2}{10} G'(t) h + \left(\frac{32t^3}{175} G'(t) + \frac{9t^4}{200} G''(t) \right) h^2 \\*[1mm]
&\qquad +\left(\frac{1037 t^4}{7875} G'(t)+\frac{48t^5}{875}  G''(t)+ \frac{9 t^6 }{2000}G'''(t)\right)h^3 + O(h^4).
\end{align*}
By plugging this into \eqref{eq:compareKh} and comparing coefficients we obtain
\begin{align*}
F_{1}(t) &= \tilde F_{1}(t) + \frac{3t^2}{10} F'(t),\\*[2mm]
F_{2}(t) &=\tilde F_{2}(t) + \frac{3t^2}{10} \tilde F_{1}'(t) + \frac{32t^3}{175} F'(t) + \frac{9t^4}{200} F''(t),\\*[2mm]
F_{3}(t) &= \tilde F_3(t) + \frac{32t^3}{175}  \tilde F_1'(t)+\frac{9t^4}{200}  \tilde F_1''(t)+\frac{3t^2}{10} \tilde F_2'(t)
+\frac{1037 t^4 }{7875}F'(t)+\frac{48t^5 }{875} F''(t)+ \frac{9 t^6 F'''(t)}{2000}.
\end{align*}
Combined with \eqref{eq:tildeF22}, this finishes the proof of \eqref{eq:F22}.

\subsection{Simplifying the form of Choup's Edgeworth expansions}\label{sect:Choup}
When, instead of the detour via $\tilde K_1$, Thm.~\ref{thm:detexpan} is directly applied to the kernel $K_1$ in \eqref{eq:K1}, 
we get 
\[
F_{1}(t) = - \frac{1}{5}F''(t) + \frac{3}{10} F(t) \tr\big((I-K_0)^{-1} L\big)\big|_{L^2(t,\infty)},
\]
where
\begin{subequations}\label{eq:t2F21}
\begin{equation}\label{eq:L}
L(x,y) = (x^2+xy+y^2)\Ai(x)\Ai(y) - (x+y)\Ai'(x)\Ai'(y).
\end{equation}
Now, a comparison with \eqref{eq:F21} proves the useful formula\footnote{Note that our derivation of this formula does only depend on Fredholm determinants and does not use any representation in terms of Painlevé II. 
Based on Painlevé representations, it has been derived, implicitly though, in the recent work of Forrester and Mays  \cite{arxiv.2205.05257}: see Eqns. (1.16), (1.19), (2.17) and (2.29) there. A further alternative derivation follows from observing that, by repeated application of \eqref{eq:airyrule},
\[
\tr\big((I-K_0)^{-1} L\big)\big|_{L^2(t,\infty)} = -2u_{10}(t) + u_{22}(t) - 2u_{31}(t) + 2u_{40}(t)  
\]
and using the table for the functions $F(t)\cdot u_{jk}(t)$ ($0\leq j+k \leq 8$) compiled in \cite[p.~68]{MR2787973} (which is based on an extension of formulae of the Tracy--Widom theory that represents $F$ in terms of Painlevé II).}
\begin{equation}
F(t) \tr\big((I-K_0)^{-1} L\big)\big|_{L^2(t,\infty)} = t^2 F'(t).
\end{equation}
\end{subequations}
As an application to the existing literature, this formula helps us to simplify the results obtained by Choup for the soft-edge limit expansions of GUE and LUE: that is, 
when studying the distribution of the largest eigenvalue distribution function in $\text{GUE}_n$ and $\text{LUE}_{n,\nu}$ (dimension $n$, parameter $\nu$) as $n\to\infty$. In fact, since the kernel $L$ appears in the first finite-size correction term of a corresponding kernel expansion \cite[Thm.~1.2/1.3]{MR2233711}, lifting that expansion to the Fredholm determinant by Thm.~\ref{thm:detexpan}
allows us to recast \cite[Thm.~1.4]{MR2233711} in a simplified form: namely, denoting the maximum eigenvalues by $\lambda_{n}^\text{G}$ and $\lambda_{n,\nu}^\text{L}$, we obtain, locally uniform in $t$ as $n\to \infty$,
\begin{align}
\prob\Big( \lambda_{n}^\text{G} \leq \sqrt{2n} + t \cdot 2^{-1/2} n^{-1/6} \Big) &= F(t) + \frac{n^{-2/3}}{40}\big(2t^2 F'(t) - 3F''(t)\big) + O(n^{-1}),\label{eq:GUEn}\\*[2mm]
\prob\Big( \lambda_{n,\nu}^\text{L} \leq 4n+2\nu+t\cdot 2(2n)^{1/3} \Big) &= F(t) - \frac{2^{1/3}n^{-2/3}}{10} \big(t^2 F'(t) + F''(t)\big) + O(n^{-1});
\end{align}
 a result, which answers a question suggested by Baik and Jenkins \cite[p.~4367]{MR3161478}.

\section{Expansion of the Poissonized Length Distribution}\label{sect:BaikJenkins}

The Poissonization of the length distribution requires the hard-to-soft edge transition of Thm.~\ref{thm:hard2soft}
to be applied to the probability distribution $E_2^\text{hard}(4r;\nu)$ (for integer $\nu=l$, but we consider the case of general $\nu>0$ first). For large intensities $r$ the mode of this distribution is located in the range of those parameters $\nu$ for which the scaled variable
\begin{subequations}
\begin{equation}\label{eq:tnu}
t_\nu(r)  := \frac{\nu-2\sqrt{r}}{r^{1/6}} \qquad (r>0).
\end{equation}
stays bounded. It is convenient to note that $t_\nu(r)$ satisfies the differential equation
\begin{equation}\label{eq:tnuprime}
t_\nu'(r) = -r^{-2/3} - \frac{r^{-1}}{6} t_\nu(r)\qquad (r>0).
\end{equation}
\end{subequations}
In these terms we get the following theorem.

\begin{theorem}\label{thm:poissonExpandnu}
There holds the expansion
\begin{equation}\label{eq:poissonExpandnu}
E_2^\text{\em hard}(4r;\nu) =  F(t) + \sum_{j=1}^m F_{j}^P(t)\, r^{-j/3} + r^{-(m+1)/3} \cdot O\big(e^{-t}\big)\bigg|_{t=t_\nu(r)},
\end{equation}
which is uniformly valid when $r,\nu \to\infty$ subject to 
\[
t_0 \leq t_\nu(r) \leq r^{1/3},
\]
with $m$ being any fixed non-negative integer and $t_0$ any fixed real number. Preserving uniformity, the expansion can be repeatedly differentiated w.r.t. the variable $r$. Here the $F_{j}^P$ are certain smooth functions{\rm;} the first three are\/\footnote{To validate formulae (\ref{eq:F22P}a--c), Fig.~\ref{fig:F23P} plots $F_3^P(t)$ next to the approximation
\begin{equation}\label{eq:F23P}
F_{3}^P\big(t_\nu(r)\big) \approx r \cdot \Big( E_2^{\text{hard}}(4r;\nu) - F(t) - F_{1}^P(t)r^{-1/3} - F_{2}^P(t) r^{-2/3}\Big)\,\bigg|_{t=t_\nu(r)}
\end{equation}
for $r=250$ and $r=2000$, varying $\nu$ in such a way that $t=t_\nu(r)$ covers the interval $[-6,2]$.}
\begin{subequations}\label{eq:F22P}
\begin{align}
F_{1}^P(t) &= -\frac{t^2}{60} F'(t) - \frac{1}{10} F''(t),\label{eq:F21P}\\*[2mm]
F_{2}^P(t) &= \Big(\frac{1}{350} + \frac{2t^3}{1575}\Big) F'(t) + \Big(\frac{11t}{1050} + \frac{t^4}{7200}\Big) F''(t) + \frac{t^2}{600} F'''(t) + \frac{1}{200} F^{(4)}(t),\\*[2mm]
F_3^P(t) &= 
-\Big(\frac{t}{1125}+\frac{41t^4}{283500}\Big) F'(t)
-\Big(\frac{11 t^2}{6300}+ \frac{t^5}{47250}\Big) F''(t)   \\*[1mm]
&\qquad-\Big(\frac{61}{31500}+\frac{19 t^3}{63000}+\frac{t^6}{1296000} \Big) F'''(t)
-\Big(\frac{11 t}{10500} + \frac{t^4}{72000}\Big) F^{(4)}(t) \notag \\*[1mm]
&\qquad -\frac{t^2}{12000}F^{(5)}(t)
-\frac{1}{6000}F^{(6)}(t). \notag
\end{align}
\end{subequations}
\end{theorem}

\begin{figure}[tbp]
\includegraphics[width=0.325\textwidth]{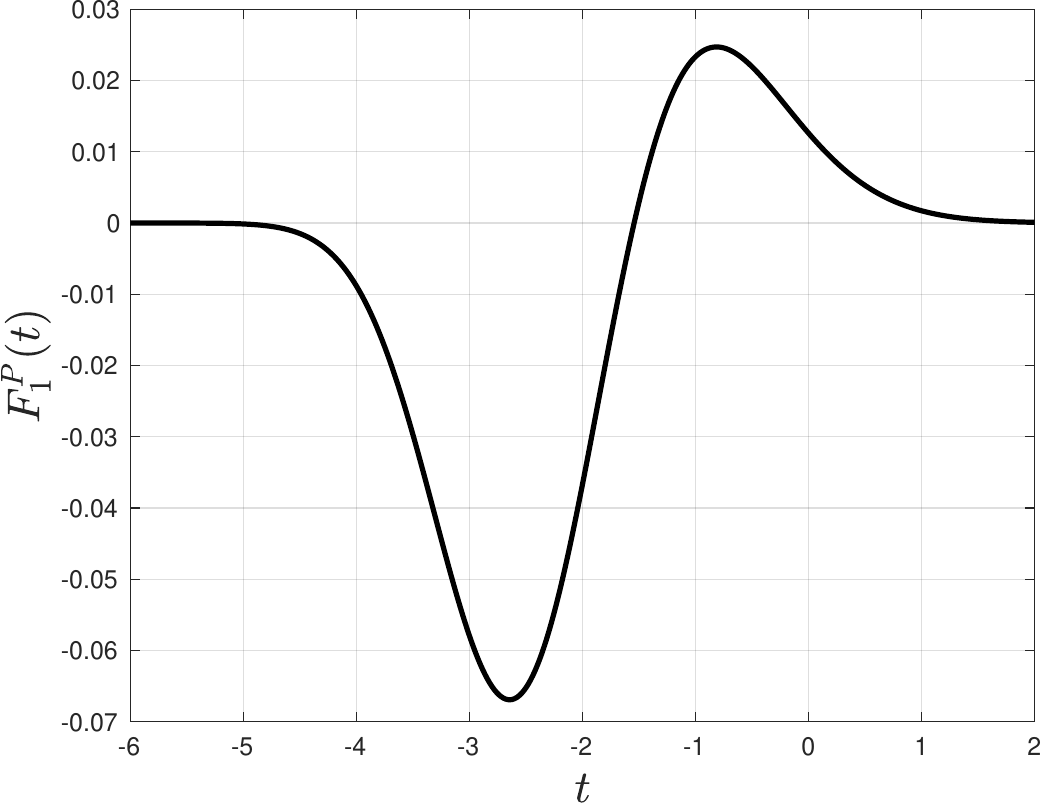}\hfil\,
\includegraphics[width=0.325\textwidth]{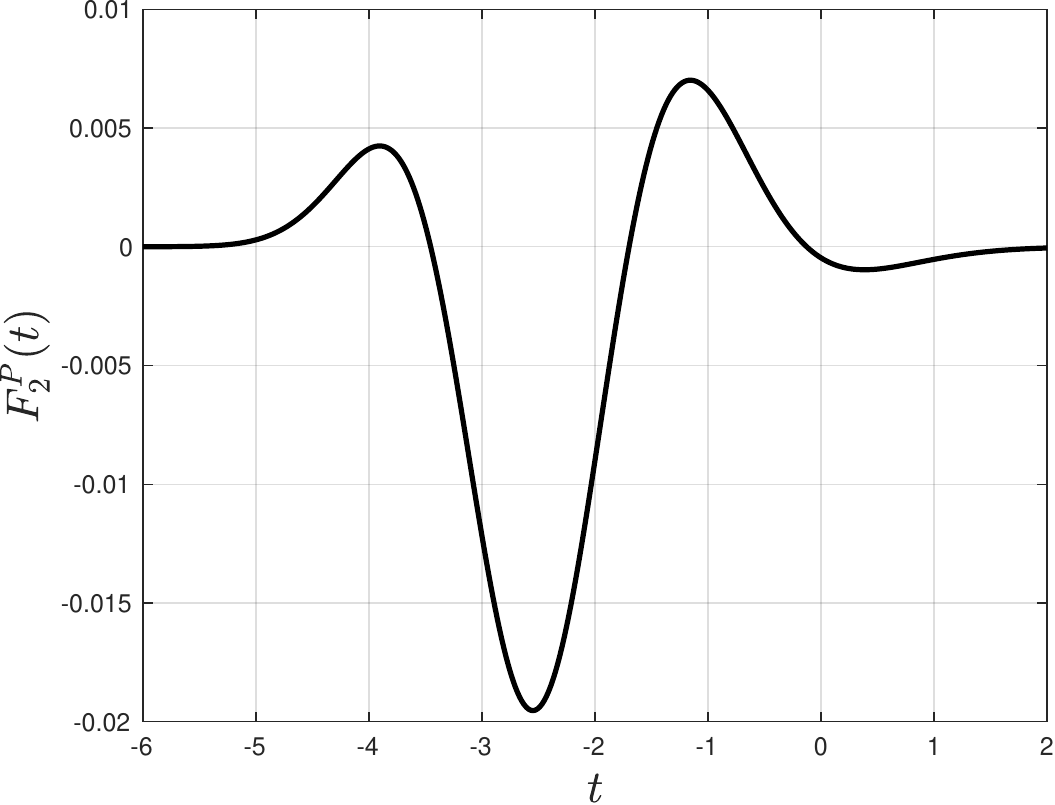}\hfil\,
\includegraphics[width=0.3145\textwidth]{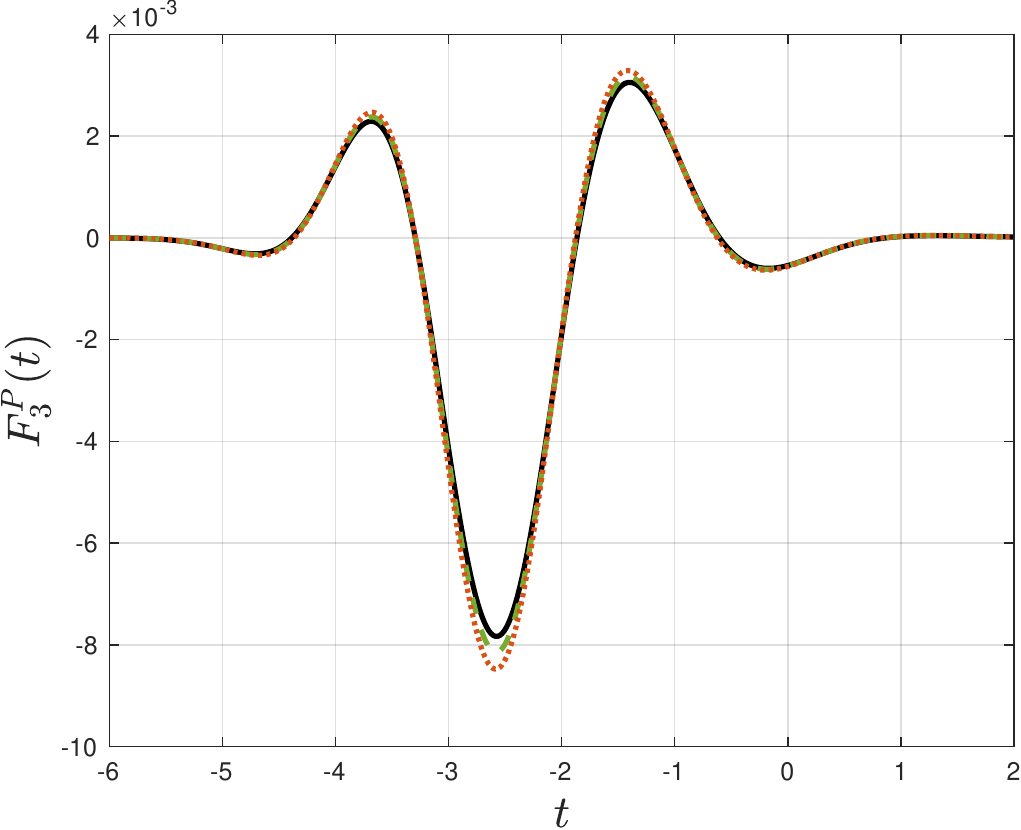}
\caption{{\footnotesize Plots of $F_{1}^P(t)$ (left panel) and $F_{2}^P(t)$ (middle panel) as in (\ref{eq:F22P}a/b).
The right panel shows $F_3^P(t)$ as in (\ref{eq:F22P}c) (black solid line) with the approximations \eqref{eq:F23P} for $r=250$ (red dotted line) and $r=2000$ (green dashed line); the parameter $\nu$ has been varied such that $t_\nu(r)$ covers the range of  $t$ on display. Note that the functions $F_{j}^P(t)$ ($j=1,2,3$) are about two orders of magnitude smaller in scale than their counterparts in Fig.~\ref{fig:hard2soft}.}}
\label{fig:F23P}
\end{figure}

\begin{proof} For $r, \nu > 0$ (i.e., equivalently, $t>-2r^{1/3}$ and $s < h_\nu^{-1}$) the transformations 
\[
4r = \phi_\nu(s), \quad t = t_\nu(r), 
\]
are inverted by the expressions
\begin{equation}\label{eq:sh}
s = \frac{t}{\big(1+\tfrac{t}2 r^{-1/3}\big)^{1/3}},\quad h_\nu = \frac{r^{-1/3}}{2\big(1+\tfrac{t}2  r^{-1/3}\big)^{2/3}}.
\end{equation}
For $t_0\leq t \leq r^{1/3}$ we get
\[
s_0:= (\tfrac{2}{3})^{1/3} t_0 \leq (\tfrac{2}{3})^{1/3} t \leq s < h_\nu^{-1}
\]
and observe that in this range of $t$ the expressions in \eqref{eq:sh} expand as uniformly convergent power series in powers of $r^{-1/3}$, starting with
\[
s = t - \frac{t^2}{6}r^{-1/3} + \frac{t^3}{18}r^{-2/3} -\frac{7 t^4}{324}r^{-1} + \cdots, \quad h_\nu = \frac{1}{2}r^{-1/3} - \frac{t}{6}r^{-2/3} +\frac{5 t^2}{72}r^{-1} + \cdots.
\]
If we plug these uniformly convergent power series into the uniform expansion of Thm.~\eqref{thm:hard2soft}, 
\[
E_2^\text{hard}(4r;\nu) = E_2^\text{hard}(\phi_\nu(s);\nu) = F(s) + \sum_{j=1}^m F_{1}(s) h_\nu^j + h_\nu^{m+1} O(e^{-3s/2}),
\]
we obtain the asserted form of the expansion \eqref{eq:poissonExpandnu} (as well as the claim about the repeated differentiability), simplifying the exponential error term by observing that $(3/2)^{2/3} > 1$. 
In particular, the first three correction terms in \eqref{eq:poissonExpandnu} are thus
\begin{align*}
F_{1}^P(t) &= \frac{1}{2} F_{1}(t) - \frac{t^2}{6} F'(t),\\*[2mm]
F_{2}^P(t) &= \frac{1}{4} F_{2}(t) - \frac{t}{6} F_{1}(t) - \frac{t^2}{12} F_{1}'(t) + \frac{t^3}{18} F'(t) + \frac{t^4}{72}F''(t),\\*[2mm]
F_{3}^P(t) &= \frac{1}{8}F_3(t) -\frac{t}{6} F_2(t)-\frac{t^2}{24}  F_2'(t)+\frac{5t^2}{72}  F_1(t)+\frac{ t^3}{18} F_1'(t)+\frac{t^4}{144}  F_1''(t)\\*[2mm]
&\qquad -\frac{7t^4}{324}  F'(t) -\frac{t^5}{108}  F''(t)-\frac{t^6 }{1296}F'''(t).
\end{align*}
Together with the expressions given in \eqref{eq:F22} this yields the functional form asserted in \eqref{eq:F22P}.
\end{proof}

By \eqref{eq:poisson}, specializing Thm.~\ref{thm:poissonExpandnu} to the case of integer parameter $\nu = l$ yields
the expansion
\begin{equation}\label{eq:poissonExpand}
\prob\big(L_{N_r} \leq l\big) = F(t) + \sum_{j=1}^m F_{j}^P(t)\, r^{-j/3} + r^{-(m+1)/3} \cdot O\big(e^{-t}\big) \bigg|_{t=t_l(r)}
\end{equation} 
which is uniformly valid under the conditions stated there. 

\begin{remark}\label{rem:gauss}
In the literature, scalings are often applied to the probability distribution rather than to the expansion terms. Since $L_{N_r}$ is a an integer-valued random variable, one has to exercise some care with the scaled distribution function being piecewise constant. Namely, for $t\in \R$ being any fixed number, one has
\[
\prob\left(\frac{L_{N_r} - 2\sqrt{r}}{r^{1/6}}\leq t\right) = \prob\big(L_{N_r} \leq l\big),\qquad l = \big\lfloor2\sqrt{r} + tr^{1/6}\big\rfloor,
\]
where $\lfloor\cdot\rfloor$ denotes the Gauss bracket. Thus, by defining
\[
t^{(r)} = \frac{\big\lfloor2\sqrt{r} + tr^{1/6}\big\rfloor - 2\sqrt{r}}{r^{1/6}}
\]
and noting that $t^{(r)}$ stays bounded when $r\to\infty$ while $t$ is fixed, \eqref{eq:poissonExpand} takes the form
\begin{equation}\label{eq:probPoissonExpan}
\prob\left(\frac{L_{N_r} - 2\sqrt{r}}{r^{1/6}}\leq t\right)  = F\big(t^{(r)}\big) + \sum_{j=1}^m F_{j}^P\big(t^{(r)}\big)\, r^{-j/3} + O\big(r^{-(m+1)/3}\big)\qquad (r\to\infty).
\end{equation}
If one chooses to re-introduce the continuous variable $t$ in (parts of) the expansion terms, one has to take into account that
\begin{equation}\label{eq:gaussbracket}
t^{(r)} = t + O(r^{-1/6})\qquad (r\to\infty)
\end{equation}
 where the exponent $-1/6$ in the error term is sharp.
For example, this gives (as previously obtained by Baik and Jenkins \cite[Thm.~1.3]{MR3161478} using the technology of Riemann--Hilbert problems to prove the expansion and Painlevé representations to put $F_{1}^P$ into the simple functional form \eqref{eq:F21P})
\begin{equation}\label{eq:BaikJenkinsOriginal}
\prob\left(\frac{L_{N_r} - 2\sqrt{r}}{r^{1/6}}\leq t\right)  = F\big(t^{(r)}\big) + F_{1}^P(t)\, r^{-1/3} + O(r^{-1/2})\qquad (r\to\infty),
\end{equation}
where the $O(r^{-1/2})$ error term is governed by the Gauss bracket in \eqref{eq:gaussbracket} and cannot be improved upon---completely dominating the order $O(r^{-2/3})$ correction term in \eqref{eq:probPoissonExpan}.
Therefore, claiming an $O(r^{-2/3})$ error term to hold in \eqref{eq:BaikJenkinsOriginal} as stated in \cite[Prop.~1.1]{arxiv.2205.05257} neglects the effect of the Gauss bracket.\footnote{Furthermore, the right panel of \cite[Fig.~3]{arxiv.2205.05257} is not showing an approximation of the $O(r^{-2/3})$ term in \eqref{eq:probPoissonExpan}, let alone in \eqref{eq:BaikJenkinsOriginal}, but instead an approximation of the $O(\nu^{-4/3})$ term in the auxiliary expansion
\[
E_2^\text{hard}\Big(\Big(\nu - t (\tfrac{\nu}{2})^{1/3} + \tfrac{t^2}{6} (\tfrac{\nu}{2})^{-1/3} \Big)^2; \nu\Big) = F(t) + \hat F_{1}(t)\, (\tfrac{\nu}{2})^{-2/3} +  \hat F_{2}(t)\, (\tfrac{\nu}{2})^{-4/3} + O(\nu^{-2}) \quad (\nu\to \infty),
\]
cf. \cite[Eqs.~(2.3/2.33)]{arxiv.2205.05257}. Now, Thm.~\ref{thm:hard2soft} and the formulae in \eqref{eq:F22P} yield the simple relations
\[
\hat F_{1}(t) = F_{1}^P(t),\qquad \hat F_{2}(t) = F_{2}^P(t) + \frac{t}{3}F_{1}^P(t),
\]
which are consistent with \cite[Fig.~3]{arxiv.2205.05257}; the additional term $t F_{1}^P(t)/3$ explains the different shape of $\hat F_{2}(t)$, as displayed in the right panel there, when compared to $F_{2}^P(t)$, as shown in the middle panel of Fig.~\ref{fig:F23P} here.
}
\end{remark}

\part*{Part II: Results Based on the Tameness Hypothesis}

\section{De-Poissonization and the Expansion of the Length Distribution}\label{sect:main}

\subsection{Expansion of the CDF}\label{sect:CDFexpan}

In this section we prove (subject to a tameness hypothesis on the zeros of the generating functions in a sector of the complex plane) an expansion of the CDF $\prob(L_n \leq l)$ of the length distribution near its mode. The general form of such an expansion was conjectured in the recent papers \cite{arxiv.2206.09411,arxiv.2205.05257} where approximations of the graphical form of the first few terms were provided (see \cite[Figs.~4/6]{arxiv.2206.09411} and \cite[Fig.7]{arxiv.2205.05257}). Here, for the first time, we give the functional form of these terms. The underlying tool is analytic de-Poissonization, a technique that was developed in the 1990s in theoretical computer science and analytic combinatorics.

To prepare for the application of analytic de-Poissonization in the form of the Jacquet--Szpankowski Thm.~\ref{thm:jasz}, we consider any fixed compact interval $[t_0,t_1]$ and a sequence of integers $l_n \to \infty$ such that
\begin{equation}\label{eq:tnstar}
t_0 \leq t_n^*:= t_{l_n}(n)\leq t_1 \qquad (n=1,2,3,\ldots).
\end{equation}
When $n-n^{3/5} \leq r \leq n+n^{3/5}$ and $n\geq n_0$ with $n_0$ large enough (depending only on $t_0, t_1$)
 we thus get the uniform bounds\footnote{Observe that 
\[
2\sqrt{r} + t r^{1/6} = 2\sqrt{n} + t n^{1/6} + O(n^{1/10})
\]
uniformly for $t_0\leq t \leq t_1$ and $n-n^{3/5}\leq r \leq n + n^{3/5}$ as $n\to\infty$.}
\[
2\sqrt{r} + (t_0-1) r^{1/6} \leq l_n \leq 2\sqrt{r} + (t_1+1) r^{1/6}.
\]
We write the induced Poisson generating function, and exponential generating function, of the length distribution as
\begin{equation}\label{eq:egf_family}
P_k(z) := P(z; l_k) = e^{-z} \sum_{n=0}^\infty \prob(L_n \leq l_k) \frac{z^n}{n!},\qquad f_k(z):= e^z P_k(z).
\end{equation}
By \eqref{eq:poisson} we have $P_n(r) = E_2^\text{hard}(4r;l_n)$ for real $r>0$, so that Thm.~\ref{thm:poissonExpandnu}
(see also \eqref{eq:poissonExpand}) gives the expansion
\begin{equation}\label{eq:PnExpansion}
P_n(r) =  F(t) + \sum_{j=1}^m F_{j}^P(t) \, r^{-j/3} + O(r^{-(m+1)/3})\,\bigg|_{t=t_{l_n}(r)},
\end{equation}
uniformly valid when $n - n^{3/5} \leq r \leq n + n^{3/5}$ as $n\to\infty$, $m$ being any fixed non-negative integer. Here, the implied constant in the error term depends only on $t_0, t_1$, but not on the specific sequence $l_n$. Preserving uniformity, the expansion can be repeatedly differentiated w.r.t. the variable $r$. In particular, using the differential equation \eqref{eq:tnuprime} we get that $P_n^{(j)}(n)$ expands in powers of $n^{-1/3}$, starting with a leading order term of the form
\begin{subequations}\label{eq:Pnnexpand}
\begin{equation}\label{eq:Pnnlead}
P_n^{(j)}(n)= (-1)^j F^{(j)} (t^*_n) n^{-2j/3} + O(n^{-(2j+1)/3})\qquad (n\to\infty);
\end{equation}
the specific cases to be used below are (see~\eqref{eq:PrPrime} for $P_n'(n)$)
\begin{align}
P_n(n) &= F(t) + F_{1}^P(t) n^{-1/3} + F_{2}^P(t) n^{-2/3} + F_{3}^P(t) n^{-1} + O(n^{-4/3})\,\bigg|_{t=t_n^*},\\*[2mm] 
P_n''(n) &= F''(t) n^{-4/3} + \left(F_{1}^{P}\!\!\phantom{|}''(t) + \frac{5}{6} F'(t) + \frac{t}{3} F''(t) \right) n^{-5/3}\\*[1mm] 
&\hspace*{-0.75cm}+ \left(\frac{3}{2}F_{1}^{P}\!\!\phantom{|}'(t) + \frac{t}{3} F_{1}^{P}\!\!\phantom{|}''(t) + F_{2}^{P}\!\!\phantom{|}''(t)+ \frac{7 t}{36} F'(t)+\frac{t^2}{36}  F''(t) \right) n^{-2}+ O(n^{-7/3})\,\bigg|_{t=t_n^*},\notag\\*[2mm]
P_n'''(n) &= -F'''(t) n^{-2} + O(n^{-7/3}) \,\bigg|_{t=t_n^*},\\*[2mm]
P_n^{(4)}(n) &= F^{(4)}(t) n^{-8/3} + \left(F_{1}^{P}\!\!\phantom{|}^{(4)}(t)+5 F'''(t)+ \frac{2 t}{3} F^{(4)}(t)\right) n^{-3} +  O(n^{-10/3})\,\bigg|_{t=t_n^*},\\*[2mm]
P_n^{(6)}(n) &= F^{(6)}(t) n^{-4} + O(n^{-13/3})\,\bigg|_{t=t_n^*},
\end{align}
\end{subequations}
where the implied constants in the error terms depend only on $t_0, t_1$.

We recall from the results of \cite[Sect.~2]{arxiv.2206.09411} (note the slight differences in notation), and the proofs given there, that the exponential generating functions $f_n(z)$ are entire functions of genus zero having, for each $0<\epsilon<\pi/2$, only finitely many zeros\footnote{Because of $f_n(r)>0$ for $r>0$, the real zeros of $f_n$ are negative and the complex ones are coming in conjugate pairs.} in the sector $|\!\arg z| \leq \pi/2 + \epsilon$. If we denote the real auxiliary functions (cf. Def.~\ref{def:Hayman}) of $f_n(r) = e^r P_n(r)$ by $a_n(r)$ and $b_n(r)$, the expansion \eqref{eq:PnExpansion}, and its derivatives based on \eqref{eq:tnuprime}, give (cf. also \eqref{eq:arbr})
\begin{equation}\label{eq:anbn}
a_n(r) = r + O(r^{1/3}),\qquad b_n(r) = r + O(r^{2/3}),
\end{equation}
uniformly valid when $n - n^{3/5} \leq r \leq n + n^{3/5}$ as $n\to\infty$;  the implied constants in the error terms depend only on $t_0, t_1$.

Analytically, we lack the tools to study the asymptotic distribution of the finitely many zeros of $f_n(z)$ in the sector $|\!\arg z| \leq \pi/2 + \epsilon$ as $n\to\infty$. Numerically, we proceed as follows. The meromorphic logarithmic derivative of $f_n(z)$ takes the form \cite[§3.1]{arxiv.2206.09411}
\[
\frac{f_n'(z)}{f_n(z)} = 1 - \frac{v_{l_n}(z)}{z},
\]
where $v_l$ satisfies a Jimbo--Miwa--Okamoto $\sigma$-form of the Painlevé III equation \cite[Eq.~(31)]{arxiv.2206.09411}, or alternatively, a certain Chazy I equation \cite[Eq.~(34)]{arxiv.2206.09411}. Because of $f_n(0)=1$ the zeros of $f_n$ are in a one-to-one correspondence to the pole field of the meromorphic function $v_{l_n}$. Fornberg and Weideman \cite{MR2804960} developed a numerical method, the {\em pole field solver}, specifically for the task of numerically studying the pole fields of equations of the Painlevé class. They documented results for Painlevé I \cite{MR2804960}, Painlevé II \cite{MR3260257}, its imaginary variant \cite{MR3390079} and, together with Fasondini, for (multivalued) variants of the Painlevé III, V, and VI equations \cite{MR3656730, MR3724928}. 

Now, extensive numerical experiments with the pole field solver applied to $v_{l_n}$ (which will be documented in a separate publication) strongly hint at the property that the zeros of the exponential generating functions $f_n(z)$ in the sectors $|\!\arg z| \leq \pi/2 + \epsilon$ satisfy a uniform tameness condition as in Def.~\ref{def:nonresonant} (see also Remark~\ref{rem:nonresonant}): the zeros are neither coming too close to the positive real axis nor are they getting too large. 
Given this state of affairs, the results on the expansions of the length distribution will be subject to the following:

\subsection*{Tameness hypothesis} For any real $t_0<t_1$ and sequence of integers $l_n\to\infty$ satis\-fy\-ing~\eqref{eq:tnstar} the zeros of the induced family $f_n(z)$ of exponential generating functions \eqref{eq:egf_family} are uniformly tame (see Def.~\ref{def:nonresonant}), with parameters and implied constants only depending on $t_0$ and $t_1$.

\begin{figure}[tbp]
\includegraphics[width=0.325\textwidth]{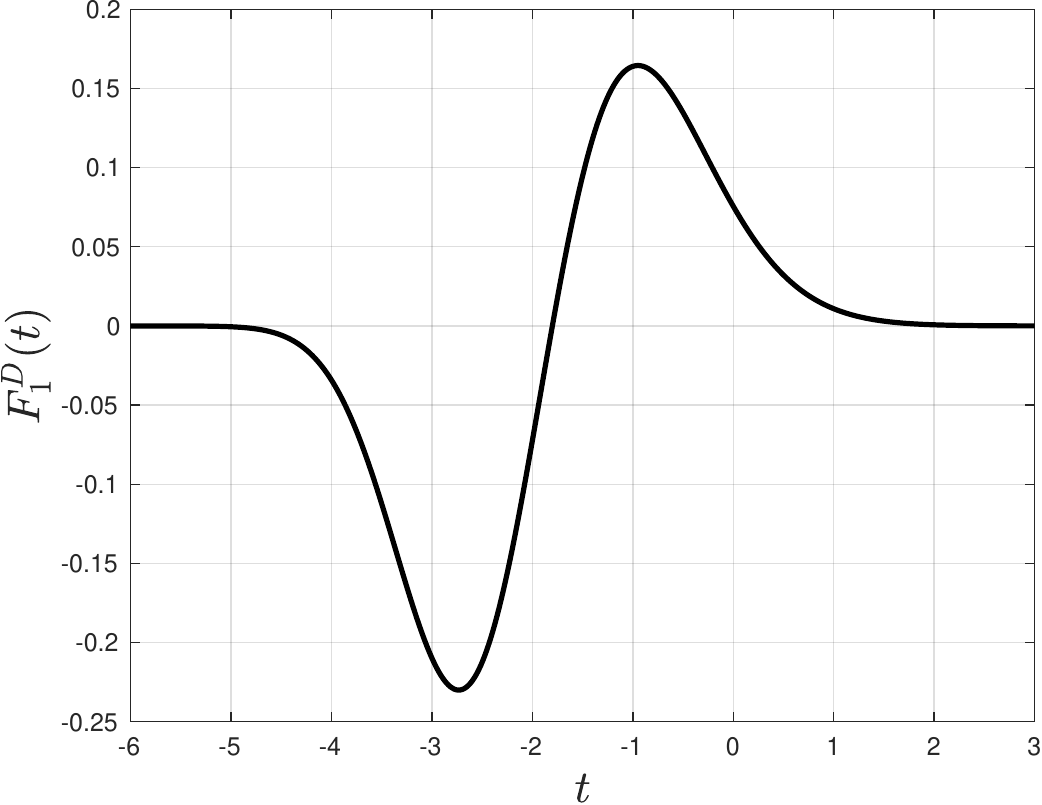}\hfil\,
\includegraphics[width=0.325\textwidth]{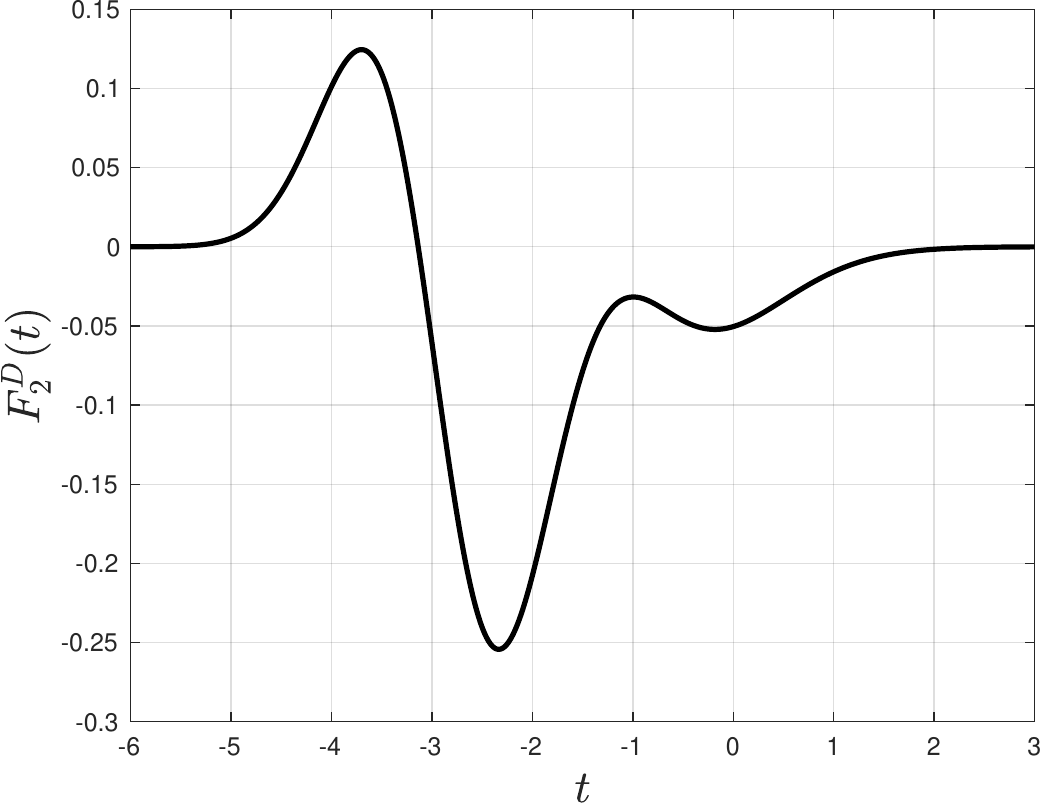}\hfil\,
\includegraphics[width=0.325\textwidth]{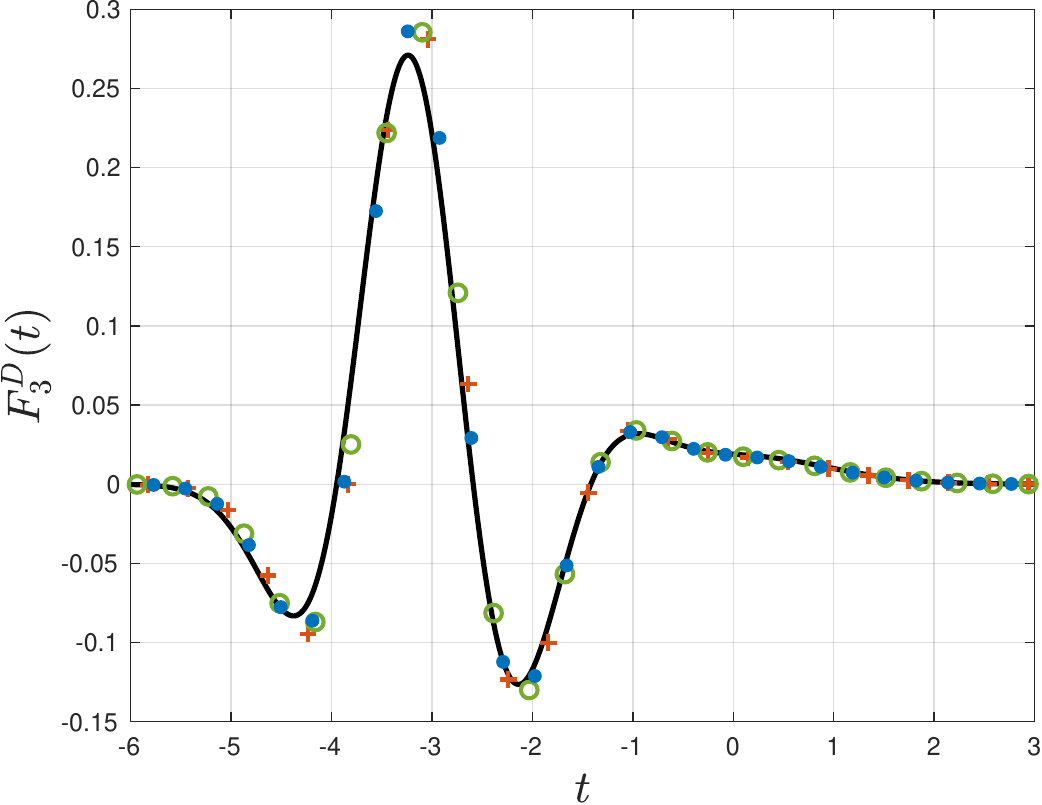}
\caption{{\footnotesize Plots of $F_{1}^D(t)$ (left panel), $F_{2}^D(t)$ (middle panel) as in \eqref{eq:F22D}; both agree with the numerical prediction of their graphical form given in the left panels of \cite[Figs.~4/6]{arxiv.2206.09411}.
The right panel shows $F_3^D(t)$ as in \eqref{eq:F22D} (black solid line) with the approximations \eqref{eq:F23D} for $n=250$ (red $+$), $n=500$ (green $\circ$) and  $n=1000$ (blue $\bullet$); the integer $l$ has been varied such that $t_l(n)$ spreads over the range of $t$ displayed here. Evaluation of~\eqref{eq:F23D} uses the table of exact values of $\prob(L_n\leq l)$ up to $n=1000$ that was compiled in \cite{arxiv.2206.09411}.}}
\label{fig:F23D}
\end{figure}

\begin{theorem}\label{thm:lengthdistexpan} Let $t_0<t_1$ be any real numbers and
assume the tameness hypothesis. 
Then there holds the expansion
\begin{equation}\label{eq:distExpand}
\prob(L_n \leq l) = F(t) + \sum_{j=1}^m F_{j}^D(t)\, n^{-j/3} + O(n^{-(m+1)/3})\,\bigg|_{t=t_l(n)},
\end{equation}
which is uniformly valid when $n,l \to\infty$ subject to $t_0 \leq t_l(n) \leq t_1$ with $m$ being any fixed non-negative integer. Here the $F_{j}^D$ are certain smooth functions{\rm;} the first three are\/\footnote{To validate the expansion \eqref{eq:distExpand} and the formulae (\ref{eq:F22D}a--c), Fig.~\ref{fig:F23D} plots $F_3^D(t)$ next to the approximation
\begin{equation}\label{eq:F23D}
F_{3}^D\big(t_l(n)\big) \approx n \cdot \Big( \prob(L_n \leq l) - F(t) - F_{1}^D(t)n^{-1/3} - F_{2}^D(t) n^{-2/3}\Big)\,\bigg|_{t=t_l(n)}
\end{equation}
for $n=250$, $n=500$ and $n=1000$, varying the integer $l$ in such a way that $t=t_l(n)$ spreads over $[-6,3]$.}
\begin{subequations}\label{eq:F22D}
\begin{align}
F_{1}^D(t) &= -\frac{t^2}{60} F'(t) - \frac{3}{5} F''(t),\label{eq:F21D}\\*[1mm]
F_{2}^D(t) &= \Big(-\frac{139}{350} + \frac{2t^3}{1575}\Big) F'(t) + \Big(-\frac{43t}{350} + \frac{t^4}{7200}\Big) F''(t) + \frac{t^2}{100} F'''(t) + \frac{9}{50} F^{(4)}(t),\\*[2mm]
F_3^D(t) &= 
-\Big(\frac{562 t }{7875} +\frac{41t^4}{283500}\Big) F'(t)
+\Big(\frac{t^2}{300}-\frac{t^5}{47250}\Big) F''(t)\\*[1mm]
&\qquad+\Big(\frac{5137}{15750}+\frac{9 t^3}{7000}-\frac{t^6}{1296000}\Big) F'''(t)
+\Big(\frac{129 t}{1750}-\frac{t^4}{12000}\Big) F^{(4)}(t)\notag\\*[1mm]
&\qquad-\frac{3 t^2}{1000}F^{(5)}(t)
-\frac{9}{250}F^{(6)}(t).\notag
\end{align}
\end{subequations}
\end{theorem}

\begin{proof} Following up the preparations preceding the formulation of the theorem, the tameness hypothesis allows us to apply Corollary~\ref{cor:nonresonant}, bounding $f_n(z)=e^z P_n(z)$ by
\begin{equation}\label{eq:cases}
\big| f_n(re^{i\theta}) \big| \leq 
\begin{cases}
2f_n(r) e^{-\frac{1}{2}\theta^2 r}, &\qquad 0\leq |\theta| \leq r^{-2/5},\\*[2mm]
2f_n(r) e^{-\frac{1}{2}r^{1/5}}, &\qquad r^{-2/5}\leq |\theta| \leq \pi,
\end{cases}
\end{equation}
for $n - n^{3/5} \leq r \leq n + n^{3/5}$ and $n\geq n_0$ when $n_0$ is sufficiently large (depending on $t_0$, $t_1$). Using the trivial bounds (for $r>0$ and $|\theta|\leq \pi$) 
\[
0\leq f_n(r)\leq e^r, \qquad 0\leq  P_n(r) \leq 1, \qquad 1 - \tfrac{1}{2}\theta^2 \leq \cos\theta,
\]
the first case in \eqref{eq:cases} can be recast in form of the bound
\[
\big|P_n(re^{i\theta})\big| \leq 2 P_n(r) e^{r \left(1-\cos\theta - \tfrac12 \theta^2\right)} \leq 2,
\]
which proves condition (I) of Thm.~\ref{thm:jasz} with $B=2$, $D=1$, $\beta=0$ and $\delta=2/5$; whereas the second
case implies 
\[
|f_n(n e^{i\theta})| \leq 2 f_n(n) e^{-\tfrac12 n^{1/5}} \leq 2 \exp\big(n - \tfrac12 n^{1/5}\big),
\]
which proves condition (O) of Thm.~\ref{thm:jasz} with $A=1/2$, $C=0$, $\alpha=1/5$ and $\gamma=0$.
Hence, there holds the Jasz expansion \eqref{eq:jasz}, namely
\[
\prob(L_n \leq l_n) = P_n(n) + \sum_{j=2}^M b_j(n) P_n^{(j)}(n) + O(n^{-(M+1)/5})
\]
for any $M=0,1,2,\ldots$ as $n\geq n_1$; here $n_1$ and the implied constant depend on $t_0$, $t_1$.
By noting that the diagonal Poisson--Charlier polynomials $b_j$ have degree $\leq \lfloor j/2\rfloor$ and by choosing $M$ large enough, the expansions \eqref{eq:Pnnexpand} of $P_n^{(j)}(n)$ in terms of powers of $n^{-1/3}$  yield that there are smooth functions $F_{j}^D$ such that
\[
\prob(L_n \leq l_n) = F(t) + \sum_{j=1}^m F_{j}^D(t)\, n^{-j/3} + O(n^{-(m+1)/3})\,\bigg|_{t=t_n^*}
\]
as $n\to\infty$; $m$ being any fixed non-negative integer. Given the uniformity of the bound for fixed $t_0$ and $t_1$, we can replace $l_n$ by $l$ and $t_n^*$ by $t_l(n)$ as long as we respect $t_0 \leq t_l(n) \leq t_1$. This finishes the proof of \eqref{eq:distExpand}.

The first three functions $F_{1}^D$, $F_{2}^D(t)$, $F_{3}^D(t)$ can be determined using the particular case~\eqref{eq:jaszP4} of the Jasz expansion from Example~\ref{ex:jaszP4} (which applies here because of \eqref{eq:Pnnlead}), namely
\[
\prob(L_n \leq l_n) =P_n(n) - \frac{n}{2} P_n''(n) + \frac{n}{3} P_n'''(n) + \frac{n^2}{8} P_n^{(4)}(n) - \frac{n^3}{48}P_n^{(6)}(n) + O(n^{-4/3}).
\]
Inserting the formulae displayed in \eqref{eq:Pnnexpand} we thus obtain
\begin{subequations}\label{eq:F2FPrel}
\begin{align}
F_{1}^D(t) &= F_{1}^P(t) - \frac{1}{2} F''(t),\\*[2mm]
F_{2}^D(t) &= F_{2}^P(t) - \frac{1}{2} F_{1}^{P}\!\!\phantom{|}''(t) - \frac{5}{12}F'(t) - \frac{t}{6}F''(t) + \frac{1}{8} F^{(4)}(t),\\*[2mm]
F_3^D(t) &= F_3^P(t)-\frac{1}{2} F_{2}^{P}\!\!\phantom{|}''(t)-\frac{3 }{4} F_{1}^{P}\!\!\phantom{|}'(t)-\frac{t}{6} F_{1}^{P}\!\!\phantom{|}''(t)+\frac{1}{8} F_{1}^{P}\!\!\phantom{|}^{(4)}(t)\\*[1mm]
&\qquad -\frac{7t}{72} F'(t)-\frac{t^2}{72}  F''(t)+\frac{7}{24} F^{(3)}(t)+\frac{t}{12} F^{(4)}(t)-\frac{1}{48} F^{(6)}(t).\notag
\end{align}
\end{subequations}
Together with the expressions given in \eqref{eq:F22P} this yields the functional form asserted in \eqref{eq:F22D}.
\end{proof}

\begin{figure}[tbp]
\includegraphics[width=0.325\textwidth]{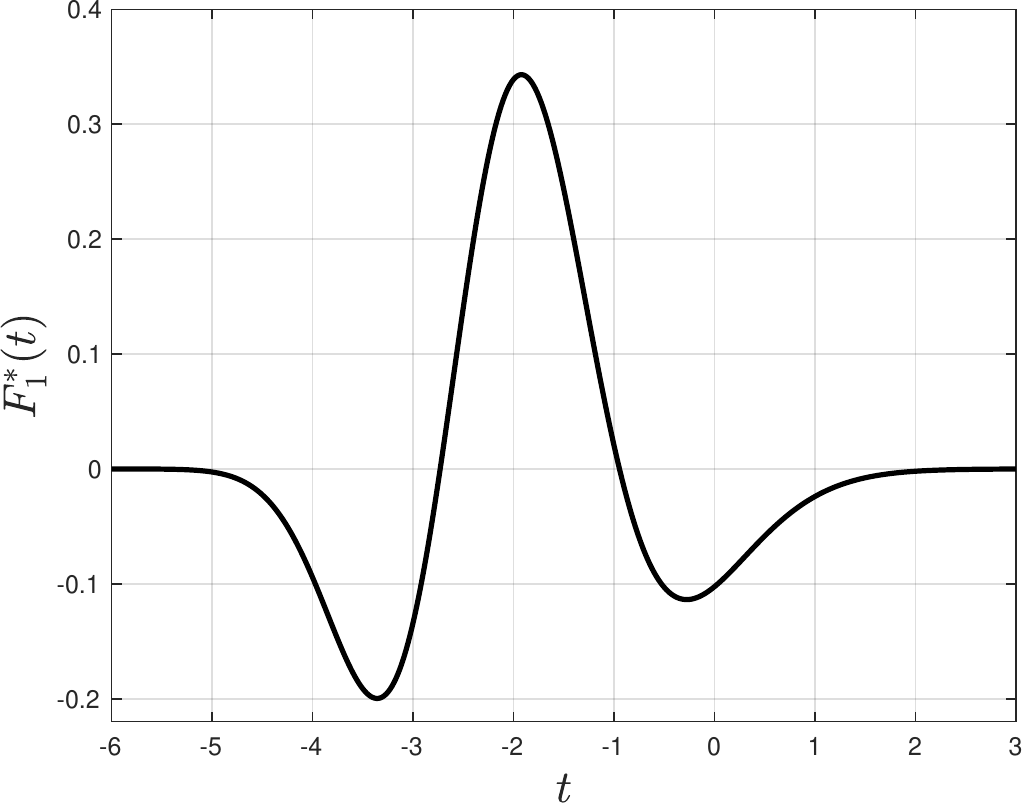}\hfil\,
\includegraphics[width=0.325\textwidth]{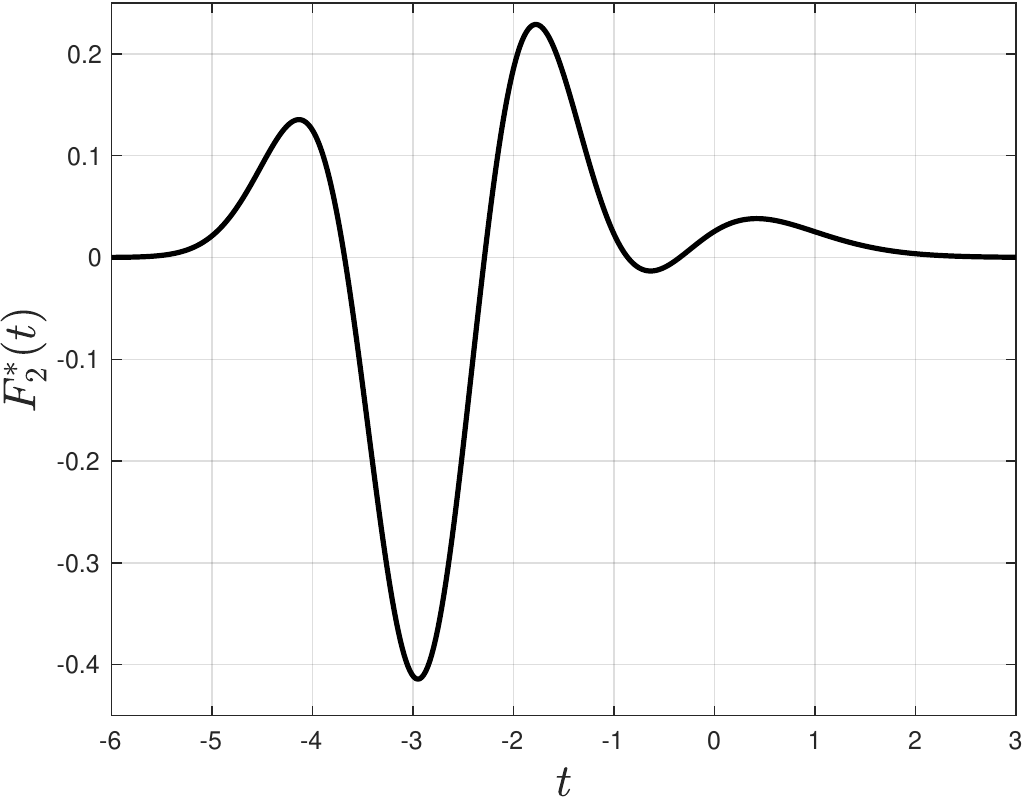}\hfil\,
\includegraphics[width=0.325\textwidth]{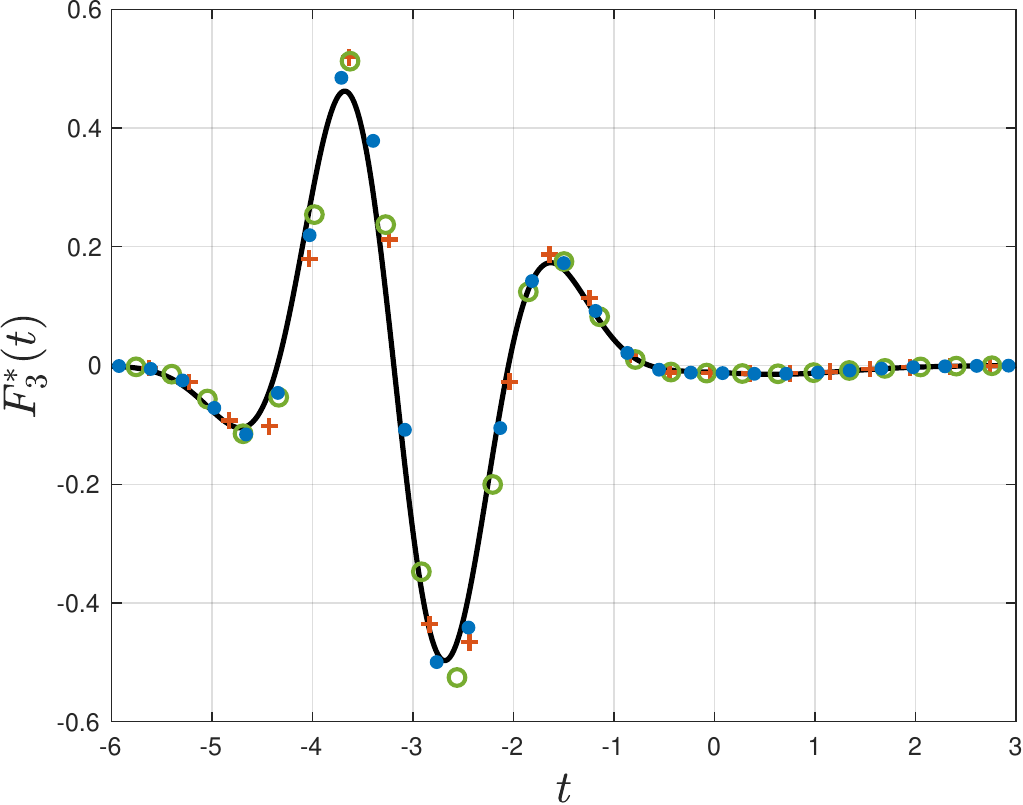}
\caption{{\footnotesize Plots of $F_{1}^*(t)$ (left panel) and $F_{2}^*(t)$ (middle panel) as in (\ref{eq:F22star}a/b); both agree with the numerical prediction of their graphical form given in the right panels of \cite[Figs.~4/6]{arxiv.2206.09411}.
The right panel shows $F_3^*$ as in (\ref{eq:F22star}c) (black solid line) with the approximations \eqref{eq:F23star} for $n=250$ (red $+$), $n=500$ (green $\circ$) and  $n=1000$ (blue $\bullet$); the integer $l$ has been varied such that $t_{l-1/2}(n)$ spreads over the range of $t$ displayed here. Evaluation of~\eqref{eq:F23star} uses the table of exact values of $\prob(L_n= l)$ up to $n=1000$ that was compiled in \cite{arxiv.2206.09411}.}}
\label{fig:F23star}
\end{figure}

\subsection{Expansion of the PDF}

Subject to its assumptions, Thm.~\ref{thm:lengthdistexpan} implies for the PDF of the length distribution that
\begin{multline*}
\prob(L_n=l) = \prob(L_n \leq l) - \prob(L_n \leq l-1) \\*[1mm]
= \big(F(t_l(n)) - F(t_{l-1}(n))\big) + \sum_{j=1}^m \big(F_{j}^D(t_l(n))-F_{j}^D(t_{l-1}(n))\big)\, n^{-j/3} + O(n^{-(m+1)/3}).
\end{multline*}
Applying the central differencing formula (which is, basically, just a Taylor expansion for smooth $G$ centered at the midpoint)
\[
G(t+h) - G(t) = h G'\big(t+\tfrac{h}{2} \big) + \frac{h^3}{24}G'''\big(t+\tfrac{h}{2} \big) +  \frac{h^5}{1920}G^{(5)} \big(t+\tfrac{h}{2} \big)  +\frac{h^7 }{322560} G^{(7)} \big(t+\tfrac{h}{2} \big)+ \cdots ,
\]
with increment $h=n^{-1/6}$, we immediately get the following corollary of Thm.~\ref{thm:lengthdistexpan}.

\begin{corollary}\label{cor:PDFexpan} Let $t_0<t_1$ be any real numbers and
assume the tameness hypothesis. Then there holds the expansion
\begin{equation}\label{eq:PDFexpan}
n^{1/6}\, \prob(L_n=l) =  F' (t) + \sum_{j=1}^m F_{j}^*(t) n^{-j/3} + O(n^{-(m+1)/3})\,\bigg|_{t=t_{l-1/2}(n)},
\end{equation}
which is uniformly valid when $n,l \to\infty$ subject to the constraint $t_0 \leq t_{l-1/2}(n) \leq t_1$ with $m$ being any fixed non-negative integer. Here the $F_{j}^*$ are certain smooth functions{\rm;} in particular\/\footnote{To validate the expansion \eqref{eq:PDFexpan} and the formulae (\ref{eq:F22star}a--c), Fig.~\ref{fig:F23star} plots $F_3^*(t)$ next to the approximation
\begin{equation}\label{eq:F23star}
F_{3}^*\big(t_{l-1/2}(n)\big) \approx n^{7/6} \cdot \Big( \prob(L_n = l) -  F'(t) n^{-1/6} - F_{1}^*(t)n^{-1/2} - F_{2}^*(t) n^{-5/6}\Big) \,\bigg|_{t=t_{l-1/2}(n)}
\end{equation}
for $n=250$, $n=500$ and $n=1000$, varying the integer $l$ in such a way that $t=t_{l-1/2}(n)$ spreads over $[-6,3]$.}%
\begin{subequations}\label{eq:F22star}
\begin{align}
F_{1}^*(t) &= -\frac{t}{30} F'(t) - \frac{t^2}{60}F''(t) - \frac{67}{120} F'''(t), \label{eq:F21star} \\*[2mm]
F_{2}^*(t) &= \frac{2t^2}{525}F'(t) + \Big(-\frac{629}{1200}+\frac{23t^3}{12600}\Big)F''(t) + \Big(-\frac{899t}{8400}+\frac{t^4}{7200}\Big)F'''(t) \\*[1mm]
&\qquad\quad + \frac{67t^2}{7200}F^{(4)}(t) + \frac{1493}{9600}F^{(5)}(t), \notag
\intertext{}\notag\\*[-14mm]
F_3^*(t) &=
-\Big(\frac{373}{5250} + \frac{41 t^3}{70875} \Big) F'(t)
-\Big(\frac{1781 t}{28000} + \frac{71 t^4 }{283500}\Big) F''(t)\\*[1mm]
&\qquad+\Big(\frac{63 t^2}{8000}-\frac{13 t^5}{504000}\Big) F'''(t) 
+\Big(\frac{41473 }{112000}+\frac{13 t^3}{12096} -\frac{t^6}{1296000} \Big) F^{(4)}(t)\notag\\*[1mm]
&\qquad+\Big(\frac{131057 t}{2016000} - \frac{67 t^4}{864000}\Big) F^{(5)}(t)
-\frac{1493 t^2}{576000} F^{(6)}(t)
-\frac{232319}{8064000}F^{(7)}(t).\notag
\end{align}
\end{subequations}
\end{corollary}

\begin{figure}[tbp]
\includegraphics[width=0.45\textwidth]{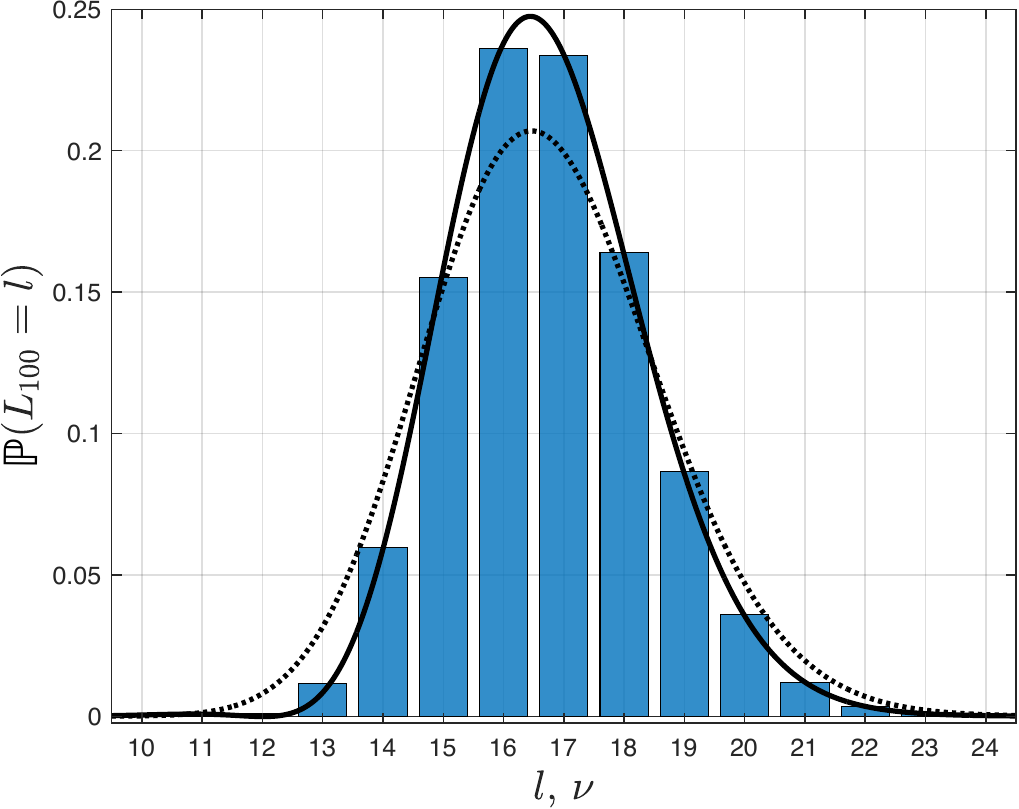}\hfil\,
\includegraphics[width=0.45\textwidth]{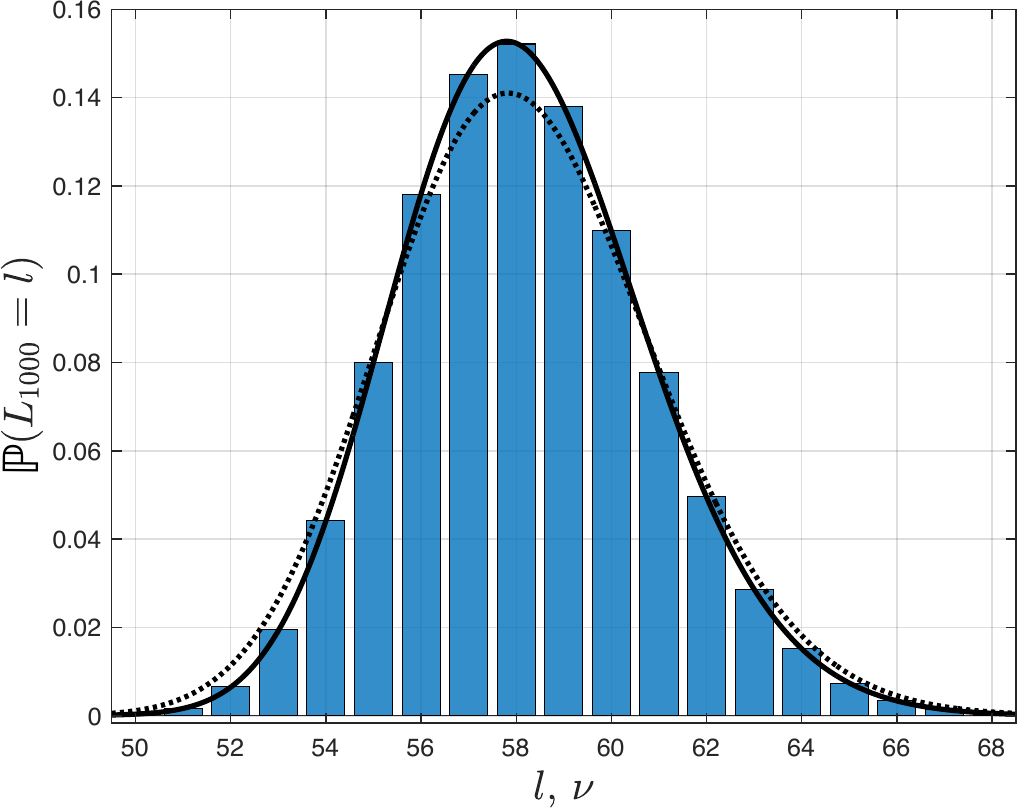}
\caption{{\footnotesize The exact discrete length distribution $\prob(L_{n}=l)$ (blue bars centered at the integers $l$) vs. the asymptotic expansion \eqref{eq:PDFexpan} for $m=0$ (the Baik--Deift--Johansson limit, dotted line) and for $m=2$ (the limit with the first two finite-size correction terms added, solid line). Left: $n=100$; right: $n=1000$. The expansions are displayed as functions of the continuous variable $\nu$, evaluating the right-hand-side of \eqref{eq:PDFexpan} in $t=t_{\nu-1/2}(n)$. The exact values are from the table compiled in \cite{arxiv.2206.09411}. Note that a graphically accurate continuous approximation of the discrete distribution must intersect the bars right in the middle of their top sides: this is, indeed, the case for $m=2$ (except at the left tail for $n=100$). In contrast, the uncorrected limit law ($m=0$) is noticeable inaccurate for this range of $n$.}}
\label{fig:PDF}
\end{figure}

\begin{remark}
The case $m=0$ of Corollary~\ref{cor:PDFexpan} gives
\[
n^{1/6} \, \prob(L_n=l) = F'(t_{l-1/2}(n)) + O(n^{-1/3}),
\]
where the exponent in the error term cannot be improved. By noting 
\[
t_{l-1/2}(n) = t_l(n) - \tfrac{1}{2}n^{-1/6}
\]
we understand that, for fixed large $n$, visualizing the discrete length distribution near its mode by plotting the points 
\[
\big(t_l(n),n^{1/6}\,\prob(L_n=l)\big)
\] 
next to the graph $(t,F'(t))$ introduces a perceivable bias: namely, all points are shifted by an amount of $n^{-1/6}/2$ to the right of the graph. 
Exactly such a bias can be observed in the first ever published plot of the PDF vs. the density of the Tracy--Widom distribution by Odlyzko and Rains in \cite[Fig.~1]{MR1771285}: the Monte-Carlo data for $n=10^6$  display  a consistent shift by $0.05$.

A bias free plot is shown in Fig.~\ref{fig:PDF}, which in addition displays an improvement of the error of the limit law by a factor of $O(n^{-2/3})$ that is obtained by adding the first two finite-size correction terms.
\end{remark}

\section{Expansions of Stirling-Type Formulae}\label{sect:stirling}

In our work \cite{arxiv.2206.09411} we advocated the use of a Stirling-type formula to approximate the length distribution for larger $n$ (because of being much more efficient and accurate than Monte-Carlo simulations). To recall some of our findings there, let us denote the exponential generating function  and its Poisson counterpart simply by
\[
f(z) = \sum_{n=0}^\infty \prob(L_n\leq l) \frac{z^n}{n!},\qquad P(z) = e^{-z} f(z), 
\] 
suppressing the dependence on the integer parameter $l$ from the notation for the sake of brevity. It was shown in \cite[Thm.~2.2]{arxiv.2206.09411} that the entire function $f$ is $H$-admissible so that there is the normal approximation
(see Def.~\ref{def:Hayman} and Thm.~\ref{thm:hayman})
\begin{equation}\label{eq:normal}
 \prob(L_n\leq l) = \frac{n! f(r)}{r^n \sqrt{2\pi b(r)}} \left(\exp\left(-\frac{(n-a(r))^2}{2b(r)}\right)+ o(1)\right) \qquad (r\to \infty)
 \end{equation}
uniformly in $n=0,1,2,\ldots$ while $l$ is any {\em fixed} integer. Here, $a(r)$ and $b(r)$ are the real auxiliary functions
\[
a(r) = r \frac{f'(r)}{f(r)},\qquad b(r) = r a'(r).
\]
We consider the two cases $r=r_n$, $a(r_n)=n$ and $r=n$. After dividing \eqref{eq:normal} by the 
classical Stirling factor (which does not change anything of substance; see Remark~\ref{rem:vanillaStirling})
\begin{equation}\label{eq:tau}
\tau_n := \frac{n!}{\sqrt{2\pi n}} \left(\frac{e}{n}\right)^n \sim 1+\frac{n^{-1}}{12} + \frac{n^{-2}}{288} + \cdots,
\end{equation}
we get after some re-arranging of terms the {\em Stirling-type formula} $S_{n,l}$ ($r=r_n$) and the {\em simplified Stirling-type formula} $\tilde S_{n,l}$ ($r=n$):
\begin{subequations}\label{eq:Stirling}
\begin{align}
S_{n,l} &:= \frac{P(r_n)}{\sqrt{b(r_n)/n}} \exp\left(n\,\Lambda\!\left(\frac{r_n-n}n\right)\right),\quad \Lambda(h) = h - \log(1+h),\label{eq:Snl}\\*[2mm]
\tilde S_{n,l} &:= \frac{P(n)}{\sqrt{b(n)/n}} \exp\left(-\frac{(n-a(n))^2}{2b(n)}\right).\label{eq:Snltilde}
\end{align}
\end{subequations}
As shown in \cite{arxiv.2206.09411}, both approximations are amenable for a straightforward numerical evaluation using the tools developed in \cite{MR2895091, MR2600548,MR3647807}. For fixed $l$, the normal approximation \eqref{eq:normal} implies
\[
\prob(L_n\leq l) = S_{n,l} \cdot (1 + o(1))\qquad (n\to\infty);
\]
but numerical experiments reported in \cite[Fig.~3, Eq.~(8b)]{arxiv.2206.09411} suggest that there holds
\[
\prob(L_n\leq l) = S_{n,l} + O(n^{-2/3})
\]
{\em uniformly} when $n,l\to\infty$ while $t_l(n)$ stays bounded. Subject to the tameness hypothesis of Thm.~\ref{thm:lengthdistexpan} we prove this observation as well as its counterpart for the simplified Stirling-type formula, thereby unveiling the functional form of the error term $O(n^{-2/3})$:\/\footnote{Note that the expansions \eqref{eq:SnlTilde2} for $\tilde S_{n,l}$ and \eqref{eq:Snl2} for $S_{n,l}$ given in the proof do not require the tameness hypothesis. It is only required to facilitate the comparison with the result of Thm.~\ref{thm:lengthdistexpan}, which then yields \eqref{eq:StirlingExpand}.}

\begin{theorem}\label{thm:stirling}
Let $t_0<t_1$ be any real numbers and
assume the {tameness hypothesis}. Then, for the Stirling-type formula $S_{n,l}$ and its simplification $\tilde S_{n,l}$, there hold the expansions (note that both are starting at $j=2$)
\begin{subequations}\label{eq:StirlingExpand}
\begin{align}
\prob(L_n \leq l) &= S_{n,l} + \sum_{j=2}^m F_{j}^S\big(t_l(n)\big)\, n^{-j/3} + O(n^{-(m+1)/3}),\label{eq:SExpan}\\*[1mm]
\prob(L_n \leq l) &= \tilde S_{n,l} + \sum_{j=2}^m \tilde F_{j}^S\big(t_l(n)\big)\, n^{-j/3} + O(n^{-(m+1)/3}),\label{eq:StildeExpan}
\end{align}
\end{subequations}
which are uniformly valid when $n,l \to\infty$ subject to $t_0 \leq t_l(n) \leq t_1$ with $m$ being any fixed non-negative integer. Here the $F_{j}^S$ and $\tilde F_{j}^S$ are certain smooth functions{\rm;} the first being\/\footnote{The functional form of the terms $F_{2}^S$, $\tilde F_{2}^S$ differs significantly from the one of corresponding terms in the previous theorems. Though they still share the form  
\[
F(t) \cdot \big(\text{rational polynomial in $u_{00}(t), u_{10}(t), u_{11}(t), u_{20}(t), u_{21}(t), u_{30}(t)$}\big),
\]
using the algorithmic ideas underlying the tabulation  of $F(t)\cdot u_{jk}(t)$ $(0\leq j+k \leq 8$) in \cite[p.~68]{MR2787973} one can show that $F_2^S(t)$ and $\tilde F_2^{S}$ do not simplify to the form \eqref{eq:ShinaultForm} of a linear combination of derivatives of $F$ with (rational) polynomial coefficients (at least not for orders up to $50$, cf. the discussion in Appendix~\ref{app:ST}).}
\begin{subequations}\label{eq:F22S}
\begin{align}
F_{2}^S(t) &= -\frac{3}{4}\frac{F'(t)^4}{F(t)^3} + \frac{3}{2}\frac{F'(t)^2F''(t)}{F(t)^2}- \frac{3}{8}\frac{F''(t)^2}{F(t)} -  \frac{1}{2}\frac{F'(t) F'''(t)}{F(t)} + \frac{1}{8} F^{(4)}_2(t) ,\label{eq:F22Splain} \\*[2mm]
\tilde F_{2}^S(t) &= -\frac{1}{2}F'(t) + \frac{1}{4}\frac{F'(t)^4}{F(t)^3}- \frac{3}{8}\frac{F''(t)^2}{F(t)}  + \frac{1}{8} F^{(4)}_2(t).\label{eq:F22Stilde} 
\end{align}
\end{subequations}
The solution $r_n$ of the equation $a(r_n)=n$, required to evaluate $S_{n,l}$, satisfies the expansion\/\footnote{This provides excellent initial guesses for solving $a(r_n)=n$ by iteration; cf.~\cite[Sect.~3.4]{arxiv.2206.09411}. It also helps to understand the quantitative observations made in \cite[Example~12.5]{MR2754188}.}
\[
r_n = n + \frac{F'\big(t_l(n) \big)}{F \big(t_l(n) \big)} n^{1/3} + O(1),
\]
which is uniformly valid under the same conditions.
\end{theorem}

\begin{proof} We restrict ourselves to the case $m=2$, focussing on the concrete functional form of the expansion terms; nevertheless the general form of the expansions \eqref{eq:StirlingExpand} should become clear along the way.

\medskip
{\em Preparatory steps.}
Because of $P(r) = E_2^\text{hard}(4r;l)$ (using the notation preceding Thm.~\ref{thm:stirling}), Thm.~\ref{thm:poissonExpandnu} gives that
\begin{equation}\label{eq:PrPlain}
P(r) = \left.F (t) + F_{1}^P (t) \cdot r^{-1/3} + F_{2}^P (t) \cdot r^{-2/3} + O(r^{-1})\,\right|_{t=t_l(r)},
\end{equation}
which is uniformly valid when $r,l \to\infty$ subject to the constraint $t_0 \leq t_l(r) \leq t_1$ (the same constraint applies to the expansions to follow). Preserving uniformity, the expansion can be repeatedly differentiated w.r.t. the variable $r$, which yields by using the differential equation \eqref{eq:tnuprime} satisfied by $t_l(r)$ (cf. also \eqref{eq:Pnnlead})
\begin{equation}\label{eq:PrPrime}
P'(r) = \left. -F'(t) r^{-2/3} - \Big(F_{1}^P\!\!\phantom{|}'(t) + \frac{t}{6} F'(t)\Big)r^{-1} + O(r^{-4/3})\, \right|_{t=t_l(r)}.
\end{equation}
Recalling $f(r) = e^r P(r)$, we thus get 
\begin{subequations}\label{eq:arbr}
\begin{equation}
a(r) = r + r \frac{P'(r)}{P(r)} = \left. r + a_1(t) r^{1/3} + a_2(t) + O(r^{-1/3})\, \right|_{t=t_l(r)}
\end{equation}
with the coefficient functions
\begin{equation}
a_1(t) = -\frac{F'(t)}{F(t)},\qquad a_2(t) = -\frac{1}{F(t)}\Big(F_{1}^P\!\!\phantom{|}'(t) + \frac{t}{6} F'(t) \Big) + \frac{F_{1}^P(t)F'(t)}{F(t)^2}; 
\end{equation}
a further differentiation yields
\begin{equation}
b(r) = r a'(r) = \left. r- a_1'(t) r^{2/3} + \Big(\frac{1}{3} a_1(t) - \frac{t}{6}a_1'(t) - a_2'(t) \Big) r^{1/3} + O(1)\, \right|_{t=t_l(r)}.
\end{equation}
\end{subequations}

{\em The simplified Stirling-type formula.} Here we have $r=n$ and we write $t_* := t_l(n)$ to be brief. By inserting the expansions~\eqref{eq:PrPlain} and \eqref{eq:arbr} into the expression \eqref{eq:Snltilde}, we obtain after a routine calculation with truncated power series and collecting terms as in \eqref{eq:F2FPrel} that
\begin{equation}\label{eq:SnlTilde2}
\tilde S_{n,l} = F(t_*) + F_{1}^D(t_*) n^{-1/3} + \Big(F_{2}^D(t_*) - \tilde F_{2}^S(t_*) \Big) n^{-2/3} + O(n^{-1}),
\end{equation}
where the remaining $\tilde F_{2}^S(t)$ is given by \eqref{eq:F22Stilde}; a subtraction from \eqref{eq:distExpand} yields \eqref{eq:StildeExpan}.

\medskip

{\em The Stirling-type formula.} Here we have $r=r_n$ and we have to distinguish between $t_*$ and 
\[
t_l(r) = t_* \cdot (n/r)^{1/6} + 2 \frac{\sqrt{n}-\sqrt{r}}{r^{1/6}}.
\]
By inserting the expansion 
\begin{subequations}\label{eq:areqn}
\begin{equation}
r_n = n + r_1(t_*) n^{1/3} + r_2(t_*) + O(n^{-1/3})
\end{equation}
into $t_n(r_n)$ and $a(r_n)$ we obtain
\begin{align}
t_l(r_n) &= t_* - r_1(t_*) n^{-1/3} - \Big(r_2(t_*) + \frac{t_*}{6} r_1(t_*)\Big) n^{-2/3} + O(n^{-1})\\*[1mm]
a(r_n) &= n +(a_1(t_*) + r_1(t_*)) n^{1/3} + \Big(a_2(t_*) + r_2(t_*) -r_1(t_*) a_1'(t_*) \Big) + O(n^{-1/3}).
\end{align}
Thus the solution of $a(r_n) = n$, which by Thm.~\ref{thm:hayman} is unique, leads to the relations
\begin{equation}
r_1(t) = -a_1(t),\qquad r_2(t) = - a_2(t) -a_1(t)a_1'(t).
\end{equation}
\end{subequations}
By inserting, first, the expansions~\eqref{eq:areqn} into the expansions~\eqref{eq:PrPlain} and \eqref{eq:arbr} for the particular choice $r=r_n$ and, next, the thus obtained results into the expression \eqref{eq:Snl}, we obtain after a routine calculation with truncated power series and collecting terms as in \eqref{eq:F2FPrel} that
\begin{equation}\label{eq:Snl2}
S_{n,l} = F(t_*) + F_{1}^D(t_*)  n^{-1/3} + \Big(F_{2}^D(t_*)  - F_{2}^S(t_*) \Big) n^{-2/3} + O(n^{-1}),
\end{equation}
where the remaining $F_{2}^S(t)$ is given by \eqref{eq:F22Splain}; a subtraction from \eqref{eq:distExpand} yields \eqref{eq:SExpan}.
\end{proof}

\begin{figure}[tbp]
\includegraphics[width=0.325\textwidth]{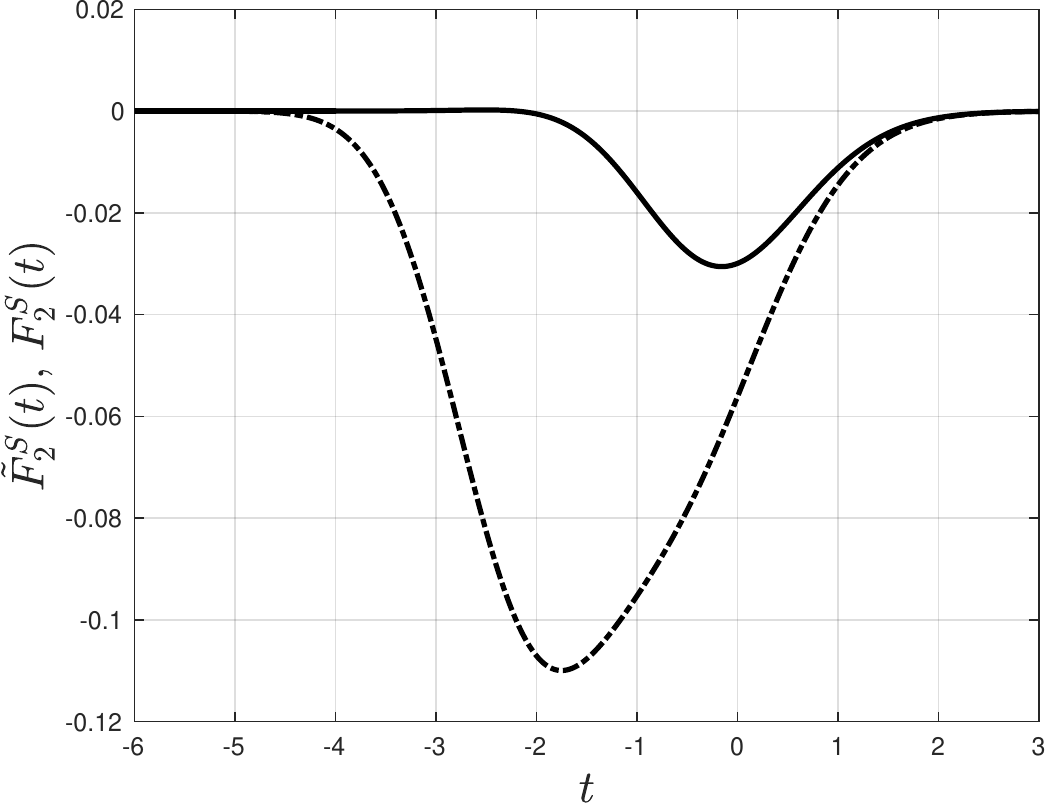}\hfil\,
\includegraphics[width=0.325\textwidth]{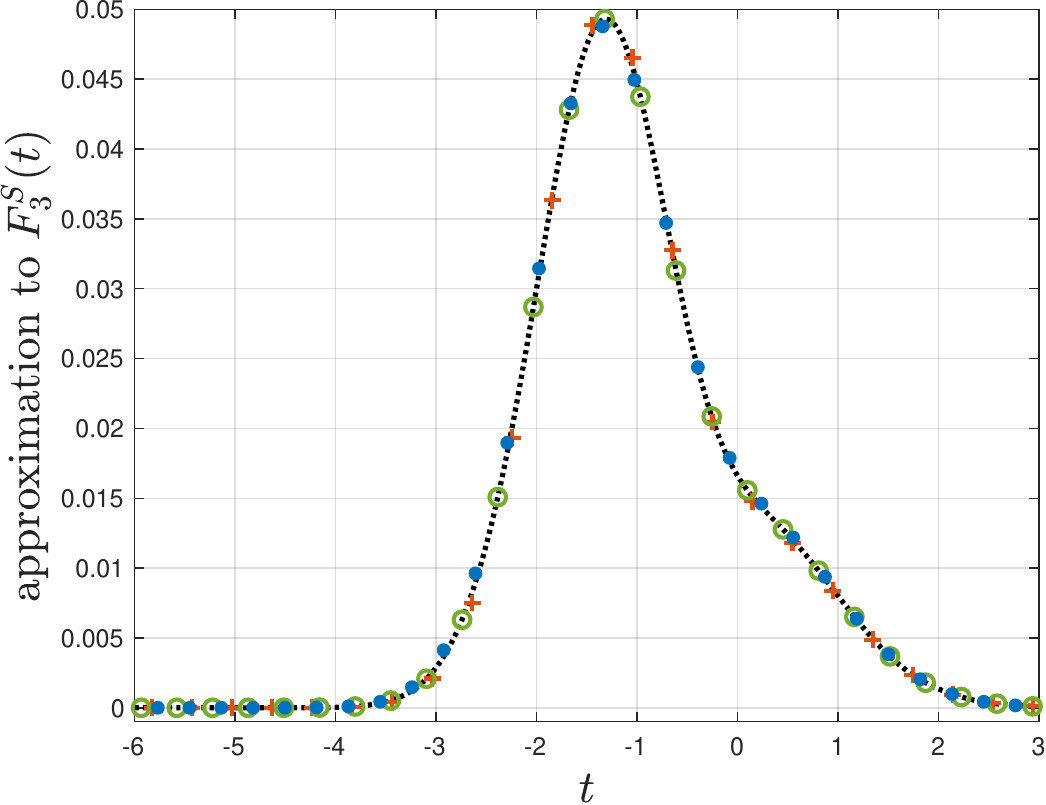}\hfil\,
\includegraphics[width=0.325\textwidth]{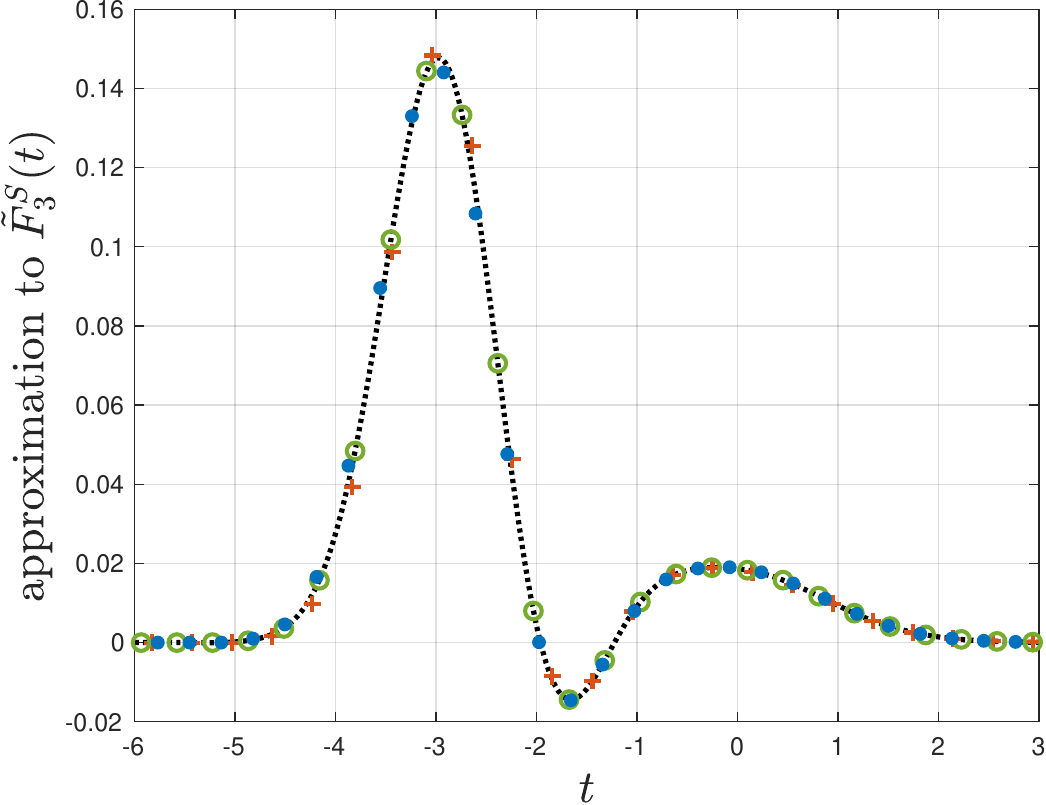}
\caption{{\footnotesize Left panel: plots of $\tilde F_{2}^S(t)$ (solid line) and $\tilde F_{2}^S(t)$ (dash-dotted line) as in \eqref{eq:F22S}. 
The middle and right panel show the approximations of $F_{3}^S(t)$ and $\tilde F_{3}^S(t)$ in \eqref{eq:F23S} for $n=250$ (red $+$), $n=500$ (green $\circ$) and  $n=1000$ (blue $\bullet$); the integer $l$ has been varied such that $t_l(n)$ spreads over the range of the variable $t$ displayed here; the dotted line displays a polynomial fit to the data points of degree $30$ to help visualizing their joint graphical form. Evaluation of \eqref{eq:F23S} uses the table of exact values of $\prob(L_n\leq l)$ up to $n=1000$ that was compiled in \cite{arxiv.2206.09411}.}}
\label{fig:F23S}
\end{figure}

To validate the expansions \eqref{eq:StirlingExpand} and the formulae (\ref{eq:F22S}a/b), Fig.~\ref{fig:F23S} plots the approximations
\begin{subequations}\label{eq:F23S}
\begin{align}
F_{3}^S\big(t_l(n)\big) &\approx n^{-1} \Big( \prob(L_n \leq l) - S_{n,l} - F_{2}^S\big(t_l(n)\big) n^{-2/3}\Big),\\*[1mm]
\tilde F_{3}^S\big(t_l(n)\big) &\approx n^{-1} \Big( \prob(L_n \leq l) - \tilde S_{n,l} - \tilde F_{2}^S\big(t_l(n)\big) n^{-2/3}\Big)
\end{align}
\end{subequations}
for $n=250$, $n=500$ and $n=1000$, varying the integer $l$ in such a way that $t=t_l(n)$ spreads over $[-6,3]$. The plot suggests the following observations:

\medskip

\begin{itemize}\itemsep=4pt
\item Apparently there holds $\tilde F_{2}^S(t) < F_{2}^S(t) < 0$ for $t\in[-6,3]$, which if generally true would imply
\[
\prob(L_n \leq l) < S_{n,l} < \tilde S_{n,l}
\] 
for $n$ being sufficiently large and $l$ near the mode of the distribution.
This one-sided approximation of the length distribution by the Stirling formula $S_{n,l}$ from above is also clearly visible in \cite[Tables~1/2]{arxiv.2206.09411}.
\item Comparing $F_{2}^S(t)$ in Fig.~\ref{fig:F23S} to  $F_{2}^D(t)$ in Fig.~\ref{fig:F23D} shows that the maximum error
\[
\qquad\quad \max_{l=1,\ldots,n} \left| \prob(L_n \leq l) - \big(F(t) + F_{1}^D(t) n^{-1/3}\big)\big|_{t=t_l(n)}\right| \approx n^{-2/3} \|F_{2}^D\|_\infty \approx 0.25n^{-2/3}
\]
of approximating the length distribution by the first finite-size correction in Thm.~\ref{thm:lengthdistexpan} is about an order of magnitude larger than the maximum error of the Stirling-type formula,
\[
\max_{l=1,\ldots,n} \left| \prob(L_n \leq l) - S_{n,l}\right| \approx n^{-2/3} \|F_{2}^S\|_\infty \approx 0.031n^{-2/3}.
\]
This property of the Stirling-type formula was already observed in \cite[Fig.~3]{arxiv.2206.09411} and was used there to approximate the graphical form of $F_{2}^D(t)$ (see \cite[Fig.~6]{arxiv.2206.09411}).
\end{itemize}

\begin{remark}\label{rem:vanillaStirling}
If one includes the classical Stirling factor \eqref{eq:tau} into the Stirling-type formula by replacing \eqref{eq:Snl} with the unmodified normal approximation \eqref{eq:normal}, that is, with
\[
S_{n,l}^* := \tau_n\frac{P(r_n)}{\sqrt{b(r_n)/n}} \exp\left(n\,\Lambda\!\left(\frac{r_n-n}n\right)\right)  = \tau_n S_{n,l},
\]
Thm.~\ref{thm:stirling} would remain valid: in fact, multiplication of \eqref{eq:SExpan} by the expansion \eqref{eq:tau} of $\tau_n$ in powers of $n^{-1}$ gives, by taking \eqref{eq:distExpand} into account,
\[
\prob(L_n \leq l) = S^*_{n,l} + \sum_{j=2}^m F_{j}^{S^*}\big(t_l(n)\big)\, n^{-j/3} + O(n^{-(m+1)/3}),
\]
where the first two coefficient functions are
\[
F_{2}^{S^*}(t)= F_{2}^{S}(t), \quad F_{3}^{S^*}(t)= F_{3}^{S}(t) - \frac{1}{12} F(t), \quad\ldots \;.
\]
Because of $\lim_{t\to+\infty}F(t) = 1$, we would loose the decay of $F_{3}^S(t)$ for large $t$, leaving us with a non-zero residual value coming from the classical Stirling factor. For this reason, we recommend dropping the factor $\tau_n$, thereby resolving an ambiguity expressed in \cite[Fn.~28]{arxiv.2206.09411}.
\end{remark}

\section{Expansions of Expected Value and Variance}\label{sect:Ulam}

Lifting the expansion \eqref{eq:PDFexpan} of the PDF of the length distribution to one of the expected value and variance requires a control of the tails (of the distribution itself and of the expansion terms) which, at least right now, we can only conjecture to hold true. 

To get to a reasonable conjecture, we recall the tail estimates for the discrete distribution (see \cite[Eqn.~(9.6/9.12)]{MR1682248}), 
\begin{align*}
\prob(L_n= l) &\leq C e^{- c |t_l(n)|^3}    &\hspace*{-3cm} (t_l(n) \leq t_0 < 0),\\*[1mm]
\prob(L_n= l) &\leq C e^{-c |t_l(n)|^{3/5}}  &\hspace*{-3cm} ( 0< t_1 \leq t_l(n)),
\end{align*} 
when $n$ is large enough with $c>0$ being some absolute constant and $C>0$ a constant that depends on $t_0, t_1$. 

On the other hand, from Thm.~\ref{thm:lengthdistexpan} and its proof we see that the $F_{j}^*(t)$ take the form
\begin{equation}\label{eq:FjstarForm}
F(t) \cdot \big(\text{rational polynomial in terms of the form $\tr((I-K_0)^{-1}K)|_{L^2(t,\infty)}$}\big),
\end{equation}
where the kernels $K$ are finite sums of rank one kernels with factors of the form \eqref{eq:airyform}. The results of Sect.~\ref{sect:Plemelj} thus show that the $F_{j}^*(t)$ are exponentially decaying when $t\to\infty$. Now, looking at the left tail, the (heuristic) estimate of the largest eigenvalue of the Airy operator ${\mathbf K_0}$ on $L^2(t,\infty)$ as given in the work of Tracy and Widom \cite[Eq.~(1.23)]{MR1257246} shows a superexponential growth bound of the operator norm
\[
\|({\mathbf I}-{\mathbf K_0})^{-1}\| \leq C |t|^{-3/4} e^{c |t|^{3/2}} \qquad (t\leq t_0).
\]
This, together with the superexponential decay (see \cite[Cor.~1.3]{MR2395479} and \cite[Thm.~1]{MR2373439} for the specific constants) 
\[
0\leq F(t) \leq C |t|^{-1/8} e^{-c |t|^3} \qquad (t\leq t_0)
\]
of the Tracy--Widom distribution itself, and with an at most polynomial growth of the trace norms $\|{\mathbf K}\|_{{\mathcal J}^1(t,\infty)}$ as $t\to-\infty$,
shows that the bounds for the discrete distribution find a counterpart for the expansion terms $F_{j}^*(t)$:
\begin{equation}\label{eq:Fjtail}
\begin{aligned}
|F_{j}^*(t)| &\leq C e^{- c |t|^3}    &\quad (t\leq t_0 < 0),\\*[1mm]
|F_{j}^*(t)| &\leq C e^{-c |t|}       &\quad ( 0< t_1 \leq t).
\end{aligned}
\end{equation}
Thus, assuming an additional amount of uniformity that would allow us to absorb the exponentially small tails in the error term of Corollary \ref{cor:PDFexpan}, we conjecture the following:

\subsubsection*{\bf Uniform Tails Hypothesis} The expansion  \eqref{eq:PDFexpan} can be sharpened to include the tails in the form
\begin{equation}\label{eq:PDFexpan2}
n^{1/6}\,\prob(L_n=l) = 
 F' (t) + \sum_{j=1}^m F_{j}^*(t) n^{-j/3} + n^{-(m+1)/3} \cdot O\Big( e^{-c |t|^{3/5}}\Big)\,\bigg|_{t=t_{l-1/2}(n)},
\end{equation}
uniformly valid in $l=1,\ldots,n$ as $n\to\infty$. 

\medskip

We now follow the ideas sketched in our work \cite[§4.3]{arxiv.2206.09411} on the Stirling-type formula. By shift and rescale, the expected value of $L_n$ can be written in the form
\[
\E(L_n) = \sum_{l=1}^n l\cdot \prob(L_n=l) = 2\sqrt{n} +\frac{1}{2} + \sum_{l=1}^n t_{l-1/2}(n)\cdot n^{1/6}\, \prob(L_n=l).
\]
Inserting the expansion \eqref{eq:PDFexpan2} of the uniform tail hypothesis gives, since its error term is uniformly summable,
\begin{equation}\label{eq:EexpandVanilla}
\E(L_n) = 2\sqrt{n} +\frac{1}{2} + \sum_{j=0}^m \mu_j^{(n)} n^{1/6 - j/3} + O(n^{-1/6-m/3}),
\end{equation}
with coefficients (still depending on $n$, though), writing $F_0^* := F'$,
\[
\mu^{(n)}_j := n^{-1/6} \sum_{l=1}^n t_{l-1/2}(n) F_j^*\big(t_{l-1/2}(n)\big).
\]
By the tail estimates \eqref{eq:Fjtail} we have, writing $a :=- n^{-1/6}(2\sqrt{n} + \tfrac12)$ and $h:=n^{-1/6}$,
\[
\mu^{(n)}_j = h \sum_{l=-\infty}^\infty (a+lh) F_j^*(a+l h) + O(e^{-c n^{1/3}}). 
\]
Now, based on a precise description of its pole field in \cite{MR3493969}, it is known that the Hastings--McLeod solution of Painlevé II and a fortiori, by the Tracy--Widom theory \cite{MR1257246},  also $F$ and its derivatives can be continued analytically to the strip $|\Im z| < 2.9$. Therefore, we assume:

\subsection*{Uniform Strip Hypothesis} The $F_j^*$ ($j=1,\ldots,m$) extend analytically to a strip $|\Im z| \leq s$ of the complex $z$-plane, uniformly converging to $0$ as $z\to\infty$ in that strip.

\medskip

\noindent
Under that hypothesis, a classical result about the rectangular rule in quadrature theory (see, e.g., \cite[Eq.~(3.4.14)]{MR760629}) gives
\[
h \sum_{l=-\infty}^\infty (a+lh)  F_j^*(a+l h)  = \int_{-\infty}^\infty F_j^*(t)\,dt + O(e^{-\pi s /h}).
\]
Thus the $\mu_j^{(n)}$ and their limit quantities
\[
\mu_j = \int_{-\infty}^\infty t F_j^*(t)\,dt
\]
 differ only by an exponential small error of at most $O(e^{-c n^{1/6}})$, which can be absorbed in the error term of \eqref{eq:EexpandVanilla}; an illustration of such a rapid convergence is given in \cite[Table~3]{arxiv.2206.09411} for the case $j=0$.
 
The functional form of $F_0^* = F'$, $F_1^*$ and $F_2^*$, namely being a linear combination of higher order derivatives of $F$ with polynomial coefficients (see \eqref{eq:F22star}), allows us to express $\mu_0,\mu_1,\mu_2$ in terms of the moments
\[
M_j := \int_{-\infty}^\infty t^j F'(t)\,dt
\]
of the Tracy--Widom distribution $F$. In fact, repeated integration by parts yields the simplifying rule (where $k\geq 1$)
\[
\int_{-\infty}^\infty t^j\, F^{(k)}(t) \,dt = 
\begin{cases}
\displaystyle\frac{(-1)^{k-1} j!}{(j-k+1)!} M_{j-k+1} &\quad k\leq j+1,\\*[4mm]
0 &\quad\text{otherwise}.
\end{cases}
\]
Repeated application of that rule proves, in summary, the following contribution to Ulam's problem about the expected value when $n$ grows large.

\begin{theorem}\label{thm:Eexpand} Let $m$ be any fixed non-negative integer. Then, subject to the tameness, the uniform tails and the uniform strip hypotheses there holds, as $n\to \infty$,
\begin{equation}\label{eq:Eexpand}
\E(L_n) = 2\sqrt{n} +\frac{1}{2} + \sum_{j=0}^m \mu_j n^{1/6 - j/3} + O(n^{-1/6-m/3}),
\end{equation}
where the constants $\mu_j$ are given by
\[
\mu_0 = \int_{-\infty}^\infty t F'(t)\,dt, \qquad \mu_j = \int_{-\infty}^\infty t F_j^*(t)\,dt \quad (j=1,2,\ldots).
\]
The first few cases can be expressed in terms of the moments of the Tracy--Widom distribution\/{\rm:}
\begin{equation}\label{eq:mu2}
\begin{gathered}
\mu_0 = M_1, \quad \mu_1 = \frac{1}{60}M_2,\quad \mu_2 = \frac{89}{350} - \frac{1}{1400}M_3,\quad
\mu_3 = \frac{538}{7875} M_1+\frac{281}{4536000} M_4;
\end{gathered}
\end{equation}
highly accurate numerical values are listed in {Table~{\rm\ref{tab:numerical}}}.
\end{theorem}

Likewise, by shift and rescale, the variance of $L_n$ can be written in the form
\[
\begin{aligned}
\Var(L_n) &= \sum_{l=1}^n l^2\cdot \prob(L_n=l) - \E(L_n)^2\\*[1mm] 
& = n^{1/3} \sum_{l=1}^n t_{l-1/2}(n)^2 \cdot \prob(L_n=l) - 
\Big( \E(L_n) - 2\sqrt{n} - \frac{1}{2}\Big)^2.
\end{aligned}
\]
By inserting the expansions \eqref{eq:PDFexpan2}, \eqref{eq:Eexpand} and arguing as for Thm.~\ref{thm:Eexpand} we get the following:

\begin{table}[tbp]
\caption{Highly accurate values of $\mu_0,\ldots,\mu_3$ and $\nu_0,\ldots,\nu_3$ as computed from \eqref{eq:mu2}, \eqref{eq:nu2}  based on values for $M_j$ obtained as in \cite[Table~3]{arxiv.2206.09411} (cf.  Prähofer's values for $M_1,\ldots,M_4$, published in \cite[p.~70]{MR2787973}). For the values of $\mu_4,\mu_5$ and $\nu_4,\nu_5$ see the supplementary material mentioned in Fn.~\ref{fn:m5}.}
\label{tab:numerical}
{\footnotesize
\begin{tabular}{rrrr}
\hline\noalign{\smallskip}
$j$ & $M_{j}$\hspace*{1.75cm} & $\mu_j$ \hspace*{1.5cm}& $\nu_j$ \hspace*{1.5cm} \\
\noalign{\smallskip}\hline\noalign{\smallskip}
$0$ & $  1.00000\,00000\,00000\,00000\cdots$& $-1.77108\,68074\,11601\,62598\cdots$ & $ 0.81319\,47928\,32957\,84477\cdots$\\
$1$ & $ -1.77108\,68074\,11601\,62598\cdots$& $ 0.06583\,23878\,70339\,62521\cdots$ & $-1.20720\,50777\,85797\,46901\cdots$\\
$2$ & $  3.94994\,32722\,20377\,51300\cdots$& $ 0.26122\,27462\,52162\,60525\cdots$ & $ 0.56715\,66368\,69744\,43503\cdots$\\
$3$ & $ -9.71184\,47530\,27647\,35361\cdots$& $-0.11938\,39067\,94582\,09131\cdots$ & $ 0.01669\,21858\,10456\,60764\cdots$\\
$4$ & $ 26.02543\,54268\,39994\,56536\cdots$& $-0.00483\,35524\,95005\,83878\cdots$ & $-0.12447\,09934\,16776\,05579\cdots$\\
$5$ & $-74.20410\,74434\,81824\,47477\cdots$& $ 0.01222\,78407\,77590\,95405\cdots$ & $-0.00293\,40551\,03931\,43008\cdots$\\
\noalign{\smallskip}\hline
\end{tabular}}
\end{table}

\begin{corollary}\label{cor:Vexpand} Let $m$ be any fixed non-negative integer. Then, subject to the tameness, the uniform tails and the uniform strip hypotheses there holds, as $n\to \infty$,
\begin{equation}\label{eq:Vexpand}
\Var(L_n) =  \sum_{j=0}^m \nu_j n^{1/3 - j/3} + O(n^{-m/3}),\\*[2mm]
\end{equation}
with certain constants $\nu_j$. The first few cases can be expressed in terms of the moments of the Tracy--Widom distribution
\begin{equation}\label{eq:nu2}
\begin{gathered}
\nu_0 = -M_1^2 + M_2, \quad \nu_1 = -\frac{67}{60} + \frac{1}{30} \big(-M_1 M_2 + M_3\big),\\*[2mm]
\nu_2 = -\frac{57}{175} M_1 + \frac{1}{700} M_1 M_3 - \frac{1}{3600}M_2^2  - \frac{29}{25200}M_4,\\*[2mm]
\nu_3 = 
-\frac{1076}{7875}M_1^2 -\frac{281}{2268000}M_1M_4
+\frac{893}{7875}M_2 +\frac{1}{42000}M_2M_3 + \frac{227}{2268000} M_5;
\end{gathered}
\end{equation}
highly accurate numerical values are listed in {Table~{\rm\ref{tab:numerical}}}.
\end{corollary}

The expansions of expected value and variance can be cross-validated by looking at the numerical values for the coefficients $\mu_1,\mu_2,\mu_3$ and $\nu_1,\nu_2,\nu_3$ that we predicted in \cite[§4.3]{arxiv.2206.09411}: those values were computed by fitting, in high precision arithmetic,  expansions (back then only conjectured) of the form \eqref{eq:Eexpand} with $m=9$ 
and \eqref{eq:Vexpand} with $m=8$ to the exact tabulated data for $n=500,\ldots,1000$. A decision about which digits were to be considered correct was made by comparing the result against a similar computation for $n=600,\ldots,1000$. As it turns out, the predictions of \cite[§4.3]{arxiv.2206.09411} agree to all the decimal places shown there (that is, to $7$, $7$, $6$ and $9$, $6$, $4$ places) with the theory-based, highly accurate values given in Table~\ref{tab:numerical}. 

{
\subsection*{Acknowledgements} I would like to thank the Isaac Newton Institute for Mathematical Sciences, Cambridge (UK), for support and hospitality during the 2022 program “Applicable resurgent asymptotics: towards a universal theory (ARA2)” where work on the present paper was undertaken. This work was supported by EPSRC Grant No EP/R014604/1.} 

\setcounter{section}{0}
\renewcommand{\thesection}{\Alph{section}} 

\part*{Appendices}
\section{Variations on the Saddle Point Method}

\subsection{Analytic de-Poissonization and the Jasz expansion}\label{sect:jasz}

In their comprehensive 1998 memoir \cite{MR1625392}, Jacquet and Szpankowski gave a detailed study of what they termed {\em analytic de-Poissonization} (in form of a useful repackaging of the saddle point method), proving a selection of asymptotic expansions and applying them to various asymptotic problems in analytic algorithmics and combinatorics (with generating functions given in terms of functional equations amenable for checking the Tauberian growth conditions in the complex plane). Expositions with a selection of further applications can be found in \cite[Sect.~7.2]{MR3524836} and \cite[Chap.~10]{MR1816272}.

\subsubsection*{Formal derivation of the Jasz expansion} Following the ideas of \cite[Remark~3]{MR1625392} let us start with a purely formal derivation to motivate the algebraic form of the expansion. Suppose that the Poisson generating function
\[
P(z) = e^{-z}\sum_{n=0}^\infty a_n \frac{z^n}{n!}
\]
of a sequence $a_n$ is an entire function and consider some $r>0$. If we write the power series expansion of $P(z)$, centered at $z=r$, in the operator form
\[
P(z) = e^{(z-r)D} P(r),
\]
where $D$ denotes differentiation w.r.t. the variable $r$, we get by Cauchy's formula (with a contour encircling $z=0$ counter-clockwise with index one)
\begin{equation}\label{eq:formal1}
a_n = \frac{n!}{2\pi i} \oint P(z) e^z \frac{dz}{z^{n+1}} =  e^{-rD} \left(\frac{n!}{2\pi i} \oint e^{z(D+1)} \frac{dz}{z^{n+1}}\right) P(r) = e^{-rD} (D+1)^n P(r).
\end{equation}
By the Cauchy product of power series
\begin{equation}\label{eq:cj}
e^{-r x} (x+1)^n = \sum_{j=0}^\infty c_j(n;r) x^j,\quad c_j(n;r):= \sum_{k=0}^j \binom{n}{k}\frac{(-r)^{j-k}}{(j-k)!},
\end{equation}
we get from \eqref{eq:formal1} the formal expansion
\begin{equation}\label{eq:formal2}
a_n \sim \sum_{j=0}^\infty c_j(n;r) P^{(j)}(r).
\end{equation}
Note that the coefficients $c_j(n;r)$ are polynomials of degree $j$ in $n$ and $r$. From \eqref{eq:cj} one easily verifies that they satisfy the three-term recurrence
\[
(j+1) c_{j+1}(n;r) + (j+r-n) c_j(n;r) +r c_{j-1}(n;r) = 0\qquad (j=0,1,2,\ldots)
\]
with initial data $c_0(n;r) =1$ and $c_1(n;r) = n-r$.

\begin{remark} From \eqref{eq:cj} and \cite[§2.81]{MR0372517} one immediately gets that the $c_j(n;r)$ are, up to normalization, the {\em Poisson--Charlier polynomials}\/: 
\[
\sum_{n=0}^\infty c_j(n;r) c_k(n;r) \frac{e^{-r}r^n}{n!} = \delta_{jk} \frac{r^j}{j!}\qquad (j,k = 0,1,2,\ldots),
\]
so that they are orthogonal w.r.t. the Poisson distribution of intensity $r>0$.
In particular, \cite[Eq.~(2.81.6)]{MR0372517} gives (with $L^{(\nu)}_k(x)$ the Laguerre polynomials) the representation
\[
c_j(n;r) = L_j^{(n-j)}(r).
\]
\end{remark}

Things simplify for the particular choice $r=n$ which is suggested by the expected value of the Poisson distribution (cf.~Lemma~\ref{lem:johan}). The corresponding polynomials $b_j(n):=c_j(n;n)$, which we call the {\em diagonal Poisson--Charlier polynomials}, satisfy 
the three-term recurrence
\begin{equation}\label{eq:bj}
\begin{gathered}
b_0(n)=1,\quad
b_1(n)=0,\\*[2mm]
(j+1) b_{j+1}(n) + j b_j(n) + n b_{j-1}(n) = 0\quad (j=0,1,2,\ldots).
\end{gathered}
\end{equation}
From this we infer inductively that
\[
b_j(0)=0,\quad \deg b_j \leq \lfloor j/2\rfloor \qquad (j=1,2,\ldots).
\]
Now the formal expansion \eqref{eq:formal1} becomes what is dubbed the {\em Jasz expansion} in \cite[§VIII.18]{MR2483235}:
\begin{equation}\label{eq:jaszformal}
\begin{aligned}
a_n &\sim P(n)  + \sum_{j=2}^\infty b_j(n) P^{(j)}(n) \\*[1mm]
&= P(n) - \frac{n}{2} P''(n) + \frac{n}{3}P'''(n) + \left(\frac{n^2}{8} - \frac{n}{4}\right) P^{(4)}(n) \\*[1mm]
& \quad\qquad + \left(-\frac{n^2}{6} + \frac{n}{5}\right) P^{(5)}(n) 
+ \left(-\frac{n^3}{48} + \frac{13n^2}{72}- \frac{n}{6}\right) P^{(6)}(n)+ \cdots.
\end{aligned}
\end{equation}

\subsubsection*{Diagonal analytic de-Poissonization} Jacquet and Szpankowski were able to prove that the expansion \eqref{eq:jaszformal} can be made rigorous if the Poisson generating function satisfies a Tauberian condition in form of a growth condition at the essential singularity at $z=\infty$ in the complex plane. In fact, this can be cast to accomodate the needs of double scaling limits in a uniform fashion: for a two-parameter family of coefficients $a_{n,k}$ one expands the diagonal term $a_{n,n}$ by, first, applying the Jasz expansion w.r.t. to $n$ for $k$ fixed and, then, selecting $k=n$ only afterwards (a process that is called {\em diagonal de-Poissonization} in \cite{MR1625392}).

The following theorem is a particular case of \cite[Thm.~4]{MR1625392} (with $\Psi=1$ and the modifications discussed preceding \cite[Eq.~(27)]{MR1625392}). It repackages the saddle point method (cf. \cite[Chap.~5]{MR671583} and \cite[Chap.~VI]{MR0460820}) for the asymptotic evaluation of the Cauchy integral
\begin{equation}\label{eq:Cauchy}
a_{n} = \frac{n!}{2\pi i} \oint P(z) e^z \frac{dz}{z^{n+1}} 
\end{equation}
in a far more directly applicable fashion. Concerning the asserted uniform bounds of the implied constants, see the beginning of \cite[§5.2]{MR1625392}.

\begin{theorem}[Jacquet--Szpankowski 1998]\label{thm:jasz} Let a family of entire Poisson generating functions of the form
\[
P_k(z) = e^{-z} \sum_{n=0}^\infty a_{n,k} \frac{z^n}{n!} \qquad (k=0,1,2,\ldots)
\]
satisfy the following two conditions\/\footnote{Here, (I) means ``inside'' and (O) ``outside'' with respect to the ``polynomial cone'' $\{z=re^{i\theta} : |\theta| \leq D r^{-\delta}\}$.} for $n\geq n_0$ where $A, B,C, D$, $\alpha,\beta,\gamma,\delta$ are some constants with $A, \alpha >0$ and $0\leq \delta < 1/2$\/{\em :}
\begin{itemize}\itemsep=5pt
\item[(I)] If $|r - n| \leq D n^{1-\delta}$ and $|\theta| \leq Dr^{-\delta}$ then $|P_n(r e^{i\theta})| \leq B n^\beta$.
\item[(O)] If $|\theta|>Dn^{-\delta}$ then $|P_n(n e^{i\theta}) \exp(n e^{i\theta})| \leq C n^\gamma \exp(n - An^\alpha)$.
\end{itemize}
Then, for any $m = 0,1,2,\ldots$ there holds, when $n\geq n_1$ with $n_1$ large enough,
\begin{equation}\label{eq:jasz}
a_{n,n} = P_n(n)  + \sum_{j=2}^m b_j(n) P_n^{(j)}(n) + O\big(n^{\beta - (m+1)(1-2\delta)}\big),
\end{equation}
where the $b_j(n)$ are the diagonal Poisson--Charlier polynomials \eqref{eq:bj} which have degree $\leq \lfloor j/2\rfloor$ and satisfy $b_j(0)=0\; (j\geq 1)$. The implied constant in \eqref{eq:jasz} and the constant $n_1$ depend only on $n_0$ and the constants entering the conditions {\rm (I)} and {\rm (O)}.
\end{theorem}

\begin{example}\label{ex:jaszP4} In the proof of Thm.~\ref{thm:lengthdistexpan} we use Thm.~\ref{thm:jasz} in the particular case $\beta = 0$, $\delta=2/5$ for a family of Poisson generating functions with (cf.~\eqref{eq:Pnnlead})
\[
P_n^{(j)}(n) = O(n^{-2j/3}).
\]
For $m=6$ the expansion \eqref{eq:jasz} is then given by the terms shown in \eqref{eq:jaszformal} up to an error of order $O(n^{-7/5})$, that is,
\begin{multline*}
a_{n,n} = P_n(n) - \frac{n}{2} P_n''(n) + \frac{n}{3}P_n'''(n) + \left(\frac{n^2}{8} - \frac{n}{4}\right) P_n^{(4)}(n) \\*[1mm]
+ \left(-\frac{n^2}{6} + \frac{n}{5}\right) P_n^{(5)}(n) 
+ \left(-\frac{n^3}{48} + \frac{13n^2}{72}- \frac{n}{6}\right) P_n^{(6)}(n)+ O(n^{-7/5}).
\end{multline*}
Upon relaxing the error to $O(n^{-4/3})$ and keeping only those terms which do not get absorbed in the error term,
the Jasz expansion then simplifies to
\begin{equation}\label{eq:jaszP4}
a_{n,n} = P_n(n) - \frac{n}{2} P_n''(n) + \frac{n}{3} P_n'''(n) + \frac{n^2}{8} P_n^{(4)}(n) - \frac{n^3}{48}P_n^{(6)}(n) + O(n^{-4/3}).
\end{equation}
\end{example}

\subsection{\boldmath$H$-admissibility\unboldmath\ and Hayman's Theorem XI}\label{sect:hayman}
In his 1956 memoir \cite{Hayman56} on a generalization of Stirling's formula, Hayman gave a related but different repackaging of the saddle point method for the asymptotic evaluation of the Cauchy integral \eqref{eq:Cauchy} by introducing the notion of $H$-admissible functions. We collect estimates given in course of the proofs of some of Hayman's theorems that will help us to establish the conditions (I) and (O) required for applying analytic de-Poissonization in form of Thm.~\ref{thm:jasz}.

\begin{definition}[Hayman \protect{\cite[p.~68]{Hayman56}}]\label{def:Hayman} An entire function $f(z)$ is said to be {\em $H$-admissible} if the following four conditions are satisfied:
\begin{itemize}
\item[--] [{\em positivity}\/] for sufficiently large $r>0$, there holds $f(r)>0$; inducing there the real functions
(which we call the auxiliary functions associated with $f$)
\[
a(r) = r \frac{f'(r)}{f(r)},\qquad b(r) = r a'(r);
\]
by Hadamard's convexity theorem $a(r)$ is monotonely increasing and $b(r)$ is positive.\\*[-3mm]
\item[--]  [{\em capture}\/] $b(r) \to \infty$ as $r\to\infty$;\\*[-3mm]
\item[--]  [{\em locality}\/] for some function $0<\theta_0(r)<\pi$ there holds\footnote{As is customary in asymyptotic analysis in the complex plane, we understand such asymptotics (and similar expansions with $o$- or $O$-terms) to hold {\em uniformly} in the stated angular segments for all $r\geq r_0$ with some sufficiently large $r_0>0$.}
\[
f(r e^{i\theta}) = f(r) e^{i\theta a(r) - \theta^2 b(r)/2}\, (1+ o(1)) \qquad (r\to\infty,\; |\theta|\leq \theta_0(r));
\]
\item[--]  [{\em decay}\/] for the angles in the complement there holds
\[
f(r e^{i\theta}) = \frac{o(f(r))}{\sqrt{b(r)}}\qquad (r\to\infty,\; \theta_0(r) \leq |\theta| \leq \pi).
\]
\end{itemize}
\end{definition}

Instead of providing an asymptotic expansion (with an additive error term) as in Thm.~\ref{thm:jasz}, $H$-admissibility gives just a versatile leading order term of $a_n$ in form of a normal approximation.
 However, the error term is {\em multiplicative} then. 
\begin{theorem}[Hayman \protect{\cite[Thm.~I, Cor.~II]{Hayman56}}]\label{thm:hayman} Let $f$ 
be an entire $H$-admissible function with Maclaurin series
\[
f(z) = \sum_{n=0}^\infty a_nz^n\qquad (z\in \C).
\]
Then: 
\begin{itemize}
\item[I.]  {\em [normal approximation]} There holds, uniformly in $n \in \N_0=\{0,1,2,\ldots\}$, that
\begin{equation}\label{eq:CLT}
\frac{a_n r^n}{f(r)} = \frac{1}{\sqrt{2\pi b(r)}}\left(\exp\left(-\frac{(n-a(r))^2}{2b(r)}\right)+ o(1)\right) \qquad (r\to\infty).
\end{equation}

\item[II.]  {\em [Stirling-type formula]} For $n$ sufficiently large, it follows from the positivity and capture conditions of $H$-admissibility that $a(r_n) = n$ has a unique solution $r_n$ such that $r_n\to\infty$ as $n\to\infty$ and therefore, by the normal approximation \eqref{eq:CLT}, there holds
\begin{equation}\label{eq:stirling}
a_n = \frac{f(r_n)}{r_n^n \sqrt{2\pi b(r_n)}} (1+ o(1)) \qquad (n\to\infty).
\end{equation}
\end{itemize}
\end{theorem}

For the probabilistic content of the normal approximation \eqref{eq:CLT} see, e.g., \cite{MR2095975} and \cite[Remark~2.1]{arxiv.2206.09411}.

We observe the similarity of the locality and decay conditions to the conditions (I) and (O) in the Jacquet--Szpankowski Thm.~\ref{thm:jasz}. In fact, in establishing the $H$-admissibility of certain families of functions, Hayman proved estimates that allow us to infer the validity of conditions (I) and (O). A striking example is given by the following theorem, which gives {\em uniform} bounds for a class of functions that is of particular interest to our study. 

\begin{theorem}[Hayman \protect{\cite[Thm.~XI]{Hayman56}}]\label{thm:genuszero} Let $f$ be an entire function of genus zero, having for some $\epsilon>0$ no zeros in the sector $|\!\arg z| \leq \pi/2+ \epsilon$. If $f$ satisfies the positivity condition of Def.~{\rm\ref{def:Hayman}}, then there is the universal bound
\begin{equation}\label{eq:genuszero}
\big| f(re^{i\theta}) \big| \leq 
\begin{cases}
2f(r) e^{-\frac{1}{2}\theta^2 b(r)}, &\qquad 0\leq |\theta| \leq b(r)^{-2/5},\\*[2mm]
2f(r) e^{-\frac{1}{2}b(r)^{1/5}}, &\qquad b(r)^{-2/5}\leq |\theta| \leq \pi,
\end{cases}
\end{equation}
which is valid when $b(r)$ is large enough to ensure $8b(r)^{-1/5} \csc^2(\epsilon/2)\csc(\epsilon) \leq \log 2$. Hence, if $f$ also satisfies the capture condition of Def.~{\rm\ref{def:Hayman}}, then it is $H$-admissible.
\end{theorem}

\begin{proof} Since the bound \eqref{eq:genuszero} is hidden in the two-page long proof of \cite[Thm.~XI]{Hayman56} (only the $H$-admissibility is stated explicitly there), we collect the details here. 
First, \cite[Eq.~(15.6)]{Hayman56} states that, if $|\theta| \leq 1/4$, then
\[
\log f(r e^{i\theta}) = \log f(r) + i\theta a(r) - \frac{1}{2}\theta^2 b(r) + \epsilon(r,\theta)
\]
where the error term is bounded by
\[
|\epsilon(r,\theta)| \leq c(\epsilon) \cdot b(r)\, |\theta|^3, \qquad c(\epsilon):= 8 \csc^2(\epsilon/2)\csc(\epsilon).
\]
Now, for $b(r)$ large enough to ensure
\[
b(r)^{-1/5} \leq c(\epsilon)^{-1}\log 2 \leq\min\big(\sqrt{2\epsilon},1/2\big)
\]
we thus get with $0 < \theta_0(r) := b(r)^{-2/5} \leq \min\big(2\epsilon,1/4\big)$ and $0\leq |\theta| \leq \theta_0(r)$ that
\[
\log \big|f(r e^{i\theta})\big| = \Re \log f(r e^{i\theta}) = \log f(r) - \frac{1}{2}\theta^2 b(r)  + \Re \epsilon(r,\theta),
\quad \big| \Re \epsilon(r,\theta) \big| \leq \log2.
\]
Exponentiation gives, for $0\leq |\theta| \leq \theta_0(r)$,
\[
\big| f(re^{i\theta}) \big| \leq 2f(r) e^{-\frac12\theta^2 b(r)}.
\]
Next, if we combine this estimate with \cite[Lemma~8]{Hayman56} we get, since $\theta_0(r)\leq 2\epsilon$, that 
\[
\big| f(re^{i\theta}) \big| \leq \big| f(re^{i\theta_0(r)}) \big| \leq 2 f(r) e^{-\frac{1}{2} \theta_0(r)^2 b(r)} = 2 f(r) e^{-\frac{1}{2} b(r)^{1/5}} \qquad (\theta_0(r)\leq |\theta|\leq \pi)
\]
which finishes the proof of the universal bound \eqref{eq:genuszero}.
\end{proof}

If, instead of having no zeros in the sector $|\!\arg z| \leq \pi/2+ \epsilon$ at all, the entire function $f$ has a finite number of them, Thm. \ref{thm:genuszero} remains valid but the lower bound on $b(r)$ will now depend on these finitely many zeros. To restore uniformity we consider families of such functions whose zeros satisfy the following tameness condition.

\begin{definition}\label{def:nonresonant}  Let $f_n$ be a family of entire functions such that, for some fixed $\epsilon>0$, each of them has finitely many zeros (listed according to their multiplicities)
\[
z_{n,1},\ldots, z_{n,m_n}
\]
in the sector $|\!\arg z| \leq \pi/2+ \epsilon$, none of them being a positive real number. We call these zeros {\em uniformly tame} (w.r.t. the positive real axis and w.r.t. infinity) if there are some constants $1/5 < \mu\leq 1/3$ and $\nu >0$ such that the family of polynomials
\begin{equation}\label{eq:zeropol}
 p_n(z) = (z-z_{n,1})\cdots(z-z_{n,m_n})
\end{equation}
satisfies
\begin{equation}\label{eq:nonresonant}
\Big(r \frac{d}{dr}\Big)^2 \log p_n(r) = -\sum_{j=1}^{m_n} \frac{r z_{n,j}}{(r -z_{n,j})^2} = O(r^{1 -\mu}), \quad \quad |p_n(r e^{i\theta})| = p_n(r)(1+O(r^{-\nu})),
\end{equation}
uniformly in $n - n^{3/5} \leq r \leq n + n^{3/5}$ as $n\to\infty$. 
\end{definition}

\begin{remark}\label{rem:nonresonant}
Note that a single function $f$ would satisfy condition \eqref{eq:nonresonant} with error terms of the form $O(r^{-1})$ in both places. Therefore the tameness condition allows us to accommodate a significant growth of the implied constants in these $O(r^{-1})$ terms as $n\to\infty$: in the first case because of zeros of $f_n$ getting close to the positive real axis and in the second case because of them getting large.
\end{remark}

\begin{corollary}\label{cor:nonresonant} Let $f_n$ be a family of entire functions of genus zero with positive Maclaurin coefficients such that, for some fixed $\epsilon>0$, each of them has a most finitely many zeros in the sector $|\!\arg z| \leq \pi/2+ \epsilon$. If these zeros are uniformly tame  in the sense of Def.~{\rm\ref{def:nonresonant}} and if the auxiliary functions belonging to $f_n$ satisfy
\begin{equation}\label{eq:bgrowth}
b_n(r) = r + O(r^{2/3}) \qquad (r\to\infty),
\end{equation}
uniformly in $n - n^{3/5} \leq r \leq n + n^{3/5}$ as $n\to\infty$, then there holds the bound
\begin{equation}\label{eq:practicalgenuszero}
\big| f_n(re^{i\theta}) \big| \leq 
\begin{cases}
2f_n(r) e^{-\frac{1}{2}\theta^2 r}, &\qquad 0\leq |\theta| \leq r^{-2/5},\\*[2mm]
2f_n(r) e^{-\frac{1}{2}r^{1/5}}, &\qquad r^{-2/5}\leq |\theta| \leq \pi,
\end{cases}
\end{equation}
for all $n - n^{3/5} \leq r \leq n + n^{3/5}$ and $n\geq n_0$, $n_0$ being sufficiently large. Here $n_0$ depends only on the parameters of the tameness condition and the implied constants in \eqref{eq:nonresonant} and \eqref{eq:bgrowth}.
\end{corollary}
\begin{proof} Factoring out the finitely many zeros of $f_n$ in the sector $|\!\arg z| \leq \pi/2+ \epsilon$ by using
the polynomials \eqref{eq:zeropol}, we have
\[
f_n(z) = f_n^*(z)\cdot p_n(z) 
\]
where $f_n^*$ is an entire function of genus zero that has no zeros in that sector. Since $f_n(r)>0$ for $r>0$ and the leading coefficient of the polynomial $p_n(z)$ is one, $f_n^*$ satisfies the positivity condition of Def.~\ref{def:Hayman}.
Denoting the auxiliary functions of $f_n^*$ by $a_n^*$ and $b_n^*$, the tameness condition \eqref{eq:nonresonant} yields
\[
b_n(r) = b_n^*(r) + \Big(r \frac{d}{dr}\Big)^2 \log p_n(r)  = b_n^*(r) + O(r^{1-\mu}), \quad |p_n(r e^{i\theta})| = p_n(r)(1+O(r^{-\nu})),
\]
uniformly in $n - n^{3/5} \leq r \leq n + n^{3/5}$ as $n\to\infty$.

By \eqref{eq:bgrowth} this gives $b_n^*(r) = r(1 + O(r^{-\mu}))$, so that by Thm.~\ref{thm:genuszero} (its proof shows that we can take a factor $3/2$ instead of $2$ if $\log 2$ is replaced by $\log(3/2)$ in the lower bound on $b(r)$)
\[
\big| f_n^*(re^{i\theta}) \big| \leq 
\begin{cases}
\tfrac{3}{2}f_n^*(r) e^{-\frac{1}{2}\theta^2 b_n^*(r)}, &\qquad 0\leq |\theta| \leq b_n^*(r)^{-2/5},\\*[2mm]
\tfrac{3}{2}f_n^*(r) e^{-\frac{1}{2}b_n^*(r)^{1/5}}, &\qquad b_n^*(r)^{-2/5}\leq |\theta| \leq \pi,
\end{cases}
\]
for $n$ large enough to ensure $8b_n^*(r)^{-1/5} \csc^2(\epsilon/2)\csc(\epsilon) \leq \log(3/2)$. We write this briefly as
\[
\big| f_n^*(re^{i\theta}) \big| \leq \tfrac{3}{2}f_n^*(r) \exp\big(-\tfrac{1}{2}\min\big(\theta^2 b_n^*(r),b_n^*(r)^{1/5}\big)\big)
\]
for $n - n^{3/5} \leq r \leq n + n^{3/5}$ and $n\geq n_0$ where $n_0$ is large enough (just depending on the parameters and the implied constants in the tameness condition).
If we multiply this bound by 
\[
|p_n(r e^{i\theta})| = p_n(r)(1+O(r^{-\nu}))
\]
 and use $b_n^*(r) = r(1 + O(r^{-\mu}))$
to infer
\[
\min\big(\theta^2 b_n^*(r),b_n^*(r)^{1/5}\big) = \min(\theta^2 r, r^{1/5})\cdot(1+ O(r^{-\mu})) =  \min(\theta^2 r, r^{1/5}) + O(r^{1/5-\mu}),
\]
we obtain the asserted estimate in the compact form
\[
\big| f_n(re^{i\theta}) \big| \leq  2 f_n(r) \exp\big(-\tfrac{1}{2}\min\big(\theta^2 r,r^{1/5}\big)\big)
\]
for $n-n^{3/5}\leq r \leq n + n^{3/5}$ and $n\geq n_0$, where $n_0$ is large enough.
\end{proof}

\subsection{Bessel functions of large order in the transition region}\label{sect:bessel}

In the 1950s F. Olver started a systematic and exhaustive study of asymptotic expansions of the Bessel functions $J_\nu(z)$ for large order $\nu$ and argument $z$. For the transition region\footnote{Where $J_\nu(\nu+\tau \nu^{1/3})$ changes at about $\tau\approx 0$ from being superexponentially small (to the left) to being oscillatory (to the right).} $z= \nu+\tau \nu^{1/3}$ he obtained from applying the saddle point method to integral representations of Sommerfeld's type
  the asymptotic expansion \cite[Eq.~(3.1)]{MR48638} (cf. also \cite[§10.19(iii)]{MR2723248})
\begin{subequations}\label{eq:Olver1952}
\begin{equation}
J_\nu(\nu+\tau \nu^{1/3}) \sim \frac{2^{1/3}}{\nu^{1/3}} \Ai(-2^{1/3}\tau)\sum_{k=0}^\infty \frac{A_k(\tau)}{\nu^{2k/3}} +  \frac{2^{2/3}}{\nu^{1/3}} \Ai'(-2^{1/3}\tau) \sum_{k=1}^\infty \frac{B_k(\tau)}{\nu^{2k/3}}
\end{equation}
valid when $|\!\arg \nu| \leq \pi/2 - \delta < \pi$ with $\tau$ being any {\em fixed} complex number.  Here, $A_k(\tau)$ and $B_k(\tau)$ are certain rational polynomials of increasing degree; the first few are \cite[Eq.~(2.42)]{MR48638}\footnote{Note that we keep the indexing of the polynomials
$B_k$ as in \cite{MR48638}, which differs from \cite[§10.19(iii)]{MR2723248}.}
\begin{gather}
A_0(\tau)=1,\quad A_1(\tau) = -\frac15\tau, \quad A_2(\tau) = -\frac{9}{100}\tau^5 + \frac{3}{35}\tau^2,\\*[1mm]
B_0(\tau)=0,\quad B_1(\tau)=\frac{3}{10}\tau^2,\quad B_2(\tau) = -\frac{17}{70}\tau^3 + \frac{1}{70}. 
\end{gather}
\end{subequations}

\begin{remark}\label{rem:PQcompute} The sequence \cite[Eqs.~(2.10), (2.14), (2.18), (2.38), (2.40)]{MR48638} of formulae in Olver's 1952 paper gives an actual method\footnote{By a Mathematica implementation (for download at \url{https://arxiv.org/abs/2301.02022}) we extended (and reproduced) Olver's original table \cite[Eq.~(2.42)]{MR48638} of $A_0,\ldots,A_n$ and $B_0,\ldots, B_n$ from $n=4$ to $n=100$ in about $10$ minutes computing time. The polynomials $A_{100}$ and $B_{100}$ of degree $250$ and $248$ exhibit rational coefficients that are ratios of integers with up to $410$ digits.} to calculate $A_k(\tau)$ and $B_k(\tau)$ (combining reversion and nesting of power series with recursive formulae). The degrees of $A_k$ are the positive integers congruent to $0,1$ mod $5$ (starting with $\deg A_1=1$) and the degrees of $B_k$ are the positive integers congruent to $2,3$ mod $5$ (starting with $\deg B_1=2$). In both families of polynomials the coefficients of $\tau^m$ are zero when $m$ is not congruent mod~$3$ to the degree. \end{remark}

As stated in \cite[p.~422]{MR48638}, the expansion \eqref{eq:Olver1952} can be repeatedly differentiated with respect to $\tau$, valid under the same conditions. For a modern account of differentiability w.r.t. $\tau$ and $\nu$, adding uniformity for $\tau$ from any compact real set, see the recent work of Sher \cite[Prop.~2.8]{Sher22} which is based on the (microlocal) theory of so-called polyhomogeneous conormal joint asymptotic expansions. 

The purposes of Sect.~\ref{sect:hard-to-soft} require to identify a larger region of real $\tau$ where the expansion~\eqref {eq:Olver1952} is uniform as $\nu\to\infty$ through positive real values. To this end we use the uniform asymptotic expansions of Bessel functions for large order $\nu$, pioneered by Olver \cite{MR67250} in 1954 by analyzing turning points of the Bessel differential equation (cf.~\cite[Chap.~11]{MR0435697}, \cite[Chap.~VIII]{MR0460820} and, for exponential representations of the asymptotic series, also \cite[§4]{MR3735704}):
\begin{subequations}\label{eq:Olver1954}
\begin{equation}
J_\nu(\nu z) \sim \left(\frac{4\zeta}{1-z^2}\right)^{1/4} \left(\frac{\Ai(\nu^{2/3}\zeta)}{\nu^{1/3}} \sum_{k=0}^\infty \frac{A_k^*(\zeta)}{\nu^{2k}}+ \frac{\Ai'(\nu^{2/3}\zeta)}{\nu^{5/3}} \sum_{k=0}^\infty \frac{B_k^*(\zeta)}{\nu^{2k}}\right),
\end{equation}
uniformly for $z\in(0,\infty)$ as $\nu\to\infty$. Here, the parameters and coefficients are, for $0<z<1$, 
\begin{equation}
\frac{2}{3} \zeta^{3/2} = \log\left(\frac{1+\sqrt{1-z^2}}{z}\right) - \sqrt{1-z^2},
\end{equation}
and
\begin{align}
A_k^*(\zeta) &= \sum_{j=0}^{2k} \left(\frac{3}{2}\right)^j v_j \zeta^{-3j/2} U_{2k-j}\big((1-z^2)^{-1/2}\big),\\*[1mm]
B_k^*(\zeta) &= -\zeta^{-1/2}\sum_{j=0}^{2k+1} \left(\frac{3}{2}\right)^j u_j \zeta^{-3j/2} U_{2k-j+1}\big((1-z^2)^{-1/2}\big),
\end{align}
\end{subequations}
where the $U_k(x)$ are recursively defined rational polynomials of degree $3k$ (cf.~\cite[Eq.~(2.19)]{MR67250}) and $u_k, v_k$ ($u_0=v_0=1$) are the rational coefficients of the asymptotic expansions of the Airy function and its derivative in a sector containing the positive real axis:
\begin{equation}\label{eq:airyexpan}
\Ai(z) \sim \frac{e^{-\xi}}{2\sqrt{\pi}z^{1/4}} \sum_{k=0}^\infty (-1)^k \frac{u_k}{\xi^k}, \quad
\Ai'(z) \sim -\frac{z^{1/4}e^{-\xi}}{2\sqrt{\pi}} \sum_{k=0}^\infty (-1)^k \frac{v_k}{\xi^k},\quad \xi = \frac{2}{3}z^{3/2},
\end{equation}
as $z \to \infty$ within  $|\!\arg z| \leq \pi - \delta$. Note that $\zeta=\zeta(z)$ can be continued analytically to the $z$-plane cut along the negative real axis;\footnote{In particular, for positive real $z$, the thus defined
$\zeta(z)$ is a strictly monotonically decreasing real function with $\lim_{z\to 0^+} \zeta(z) = +\infty$, $\zeta(1)=0$ and $\lim_{z\to +\infty} \zeta(z) = -\infty$; cf. \cite[Eq.~(10.20.3)]{MR2723248}.} $A_k^*(\zeta)$ and $B_k^*(\zeta)$ can be continued accordingly. 
As stated in \cite[p.~342]{MR67250}, valid under the same conditions while preserving uniformity, the expansion can be repeatedly differentiated with respect to $z$. 

In particular, with $0<\delta <1$ fixed, the power series expansion
\begin{equation}\label{eq:zetaseries}
2^{-1/3}\zeta = (1-z) + \frac{3}{10}(1-z)^2 + \frac{32}{175} (1-z)^3 + \frac{1037}{7875}(1-z)^4 + \cdots 
\end{equation}
converges {\em uniformly} for $|1-z| \leq 1- \delta$ (because of the logarithmic singularity at $z=0$ the radius of convergence of this series is exactly $1$, so that this range of uniformity cannot be extended). If we put
\[
\nu z = \nu + \tau \nu^{1/3}, \text{ i.e.,}\quad z = 1 + \tau \nu^{-2/3},
\]
plugging the uniformly convergent series $\zeta(z)$ into the uniform large $\nu$ expansion \eqref{eq:Olver1954} recovers the form of the transition region expansion \eqref{eq:Olver1952} and proves that it holds uniformly for 
 \[
|\tau| \leq (1-\delta) \nu^{2/3}
\]
as $\nu\to\infty$ through positive real values. 

At the expense of considerably larger error terms, this result can be extended as follows:

\begin{lemma}\label{lem:Olver1952} For any non-negative integer $m$ and any real $\tau_0$ there holds, as $\nu \to\infty$ through positive real values,
\begin{multline}\label{eq:Olver52uniform}
J_\nu(\nu+\tau \nu^{1/3}) = 2^{1/3} \Ai(-2^{1/3}\tau)\sum_{k=0}^m \frac{A_k(\tau)}{\nu^{(2k+1)/3}} +  2^{2/3}\Ai'(-2^{1/3}\tau) \sum_{k=1}^{m} \frac{B_{k}(\tau)}{\nu^{(2k+1)/3}} \\*[2mm]  + \nu^{-1-2m/3} \cdot O\big(\exp(2^{1/3}\tau)\big),
\end{multline}
uniformly for $-\nu^{2/3}< \tau\leq \tau_0$. Here, $A_k(\tau)$ and $B_k(\tau)$ are the rational polynomials in \eqref{eq:Olver1952}. Preserving uniformity, the expansion \eqref{eq:Olver52uniform} can be repeatedly differentiated w.r.t. $\tau$.
\end{lemma}
\begin{proof} Let us write 
\[
J_\nu(\nu+\tau \nu^{1/3}) = E_m(\nu; \tau) + R_m(\nu;\tau),
\]
where $E_m$ denotes the sum of the expansion terms in \eqref{eq:Olver52uniform} and $R_m$ is the remainder. We split the range of $\tau$ into the two parts
\[
\text{(I): } -\frac{3}{4} \nu^{2/3} \leq \tau\leq \tau_0, \qquad
\text{(II): } -\nu^{2/3} < \tau \leq -\frac{3}{4} \nu^{2/3}.
\]
In part (I), as argued above for $\delta=1/4$,  the expansion \eqref{eq:Olver1952} is uniformly valid, that is,
\[
R_m(\nu;\tau) = \nu^{-1-2m/3} \cdot O\left( A_{m+1}(\tau) \Ai\big(-2^{1/3} \tau\big)\right) + \nu^{-1-2m/3} O\left(B_{m+1}(\tau) \Ai'\big(-2^{1/3} \tau\big)\right)
\]
uniformly for these $\tau$. Now, the superexponential decay of the Airy function $\Ai(x)$ and its derivative as $x\to \infty$ through positive values, as displayed in the expansions \eqref{eq:airyexpan}, imply the asserted uniform bound 
\[
R_m(\nu;\tau) =  \nu^{-1-2m/3} \cdot O\big(\exp(2^{1/3}\tau)\big)
\]
in part (I) of the range of $\tau$.

On the other hand, in part (II) of the range of $\tau$, we infer from \eqref{eq:airyexpan} that for $0<\epsilon < 1/2$
\[
E_m(\nu;\tau) = O\big(\exp(-(3\nu^{2/3}/4)^{1+\epsilon}\big) = \nu^{-1-2m/3} \cdot O\big(\exp(2^{1/3}\tau)\big).
\]
We now show that also
\begin{equation}\label{eq:besselbound}
J_\nu\big(\nu + \tau \nu^{1/3} \big) = \nu^{-1-2m/3} \cdot O\big(\exp(2^{1/3}\tau)\big)
\end{equation}
uniformly in part (II) of the range of $\tau$, so that all terms in \eqref{eq:Olver52uniform} are absorbed in the asserted error term. Here we observe 
\[
0 < z = 1 +\tau \nu^{2/3} \leq \frac14,\qquad 1.095\cdots \leq \frac{2}{3}\zeta^{3/2} < \infty,
\]
so that the leading order terms in \eqref{eq:Olver1954} and \eqref{eq:airyexpan} yield the bound
\[
J_\nu(\nu z) \sim \nu^{-1/2} \left(\frac{4}{1-z^2}\right)^{1/4} (\nu^{2/3} \zeta)^{1/4}\Ai(\nu^{2/3}\zeta) = \nu^{-1/2} \cdot O\big(\exp(- \tfrac{2}{3}\zeta^{3/2}\nu)\big),
\]
uniformly for the $\tau$ in (II). Because of $\frac{2}{3}\zeta^{3/2} \geq 1.095$ and $-\nu \leq -2^{1/3}\nu^{2/3} < 2^{1/3}\tau$ for $\nu\geq 2$, this bound can be relaxed, as required, to
\[
J_\nu\big(\nu + \tau \nu^{1/3} \big) = J_\nu(\nu z) =  \nu^{-1-2m/3} \cdot O\big(\exp(2^{1/3}\tau)\big).
\]

Finally, the claim about the derivatives follows from the repeated differentiability of the uniform expansion \eqref{eq:Olver1954} and the differential equation of the Airy function, $\Ai''(x) = x \Ai(x)$ (so that the general form of the expansions underlying the proof does not change).
\end{proof}

\begin{remark}\label{rem:besselshort} The cases $m=0$ and $m=1$ of Lemma~\ref{lem:Olver1952} have previously been stated as \cite[Eq.~(4.11)]{MR1986402}
and \cite[Eq.~(2.10)]{arxiv.2205.05257}. However, the proofs given there are incomplete: in \cite[p.~2978]{MR1986402} 
the power series \eqref{eq:zetaseries} is used up to the boundary of its circle of convergence, so that uniformity becomes an issue; whereas in \cite[p.~9]{arxiv.2205.05257} it is claimed that Olver's transition expansion \eqref{eq:Olver1952} would be uniform w.r.t. $\tau \in (-\infty,\tau_0]$, which is not the case.\footnote{\label{foot:bessel_counter}Besides that the principal branch of $J_\nu(z)$ ($\nu\not\in\Z$) is not defined at negative real $z$, there is a counter-example for  $\nu=n$ being a positive integer: choosing $\tau=0$ in \eqref{eq:Olver1952} gives to leading order
\[
(-1)^nJ_n(-n) = J_n(n) \sim 2^{1/3}n^{-1/3} \Ai(0)\qquad (n\to\infty),
\]
which differs significantly from applying \eqref{eq:Olver1952} formally to $\tau=-2n^{2/3}$ (for which $n+\tau n^{1/3}=-n$).}
\end{remark}

\section{Compilation of the Shinault--Tracy Table and a General Conjecture}\label{app:ST}

\subsection{The Shinault--Tracy table}
Shinault and Tracy \cite[p.~68]{MR2787973} tabulated, for $0 \leq j+k\leq 8$, explicit representations of the terms
\[
u_{jk}(s) = \tr\big((I-K_0)^{-1} \Ai^{(j)} \otimes \Ai^{(k)} \big)\big|_{L^2(s,\infty)}
\] 
as linear combinations of the form (called a {\em linear $F$-form of order $n$} here)
\begin{equation}\label{eq:generalFform}
T_n(s) = p_1(s) \frac{F'(s)}{F(s)} + p_2(s) \frac{F''(s)}{F(s)} + \cdots + p_{n}(s) \frac{F^{(n)}(s)}{F(s)},
\end{equation}
where $p_1,\ldots,p_n$ are certain rational polynomials (depending on $j$, $k$) and $n=j+k+1$. Though they sketched a method to validate each entry of their table, Shinault and Tracy did not describe how they had found those entries in the first place. However, by ``reverse engineering'' their validation method, we can give an algorithm to compile such a table.

Starting point is the Tracy--Widom theory \cite{MR1257246} of representing $F$ in terms of the Hastings--McLeod solution $q(s)$ of Painlevé II,
\begin{subequations}
\begin{align}
q''(s) &= s q(s) +2 q(s)^3, \quad\; q(s) \sim \Ai(s) \quad (s\to\infty),\\*[2mm]
u_{00}(s) &= \frac{F'(s)}{F(s)} = q'(s)^2 - s q(s)^2 - q(s)^4.
\end{align}
\end{subequations} 
From these formulae we get immediately that
\[
F^{(n)}/F \in \Q[s] [q,q'],
\]
where $\Q[s] [q,q']$ denotes the space of polynomials in $q$ and $q'$ with coefficients being rational polynomials. Because $q$ satisfies Painlevé II, $\Q[s] [q,q']$ is closed under differentiation. For a term $T \in \Q[s] [q,q']$ we define the $q$-degree $\deg_q T$ to be the largest $\alpha+\beta$ of a $q$-monomial
\[
q(s)^\alpha q'(s)^\beta
\]
that appears in expanding $T$. Inductively, the general linear $F$-form $T_n(s)$ of order $n$ satisfies
\begin{align*}
\deg_q T_n = \deg_q T_n' &= 4n,\\*[1mm]
\#\text{($q$-monomials with non-zero coefficient in $T_n$)} &= 2(n^2 - n + 2) \quad (n\geq 2), \\*[1mm]
\#\text{($q$-monomials with non-zero coefficient in $T_n'$)} &= 2n^2+1.
\end{align*}
By advancing the set of formulae of \cite{MR1257246}, Shinault and Tracy \cite[pp.~64--66]{MR2787973} obtained
\begin{subequations}\label{eq:TWextended}
\begin{align}
q_0(s) &= q(s), \quad
q_1(s) = q'(s) + u_{00}(s) q(s),\\*[2mm]
q_n(s) &= (n-2) q_{n-3}(s) + s q_{n-2}(s) - u_{n-2,1}(s) q_0(s) + u_{n-2,0}(s) q_1(s), \\*[2mm]
u_{jk}'(s) &= -q_j(s) q_k(s),
\end{align}
\end{subequations} 
where $n=2,3,\ldots$ (ignoring the term $(n-2) q_{n-3}(s)$ if $n=2$). It follows that
\[
u_{n-2,0}, u_{n-2,1} \in  \Q[s] [q,q'] \quad (2\leq n \leq j,k)\quad \Rightarrow \quad u'_{jk} \in  \Q[s] [q,q'].
\]
This suggests the following algorithm to recursively compute the linear $F$-form of order $n$ representing $u_{jk}$ (if such a form exists in the first place): 
suppose such forms have already been found for all smaller $j+k$, we have to find polynomials $p_1,\ldots,p_n \in \Q[s]$ such that    
\begin{equation}\label{eq:linearFForm}
u_{jk} = p_1 \frac{F'}{F} + p_2 \frac{F''}{F} + \cdots + p_{n} \frac{F^{(n)}}{F}.
\end{equation}
By differentiating and then comparing the coefficients of all  $q$-monomials
we get an overdetermined linear system of equations in $\Q[s]$ of size $(2n^2+1)\times 2n$ that is to be satisfied by the polynomials $p_1,\ldots,p_n, p_1',\ldots,  p_n'$. 

For instance, the term $u_{30}$ (as used in Sect.~\ref{sect:functionalform}) can be calculated from the previously established linear $F$-forms (cf. \eqref{eq:F2der} and \cite[p.~68]{MR2787973})
\[
u_{10}(s) = \frac{1}{2}\frac{F''(s)}{F(s)},\qquad u_{11}(s) = -\frac{s F'(s)}{F(s)} + \frac{F'''(s)}{3F(s)}
\]
by setting up the  $33\times 8$ linear system displayed in Table~\ref{tab:linSys_u30}
\begin{table}[tbp]
\caption{The $33\times 8$ linear system for constructing the entry $u_{30}(s)$ in the table \cite[p.~68]{MR2787973}.}
\label{tab:linSys_u30}
\vspace*{-3mm}
{\small
\[
\left(
\begin{array}{cccccccc}
 0 & 0 & 2 & 0 & -1 & 0 & 0 & 2 \\*[1mm]
 0 & 0 & 0 & 8 & 0 & -1 & 0 & 0 \\*[1mm]
 0 & 0 & 0 & 0 & 0 & 0 & -1 & 0 \\*[1mm]
 0 & 0 & 0 & 0 & 0 & 0 & 0 & -1 \\*[1mm]
 0 & 2 & 0 & 8 s & 0 & 0 & 2 & 0 \\*[1mm]
 0 & 0 & 6 & 0 & 0 & 0 & 0 & 8 \\*[1mm]
 0 & 0 & 0 & 12 & 0 & 0 & 0 & 0 \\*[1mm]
 1 & 0 & 2 s & 2 & s & 1 & 0 & 2 s \\*[1mm]
 0 & 2 & 0 & 0 & 0 & 2 s & 3 & 0 \\*[1mm]
 0 & 0 & 3 & 0 & 0 & 0 & 3 s & 6 \\*[1mm]
 0 & 0 & 0 & 4 & 0 & 0 & 0 & 4 s \\*[1mm]
 0 & 0 & -6 s & 4 & 0 & 0 & 0 & -8 s \\*[1mm]
 0 & 0 & 0 & -24 s & 0 & 0 & 0 & 0 \\*[1mm]
 0 & -2 s & 1 & -8 s^2 & 1 & -s^2 & -3 s & 1 \\*[1mm]
 0 & 0 & -6 s & -4 & 0 & 2 & -3 s^2 & -12 s \\*[1mm]
 0 & 0 & 0 & -12 s & 0 & 0 & 3 & -6 s^2 \\*[1mm]
 0 & 0 & 0 & 0 & 0 & 0 & 0 & 4 \\*[1mm]
 0 & 0 & -6 & 12 s^2 & 0 & 0 & 0 & -8 \\*[1mm]
 0 & 0 & 0 & -24 & 0 & 0 & 0 & 0 \\*[1mm]
 0 & -2 & 3 s^2 & -12 s & 0 & -2 s & s^3-3 & 6 s^2 \\*[1mm]
 0 & 0 & -6 & 12 s^2 & 0 & 0 & -6 s & 4 s^3-12 \\*[1mm]
 0 & 0 & 0 & -12 & 0 & 0 & 0 & -12 s \\*[1mm]
 0 & 0 & 0 & 24 s & 0 & 0 & 0 & 0 \\*[1mm]
 0 & 0 & 6 s & -4 s^3-4 & 0 & -1 & 3 s^2 & 12 s-s^4 \\*[1mm]
 0 & 0 & 0 & 24 s & 0 & 0 & -3 & 12 s^2 \\*[1mm]
 0 & 0 & 0 & 0 & 0 & 0 & 0 & -6 \\*[1mm]
 0 & 0 & 0 & 12 & 0 & 0 & 0 & 0 \\*[1mm]
 0 & 0 & 3 & -12 s^2 & 0 & 0 & 3 s & 6-4 s^3 \\*[1mm]
 0 & 0 & 0 & 12 & 0 & 0 & 0 & 12 s \\*[1mm]
 0 & 0 & 0 & -12 s & 0 & 0 & 1 & -6 s^2 \\*[1mm]
 0 & 0 & 0 & 0 & 0 & 0 & 0 & 4 \\*[1mm]
 0 & 0 & 0 & -4 & 0 & 0 & 0 & -4 s \\*[1mm]
 0 & 0 & 0 & 0 & 0 & 0 & 0 & -1
\end{array}
\right)\left(
\begin{array}{c}
 p_1(s) \\*[1mm]
 p_2(s) \\*[1mm]
 p_3(s) \\*[1mm]
 p_4(s) \\*[1mm]
 p_1'(s) \\*[1mm]
 p_2'(s) \\*[1mm]
 p_3'(s) \\*[1mm]
 p_4'(s) 
\end{array}
\right) =
\left(
\begin{array}{c}
 0 \\*[1mm]
 0 \\*[1mm]
 0 \\*[1mm]
 0 \\*[1mm]
 s \\*[1mm]
 0 \\*[1mm]
 \frac{1}{2} \\*[1mm]
 1 \\*[1mm]
 \frac{4 s}{3} \\*[1mm]
 0 \\*[1mm]
 \frac{1}{6} \\*[1mm]
 \frac{1}{6} \\*[1mm]
 -s \\*[1mm]
 -\frac{4 s^2}{3}  \\*[1mm]
 \frac{1}{2} \\*[1mm]
 -\frac{s}{2} \\*[1mm]
 0 \\*[1mm]
 \frac{s^2}{2} \\*[1mm]
 -1 \\*[1mm]
 -\frac{11 s}{6}\\*[1mm]
 \frac{s^2}{2} \\*[1mm]
 -\frac{1}{2} \\*[1mm]
 s \\*[1mm]
 -\frac{s^3}{6}-\frac{1}{2} \\*[1mm]
 s \\*[1mm]
 0 \\*[1mm]
 \frac{1}{2} \\*[1mm]
 -\frac{s^2}{2} \\*[1mm]
 \frac{1}{2} \\*[1mm]
 -\frac{s}{2} \\*[1mm]
 0 \\*[1mm]
 -\frac{1}{6} \\*[1mm]
 0
\end{array}
\right)
\]}
\rule{\textwidth}{0.4pt}
\end{table}%
which, as a linear system in $8$ unknown polynomials, is uniquely solved by 
\[
\left(
\begin{array}{cccccccc}
 7/12 & s/3 & 0 & 1/24 & 0 & 1/3 & 0 & 0 
\end{array}
\right)^T \in \Q[s]^{\,8}.
\]
Since the last four entries are the derivatives of the first four, this solution is consistent with the form of solution we are interested in. Generally, we first solve the linear system for 
\[
\big(p_1,\ldots,p_n,r_1,\ldots,r_n\big) \in  \Q[s]^{2n}
\] 
and then check for consistency $p_m'=r_m$, $m=1,\ldots,n$. If consistent, such a solution also satisfies \eqref{eq:linearFForm} by integrating its differentiated form: the constant of integration vanishes because both sides decay (rapidly) to zero as $s\to\infty$. In the example we have thus obtained
\[
u_{30}(s) = \frac{7F'(s)}{12F(s)} +\frac{s F''(s)}{3F(s)} + \frac{F^{(4)}(s)}{24F(s)}.
\]

So, two effects of integrability must happen for this recursive algorithm to work properly: 
\begin{itemize}\itemsep=3pt
\item the overdetermined $(2n^2+1)\times 2n$ linear system has actually a solution in $\Q[s]^{2n}$,
\item the solution is consistent (the last $n$ entries being the derivatives of the first $n$ ones).
\end{itemize}
Because of the algebraic independence of the solution $q,q'$ of Painlevé II over $\Q[s]$ (cf.~\cite[Thm.~21.1]{MR1960811}), the converse is also true: if there is a representation as a linear $F$-form at all, the algorithm succeeds by finding its unique coefficient polynomials.   

Based on a CAS implementation of the algorithm, we can report that the $u_{jk}$ are represented as linear $F$-forms of degree $j+k+1$ for $0\leq j+k\leq 50$,\footnote{A table of the resulting linear $F$-forms comes with the source files at \url{https://arxiv.org/abs/2301.02022}.} adding further evidence to the conjecture of Shinault and Tracy; a general proof, however, would require theoretical insight into the underlying integrability of \eqref{eq:TWextended}. If true, an induction shows that
\[
\deg_q u_{jk} = 4(j+k+1),\qquad \deg_q u_{jk}' = 4(j+k)+2.
\]

\subsection{A general conjecture} As a matter of fact, even certain {\em nonlinear} rational polynomials of the terms $u_{jk}$, such as those representing $\tilde F_j/F$ in \eqref{eq:F22tildeVanilla} and \eqref{eq:F3tildeU}, can be represented as {\em linear} $F$-forms. 

Namely, Thm.~\ref{thm:detexpan}, Sect.~\ref{sect:functionalform} and the perturbation theory of finite-dimensional determinants imply that  $\tilde F_j/F$ can be written as a rational linear combination of the minors of $(u_{jk})_{j,k=0}^\infty$,
that is, of determinants of the form
\begin{equation}\label{eq:conjdet}
\begin{vmatrix}
u_{j_1 k_1} & \cdots & u_{j_1 k_m} \\*[1mm]
\vdots & & \vdots \\*[1mm]
u_{j_m k_1} & \cdots & u_{j_m k_m}
\end{vmatrix}.
\end{equation}
We are thus led to the following conjecture (checked computationally up to order $n=50$):

\subsection*{Conjecture} Each minor of the form \eqref{eq:conjdet} can be represented as linear $F$-forms of order
\begin{equation}\label{eq:order}
n = j_1 +\cdots + j_m + k_1 +\cdots +  k_m + m.
\end{equation}
(Here, the case $m=1$ corresponds to the conjecture of Shinault and Tracy.) Aside from those terms which can be recast as linear combinations of minors with coefficients in $\Q[s]$, there are no other polynomial expressions of the $u_{jk}$ with coefficients in $\Q[s]$ that can be represented as linear $F$-forms.
\medskip

There are abundant examples of nonlinear rational polynomials of the terms $u_{jk}$ which cannot be represented as linear $F$-forms. For instance, 
as we have checked computationally, {\em none} of the terms (which are subterms of the minors shown below)
\[
u_{10}^2, \quad u_{00} u_{11}, \quad u_{10} u_{21},\quad  u_{11} u_{20}, \quad u_{00} u_{31}, \quad u_{10} u_{30}
\]
can be represented as linear $F$-forms of an order up to $n=50$.
 
Algorithmically, based on the already tabulated $u_{jk}$, the linear $F$-form of order $n$ for a given term $T$ such as \eqref{eq:conjdet} can be found, if existent, as follows: by expanding the equation
\[
T = p_1 \frac{F'}{F} + p_2 \frac{F''}{F} + \cdots + p_{n} \frac{F^{(n)}}{F}
\]
in $\Q[s][q,q']$ and comparing coefficients of the $q$-monomials, we get an overdetermined linear system of size $2(n^2-n+2)\times n$ for the 
coefficient polynomials $p_1,\ldots,p_n \in \Q[s]$. If there is a solution, we have found the linear $F$-form. If not, there is no such form of order $n$.

For instance, the nonlinear part in \eqref{eq:F22tildeVanilla} is
\[
u_{00}(s) u_{11}(s) - u_{10}(s)^2 
= \begin{vmatrix}
 u_{00}(s) & u_{01}(s) \\*[1mm]
 u_{10}(s) & u_{11}(s) 
 \end{vmatrix},
\]
which yields, for the order $n=4$ taken from \eqref{eq:order}, the $28\times 4$ linear system in $\Q[s]$ displayed in Table~\ref{tab:linSys_Detu01}.
\begin{table}[tbp]
\caption{The $28\times 4$ linear system for representing $u_{00}(s) u_{11}(s) - u_{10}(s)^2$ as in \eqref{eq:det00}.}
\label{tab:linSys_Detu01}
\vspace*{-3mm}
{\small
\[
\left(
\begin{array}{cccc}
 -1 & 0 & 0 & 2 \\*[1mm]
 0 & -1 & 0 & 0 \\*[1mm]
 0 & 0 & -1 & 0 \\*[1mm]
 0 & 0 & 0 & -1 \\*[1mm]
 0 & 0 & 2 & 0 \\*[1mm]
 0 & 0 & 0 & 8 \\*[1mm]
 s & 1 & 0 & 2 s \\*[1mm]
 0 & 2 s & 3 & 0 \\*[1mm]
 0 & 0 & 3 s & 6 \\*[1mm]
 0 & 0 & 0 & 4 s \\*[1mm]
 0 & 0 & 0 & -8 s \\*[1mm]
 1 & -s^2 & -3 s & 1 \\*[1mm]
 0 & 2 & -3 s^2 & -12 s \\*[1mm]
 0 & 0 & 3 & -6 s^2 \\*[1mm]
 0 & 0 & 0 & 4 \\*[1mm]
 0 & 0 & 0 & -8 \\*[1mm]
 0 & -2 s & s^3-3 & 6 s^2 \\*[1mm]
 0 & 0 & -6 s & 4 s^3-12 \\*[1mm]
 0 & 0 & 0 & -12 s \\*[1mm]
 0 & -1 & 3 s^2 & 12 s-s^4 \\*[1mm]
 0 & 0 & -3 & 12 s^2 \\*[1mm]
 0 & 0 & 0 & -6 \\*[1mm]
 0 & 0 & 3 s & 6-4 s^3 \\*[1mm]
 0 & 0 & 0 & 12 s \\*[1mm]
 0 & 0 & 1 & -6 s^2 \\*[1mm]
 0 & 0 & 0 & 4 \\*[1mm]
 0 & 0 & 0 & -4 s \\*[1mm]
 0 & 0 & 0 & -1 
\end{array}
\right)\left(
\begin{array}{c}
 p_1(s) \\*[1mm]
 p_2(s) \\*[1mm]
 p_3(s) \\*[1mm]
 p_4(s) 
 \end{array}
\right)
=
\left(
\begin{array}{c}
 0 \\*[1mm]
 \frac{s}{3} \\*[1mm]
 0 \\*[1mm]
 -\frac{1}{12} \\*[1mm]
 0 \\*[1mm]
 \frac{2}{3} \\*[1mm]
 0 \\*[1mm]
 -\frac{2 s^2}{3}  \\*[1mm]
 \frac{1}{2} \\*[1mm]
 \frac{s}{3} \\*[1mm]
 -\frac{2 s}{3}  \\*[1mm]
 \frac{s^3}{3}+\frac{1}{4} \\*[1mm]
 -\frac{5 s}{3}  \\*[1mm]
 -\frac{s^2}{2} \\*[1mm]
 \frac{1}{3} \\*[1mm]
 -\frac{2}{3} \\*[1mm]
 \frac{7 s^2}{6} \\*[1mm]
 \frac{s^3}{3}-1 \\*[1mm]
 -s \\*[1mm]
 \frac{4 s}{3}-\frac{s^4}{12} \\*[1mm]
 s^2 \\*[1mm]
 -\frac{1}{2} \\*[1mm]
 \frac{1}{2}-\frac{s^3}{3} \\*[1mm]
 s \\*[1mm]
 -\frac{s^2}{2} \\*[1mm]
 \frac{1}{3} \\*[1mm]
 -\frac{s}{3} \\*[1mm]
 -\frac{1}{12} 
\end{array}
\right).
\]}
\rule{\textwidth}{0.4pt}
\end{table}%
Its unique solution is 
\[
\left(
\begin{array}{cccc}
 \frac{1}{6} & -\frac{s}{3} & 0 & \frac{1}{12} \\
\end{array}
\right)^T \in \Q[s]^4
\]
so that we obtain the linear $F$-form
\begin{equation}\label{eq:det00}
 \begin{vmatrix}
 u_{00}(s) & u_{01}(s) \\*[1mm]
 u_{10}(s) & u_{11}(s)
 \end{vmatrix} =  \frac{F'(s)}{6F(s)} -\frac{sF''(s)}{3F(s)} + \frac{F^{(4)}(s)}{12F(s)}.
\end{equation}
Likewise we see that the integrability displayed in the evaluation of \eqref{eq:F3tilde} is based on the fact that the two minors which appear as subexpressions can be represented as linear $F$-forms; indeed, in both cases \eqref{eq:order} yields the order $n=6$ and by solving the corresponding $64\times 6$ linear systems  we get\footnote{Further examples (with minors of size $m=3,4,5$ and of order up to $n=25$) can be found in a Mathematica notebook that comes with the source files at \url{https://arxiv.org/abs/2301.02022}.}
\begin{subequations}\label{eq:det22}
\begin{align}
\begin{vmatrix}
 u_{10}(s) & u_{11}(s) \\*[2mm]
 u_{20}(s) & u_{21}(s) 
 \end{vmatrix}  &= -\frac{sF'(s)}{18F(s)}  +  \frac{s^2F''(s)}{9F(s)} -\frac{F'''(s)}{24F(s)}  -\frac{sF^{(4)}(s)}{18F(s)}  + \frac{F^{(6)}(s)}{144F(s)},\\*[3mm]
\begin{vmatrix}
 u_{00}(s) & u_{01}(s) \\*[2mm]
 u_{30}(s) & u_{31}(s) 
 \end{vmatrix}  &= \frac{sF'(s)}{10F(s)} -\frac{s^2F''(s)}{5F(s)} -\frac{3F'''(s)}{40F(s)} + \frac{F^{(6)}(s)}{80F(s)}.
\end{align}
\end{subequations}
\begin{remark}
Other applications of the technique discussed here can be found in \cite[§3.3]{2306.03798}.
\end{remark}

\bibliographystyle{spmpsci}
\bibliography{paper}

\end{document}